\documentclass[11pt,reqno,tbtags]{amsart}
\pdfoutput=1
\usepackage[]{amsmath,amssymb,amsfonts,latexsym,amsthm,enumerate}
\usepackage{amssymb,bbm,mathrsfs}
\usepackage{enumerate,graphicx,paralist,tikz}
\usepackage[numeric,initials,nobysame]{amsrefs}

\numberwithin{equation}{section}
\usepackage[marginratio=1:1,height=584pt,width=488pt,tmargin=117pt]{geometry}
\usepackage[margin=25pt,font=small,justification=centering
]{caption}

\newtheorem{maintheorem}{Theorem}

\newtheorem{theorem}{Theorem}[section]
\newtheorem*{theorem*}{Theorem}
\newtheorem{lemma}[theorem]{Lemma}
\newtheorem{claim}[theorem]{Claim}
\newtheorem{proposition}[theorem]{Proposition}

\newtheorem{corollary}[theorem]{Corollary}

\theoremstyle{definition}{
\newtheorem{example}[theorem]{Example}
\newtheorem{definition}[theorem]{Definition}
\newtheorem*{definition*}{Definition}

\newtheorem*{example*}{Example}
\newtheorem{remark}[theorem]{Remark}
\newtheorem*{remark*}{Remark}

}

\newcommand{\R}{\mathbb R}

\newcommand{\Z}{\mathbb Z}

\newcommand{\deq}{\stackrel{\scriptscriptstyle\triangle}{=}}

\newcommand{\E}{\mathbb{E}}
\renewcommand{\P}{\mathbb{P}}
\DeclareMathOperator{\var}{Var}
\DeclareMathOperator{\Cov}{Cov}

\newcommand{\gap}{\text{\tt{gap}}}

\newcommand{\tmix}{t_\textsc{mix}}

\newcommand{\tv}{{\textsc{tv}}}
\newcommand{\Po}{\operatorname{Po}}

\newcommand{\one}{\mathbbm{1}}
\newcommand{\red}{\textsc{Red}}
\newcommand{\blue}{\textsc{Blue}}
\newcommand{\green}{\textsc{Green}}
\newcommand{\conn}{\leftrightsquigarrow}

\renewcommand{\epsilon}{\varepsilon}
\renewcommand{\phi}{\varphi}

\newcommand{\cG}{\mathcal{G}}
\newcommand{\cS}{\mathcal{S}}
\newcommand{\cN}{\mathcal{N}}
\newcommand{\cA}{\mathcal{A}}
\newcommand{\cB}{\mathcal{B}}
\newcommand{\cC}{\mathcal{C}}

\newcommand{\cE}{\mathcal{E}}

\newcommand{\cI}{\mathcal{I}}

\newcommand{\cF}{\mathcal{F}}
\newcommand{\cR}{\mathcal{R}}
\newcommand{\cU}{\mathcal{U}}
\newcommand{\cW}{\mathcal{W}}
\newcommand{\cV}{\mathcal{V}}

\newcommand{\cX}{\mathcal{X}}

\DeclareMathOperator{\sign}{sign}

\newcommand{\sm}{{\mathfrak{m}}}
\newcommand{\anim}{{\mathfrak{W}}}

\newcommand{\oned}{1\textsc{d} }

\newcommand{\uni}{{(\textsc{u})}}
\newcommand{\sX}{\mathscr{X}}

\newcommand{\med}{\operatorname{med}}
\newcommand{\sH}{\mathscr{H}}
\newcommand{\fupd}{\mathscr{F}_\textsc{upd}}
\newcommand{\tfupd}{\tilde{\mathscr{F}}_\textsc{upd}}
\newcommand{\fsup}{\mathscr{F}_\textsc{sup}}
\newcommand{\tfsup}{\tilde{\mathscr{F}}_\textsc{sup}}

\newcommand{\tcut}{t_\sm}
\newcommand{\scut}{s_\star}
\newcommand{\tpluss}{t_\star}

\date{}

\begin{document}
\title[Information percolation for the Ising model]{Information percolation and cutoff \\ for the stochastic Ising model}

\author{Eyal Lubetzky}
\address{Eyal Lubetzky\hfill\break
Courant Institute 
\\ New York University\\
251 Mercer Street\\ New York, NY 10012, USA.}
\email{eyal@courant.nyu.edu}
\urladdr{}

\author{Allan Sly}
\address{Allan Sly\hfill\break
Department of Statistics\\
UC Berkeley\\
Berkeley, CA 94720, USA.}
\email{sly@stat.berkeley.edu}
\urladdr{}

\begin{abstract}
We introduce a new framework for analyzing Glauber dynamics for the Ising model.
The traditional approach for obtaining sharp mixing results
has been to appeal to estimates on spatial properties of the stationary measure
from within a multi-scale analysis of the dynamics.
Here we propose to study these simultaneously by examining ``information percolation''
clusters in the space-time slab.

Using this framework, we obtain new results for the Ising model on $(\Z/n\Z)^d$ throughout the high temperature regime:
total-variation mixing exhibits cutoff with an $O(1)$-window around the time at which the magnetization is the square-root of the volume.
(Previously, cutoff in the full high temperature regime was only known in dimensions $d\leq 2$, and only with an $O(\log\log n)$-window.)

Furthermore, the new framework opens the door to understanding the effect of the initial state on the mixing time.
We demonstrate this on the \oned Ising model, showing that starting from the uniform (``disordered'') initial distribution asymptotically
halves the mixing time, whereas almost every deterministic starting state is asymptotically as bad as starting from the (``ordered'') all-plus state.
\end{abstract}

\maketitle

\vspace{-0.5cm}

\section{Introduction}\label{sec:intro}

Glauber dynamics is one of the most common methods of sampling from the high-temperature Ising model 
(notable flavors are Metropolis-Hastings or Heat-bath dynamics),
 and at the same time provides a natural model for its evolution from any given initial configuration.

We introduce a new framework for analyzing the Glauber dynamics via ``information percolation'' clusters in the space-time slab, a unified approach to studying spin-spin correlations in $\Z^d$ over time
(depicted in Fig.~\ref{fig:clusters-t25}--\ref{fig:clusters-all} and described in~\S\ref{sec:methods}).
Using this framework, we make progress on the following.
\begin{asparaenum}
  [~~(i)]
\smallskip

  \item \emph{High-temperature vs.\ infinite-temperature:}
it is believed that when the inverse-temperature $\beta$ is below the critical $\beta_c$, the dynamics behaves qualitatively as if $\beta=0$ (the spins evolve independently).
 In the latter case, the continuous-time dynamics exhibits the \emph{cutoff phenomenon}\footnote{sharp transition in the
$L^1$-distance of a finite Markov chain from equilibrium, dropping quickly
from near 1 to near 0.} with an $O(1)$-window as shown by Aldous~\cite{Aldous}
  and refined in~\cites{DiSh2,DGM}; thus, the above paradigm suggests cutoff at any $\beta <\beta_c$.
 Indeed, it was conjectured by Peres in 2004 (see~\cite{LLP}*{Conjecture 1} and~\cite{LPW}*{Question 8, p316}) that cutoff occurs whenever there is $O(\log n)$-mixing\footnote
 {
 this pertains $\beta<\beta_c$, since at $\beta=\beta_c$ the mixing time for the Ising model on $(\Z/n\Z)^d$ is at least polynomial in $n$ by results of Aizenman and Holley~\cites{AH,Holley} (see~\cite{Holley}*{Theorem 3.3}) whereas for $\beta>\beta_c$ it is exponential in $n^{d-1}$ (cf.~\cite{Martinelli97}).
 }. Moreover, one expects the cutoff window to be $O(1)$.

Best-known results on $\Z^d$: cutoff for the Ising model in the full high-temperature regime $\beta <\beta_c$ was only confirmed in dimensions $d\leq 2$ (\cite{LS1}), and only with a bound of $O(\log\log n)$ on the cutoff window.

\smallskip

  \item \emph{Warm start (random, disordered) vs.\ cold start (ordered)}:
   within the extensive physics literature offering numerical experiments for spin systems, it is common to find
   Monte Carlo simulations at high temperature started at a random (warm) initial state where spins are i.i.d.\ (``disordered''); cf.~\cites{KLW,Sokal}.
   A natural question is whether this accelerates the
   mixing for the Ising model, and if so by how much.

Best-known results on $\Z^d$: none to our knowledge --- sharp upper bounds on total-variation mixing for the Ising model were only applicable to worst-case starting states (usually via coupling techniques).
\end{asparaenum}

\begin{figure}[t]
\vspace{-0.35cm}
\includegraphics[trim= 20cm 0mm 20cm 0mm, clip, width=1.\textwidth]{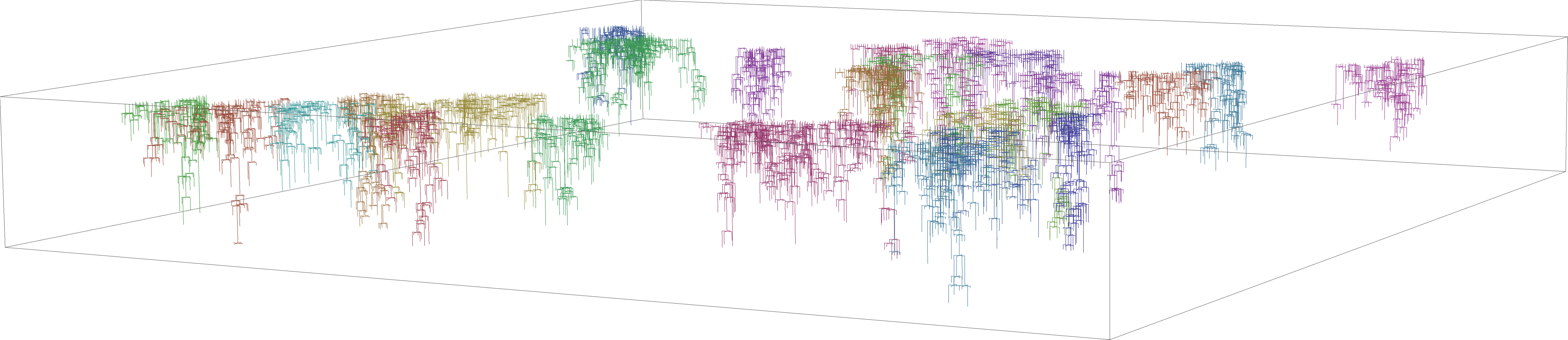}
\vspace{-0.85cm}
\caption{Information percolation clusters for the 2\textsc{d} stochastic Ising model:\\
 showing the largest 25 clusters on a $\{1,\ldots,200\}^2 \times [0,20]$ space-time slab.}
\label{fig:clusters-t25}
\vspace{-0.35cm}
\end{figure}

The cutoff phenomenon plays a role also in the second question above: indeed, whenever there is cutoff, one can compare the effect of different initial states $x_0$ on the asymptotics of the corresponding mixing time $\tmix^{(x_0)}(\epsilon)$ independently of $\epsilon$, the distance within which we wish to approach equilibrium. (For more on the cutoff phenomenon, discovered in the early 80's by Aldous and Diaconis, see~\cites{AD,Diaconis}.)

\subsection{Results}
Our first main result confirms the above conjecture by Peres that Glauber dynamics for the Ising model on $\Lambda=(\Z/n\Z)^d$, in \emph{any} dimension $d$, exhibits cutoff in the full high temperature regime $\beta<\beta_c$. Moreover, we establish cutoff within an $O(1)$-window (as conjectured) around the point
\begin{equation}\label{eq-t*-def}
\tcut = \inf\big\{ t>0 \;:\; \sm_t(v) \leq 1/\sqrt{|\Lambda|} \big\}\,,
\end{equation}
where $\sm_t(v)$ is the magnetization at the vertex $v\in\Lambda$ at time $t>0$, i.e.,
\begin{equation}\label{eq-mag-v}
\sm_t(v) = \E X_t^+(v)
\end{equation}
with $(X_t^+)$ being the dynamics started from the all-plus starting state (by translational invariance we may write $\sm_t$ for brevity).
Intuitively, at time $\tcut$ the expected value of $\sum_v X_t^+(v)$ becomes $O(\sqrt{|\Lambda|})$, within the normal deviations of the Ising distribution, and expect mixing to occur.
For instance, in the special case $\beta=0$ we have $\sm_t = e^{-t}$ and so $\tcut = \frac12\log |\Lambda|$, the known cutoff location from~\cites{AD,DiSh2,DGM}.

\begin{maintheorem}\label{mainthm-Zd}
Let $d\geq 1$ and let $\beta_c$ be the critical inverse-temperature for the Ising model on $\Z^d$. Consider continuous-time Glauber dynamics for the Ising model on the torus $(\Z/n\Z)^d$. Then for any inverse-temperature $\beta < \beta_c$ the dynamics exhibits cutoff at
$\tcut$ as given in~\eqref{eq-t*-def} with an $O(1)$-window. Moreover, there exists $C=C_{\beta,d}>0$ so that for any fixed $0<\epsilon<1$ and large $n$,
\begin{align*}
 \tmix(1-\epsilon)&\geq \tcut - C \log(1/\epsilon)  \,,\\
   \tmix(\epsilon) &\leq \tcut + C \log(1/\epsilon)\,.
 \end{align*}
\end{maintheorem}

This improves on~\cite{LS1} in two ways: \begin{inparaenum}[(a)]
\item A prerequisite for the previous method of proving cutoff for the Ising model on lattices (and all of its extensions) was that the
stationary measure would satisfy the decay-of-correlation condition known as \emph{strong spatial
mixing}, valid in the full high temperature regime for $d\leq 2$. However, for $d\geq 3$ it is known to hold only for $\beta$ small enough\footnote{At low temperatures on $\Z^d$ (see the discussion above Theorem~\ref{mainthm-Zd-external-field}) there might not be \emph{strong} spatial mixing despite an exponential decay-of-correlations (\emph{weak} spatial mixing); however, one expects to have strong spatial mixing for all $\beta<\beta_c$.}; Theorem~\ref{mainthm-Zd} removes this limitation and covers all $\beta<\beta_c$ (see also Theorem~\ref{mainthm-Zd-external-field} below).
\item A main ingredient in the previous proofs was a reduction of $L^1$-mixing to very-fine $L^2$-mixing of sub-cubes of poly-logarithmic size, which was achieved via log-Sobolev inequalities in time $O(\log\log n)$. This led to a sub-optimal $O(\log\log n)$ bound on the cutoff window, which we now improve to the conjectured $O(1)$-window.
\end{inparaenum}

\begin{figure}[t]
\vspace{-0.25cm}
\includegraphics[width=.4\textwidth]{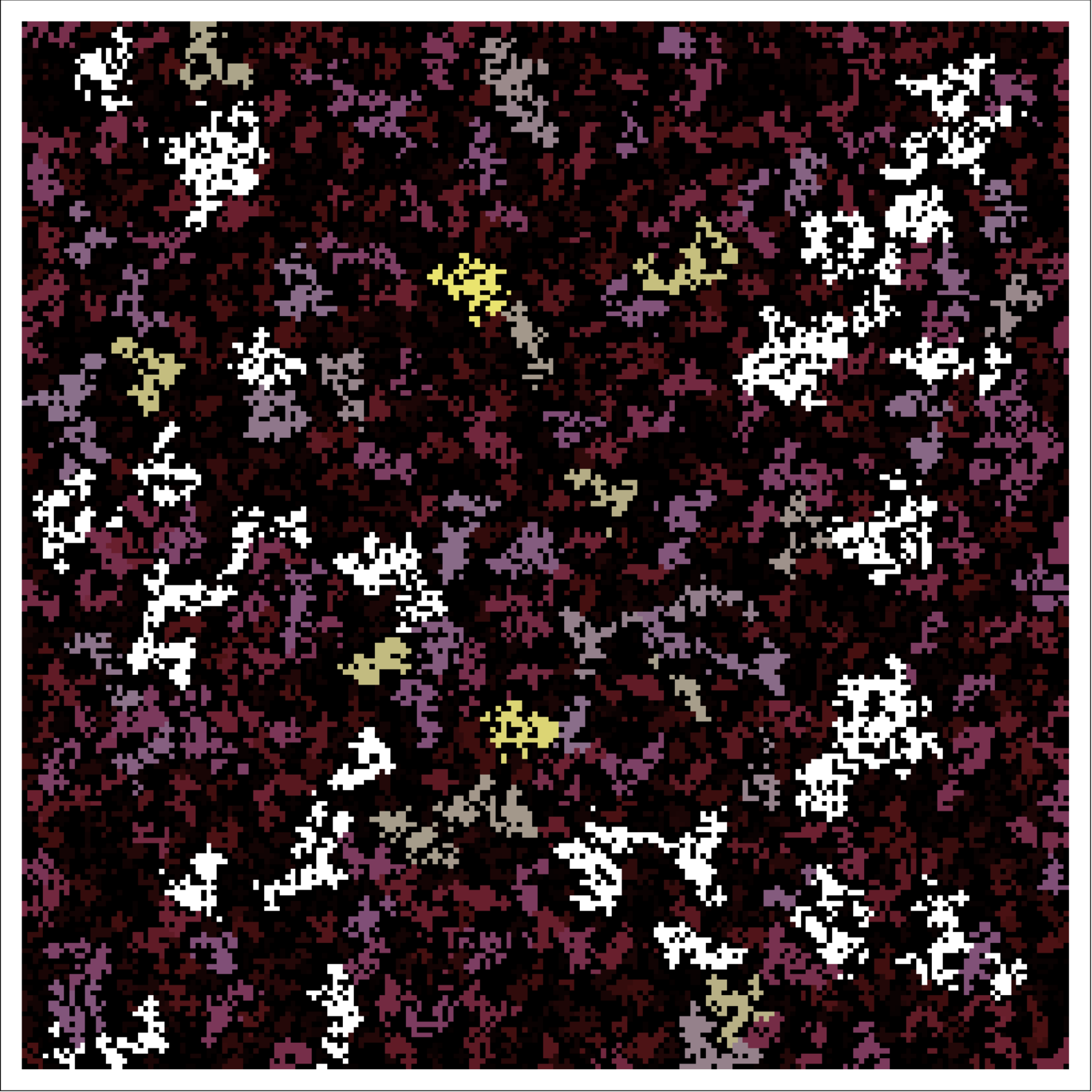}
\vspace{-0.15cm}
\caption{Top view of information percolation clusters for the 2\textsc{d} stochastic Ising model: \\
 sites of a $200\times 200$ square color-coded by their cluster size (increasing from red to white).}
\label{fig:clusters-all}
\vspace{-0.25cm}
\end{figure}

The lower bound on the $\tmix(1-\epsilon)$ in Theorem~\ref{mainthm-Zd} is realized from the all-plus starting configuration; hence, for any $d\geq 1$ this is (as expected) the worst-case starting state up to an additive $O(1)$-term:
\begin{align}\label{eq-tmix-all-plus}
 \tmix^{(+)}(\epsilon)& = \tcut + O\big(\log(1/\epsilon)\big)  \quad\mbox{  for any $\beta <\beta_c$ and $0<\epsilon<1$}\,.
 \end{align}
This brings us to the aforementioned question of understanding the mixing from specific initial states.
Here the new methods can be used to give sharp bounds, and in particular to compare the warm start using the uniform (i.i.d.)
distribution to various deterministic initial states. We next demonstrate this on the \oned Ising model (we treat higher dimensions, and more generally any bounded-degree geometry, in the companion paper~\cite{LS4}), where, informally, we show that
\begin{itemize}
  \item The uniform starting distribution is asymptotically \emph{twice faster} than the worst-case all-plus;
  \item Almost all deterministic initial states are asymptotically \emph{as bad} as the worst-case all-plus.
\end{itemize}
Formally, if $\mu_t^{(x_0)}$ is the distribution of the dynamics at time $t$ started from $x_0$ then $\tmix^{(x_0)}(\epsilon)$ is the minimal $t$ for which $\mu_t^{(x_0)}$ is within distance $\epsilon$ from equilibrium, and $\tmix^\uni(\epsilon)$ is the analogue for the average $2^{-|\Lambda|}\sum_{x_0}\mu_t^{(x_0)}$ (i.e., the \emph{annealed} version, as opposed to the \emph{quenched} $\tmix^{(x_0)}$ for a uniform $x_0$).

\begin{maintheorem}
  \label{mainthm-Z-ann-que}
Fix any $\beta>0$ and $0<\epsilon<1$, and consider continuous-time Glauber dynamics for the Ising model on $(\Z/nZ)$.
Letting $\tcut = \frac1{2\theta}\log n$ with $\theta=1-\tanh(2\beta)$, the following hold:
\begin{enumerate}[\it1.]
  \item\label{it-ann} (Annealed) Starting from a uniform initial distribution:  $\tmix^\uni(\epsilon) \sim \frac12 \tcut$.
  \item\label{it-que} (Quenched) Starting from a deterministic initial state: $\tmix^{(x_0)}(\epsilon) \sim \tmix^{(+)}(\epsilon) \sim \tcut$
  for almost every $x_0$.
\end{enumerate}
\end{maintheorem}

Unlike the proof of Theorem~\ref{mainthm-Zd}, which coupled the distributions started at worst-case states, in order to analyze the uniform initial state one is forced to compare the distribution at time $t$ directly to the stationary measure. This delicate step is achieved via the \emph{Coupling From The Past} method~\cite{PW}.

\begin{remark*}
The bound on $\tmix^{(x_0)}(\epsilon)$ applies not only to a typical starting state $X_0$, but to any \emph{deterministic} $X_0$ which satisfies that $1/\sum_v (\E X_{\tcut}(v))^2$ is sub-polynomial in $n$ --- e.g., $O((\log n)^{100})$ --- a
 condition that can be expressed via $X_0(Y_{\tcut})$ where $Y_t$ is continuous-time random walk on $\Z_n$; see Proposition~\ref{prop:quenched}.
\end{remark*}

As noted earlier, the new framework relaxes the \emph{strong spatial mixing} hypothesis from previous works into \emph{weak spatial mixing} (i.e., exponential decay-of-correlation, valid for all $\beta<\beta_c$ in any dimension).
This has consequences also for \emph{low temperatures}:
there it is
strongly believed that in dimension $d\geq 3$
 (see~\cite{Martinelli97}*{\S5} and~\cite{CM})
 under certain non-zero external magnetic field (some fixed $h\neq 0$ for all sites) there would be weak but not strong spatial mixing.
 Using the periodic boundary conditions to preclude boundary effects, our arguments remain valid also in this situation, and again we obtain cutoff:

\begin{maintheorem}[low temperature with external field]\label{mainthm-Zd-external-field}
The conclusion of Theorem~\ref{mainthm-Zd} holds in $(\Z/n\Z)^d$ for any large enough fixed inverse-temperature $\beta$ in the presence of a non-zero external magnetic field.
\end{maintheorem}

\medskip

We now discuss extensions of the framework towards showing  \emph{universality of cutoff}, whereby the cutoff phenomenon --- believed to be widespread, despite having been rigorously shown only in relatively few cases --- is not specific to the underlying geometry of the spin system, but instead occurs always at high temperatures (following the intuition behind the aforementioned conjecture of Peres from 2004).
Specializing this general principle to the Ising model, one expects the following to hold:
\begin{quote}\emph{
On any locally finite geometry the Ising model should exhibit cutoff at high temperature
(i.e., cutoff always occurs for $\beta < c_d$ where $c_d$ depends only on the maximum degree $d$).}\end{quote}
The prior technology for establishing cutoff for the Ising model fell well short of proving such a result.
Indeed, the approach in~\cite{LS1}, as well as its generalization in~\cite{LS3}, contained two major provisos:
\begin{compactenum}
  [\!\!\!(i)]
  \item heavy reliance on log-Sobolev constants to provide sharp $L^2$-bounds on local mixing (see \cites{DS1,DS2,DS,SaloffCoste});
  the required log-Sobolev bounds can in general be highly nontrivial to verify
 (see~\cites{HoSt1,MO,MO2,MOS,Martinelli97,SZ1,SZ3}).
\item an assumption on the geometry that the growth rate of balls (neighborhoods) is sub-exponential;
while satisfied on lattices (linear growth rate), this rules out trees, random graphs, expanders, etc.
\end{compactenum}
Demonstrating these limitations is the fact that the required log-Sobolev inequalities for the Ising model were established essentially only on lattices and regular trees,
whereas on the latter (say, a binary tree) it was unknown whether the Ising model exhibits cutoff at any small $\beta>0$, due to the second proviso.

In contrast with this, the above mentioned paradigm instead says that, at high enough temperatures, cutoff should occur without necessitating log-Sobolev inequalities, geometric expansion properties, etc.
Using the new framework of information percolation we can now obtain such a result.
Define the non-transitive analogue of the cutoff-location $\tcut$ from~\eqref{eq-t*-def} to be
\begin{equation}\label{eq-t*-def-non-trans}
\tcut = \inf\big\{ t>0 \;:\; \mbox{$\sum_v \sm_t(v)^2 \leq 1$} \big\}\,,
\end{equation}
with $\sm_t(v) = \E X_t^+(v)$ as in~\eqref{eq-mag-v}. 
The proof of the following theorem --- which, apart from the necessary adaptation of the framework to deal with a non-transitive geometry, required several novel ingredients to obtain the correct dependence of $\beta$ on the maximal degree --- appears in a companion paper~\cite{LS4}.
\begin{maintheorem}\label{mainthm-gen}
There exists an absolute constant $\kappa>0$ so that the following holds. Let $G$ be a graph on $n$ vertices with maximum degree at most $d$. For any fixed $0<\epsilon<1$ and large enough $n$, the continuous-time Glauber dynamics for the Ising model on $G$ with inverse-temperature $0\leq\beta<\kappa/ d$ has
\begin{align*}
 \tmix(1-\epsilon)&\geq \tcut - C \log(1/\epsilon)  \,,\\
   \tmix(\epsilon) &\leq \tcut + C \log(1/\epsilon)\,.
 \end{align*}
In particular, on any sequence of such graphs the dynamics has cutoff with an $O(1)$-window around $\tcut$.
\end{maintheorem}

The companion paper further extends Theorem~\ref{mainthm-Z-ann-que} to any bounded-degree graph at high temperature: the mixing time is at least $(1-\epsilon_\beta)\tcut$ from almost every deterministic initial state $x_0$, yet from a uniform initial distribution it is at most $(\frac12+\epsilon_\beta)\tcut$, where $\epsilon_\beta$ can be made arbitrarily small for $\beta$ small enough.

In summary, on any locally-finite geometry (following  Theorems~\ref{mainthm-Zd}--\ref{mainthm-Z-ann-que} for $\Z^d$) one roughly has that
 \begin{inparaenum}
   \item the time needed to couple the dynamics from the extreme initial states, $X_t^+$ and $X_t^-$, via the \emph{monotone coupling} (a standard upper bound on the mixing time) overestimates $\tmix$ by a factor of 2;
   \item the worst-case mixing time $\tmix$, which is asymptotically the same as when starting from almost every deterministic state,
       is \emph{another} factor of 2 worse compared to starting from the uniform distribution.
 \end{inparaenum}

\vspace{-0.2cm}
\subsection{Methods: red, green and blue information percolation clusters}\label{sec:methods}
The traditional approach for obtaining sharp mixing results for the Ising model
has been two-fold: one would first derive certain properties of the stationary Ising measure (ranging from as fundamental as strong spatial mixing to as proprietary as interface fluctuations under specific boundary conditions); these static properties would then drive a dynamical multi-scaled analysis (e.g., recursion via  block-dynamics/censoring); see~\cite{Martinelli97}.

\begin{figure}[t]
\vspace{-0.2cm}
\includegraphics[trim= 0cm 0cm 0cm 2mm, clip, width=.6\textwidth]{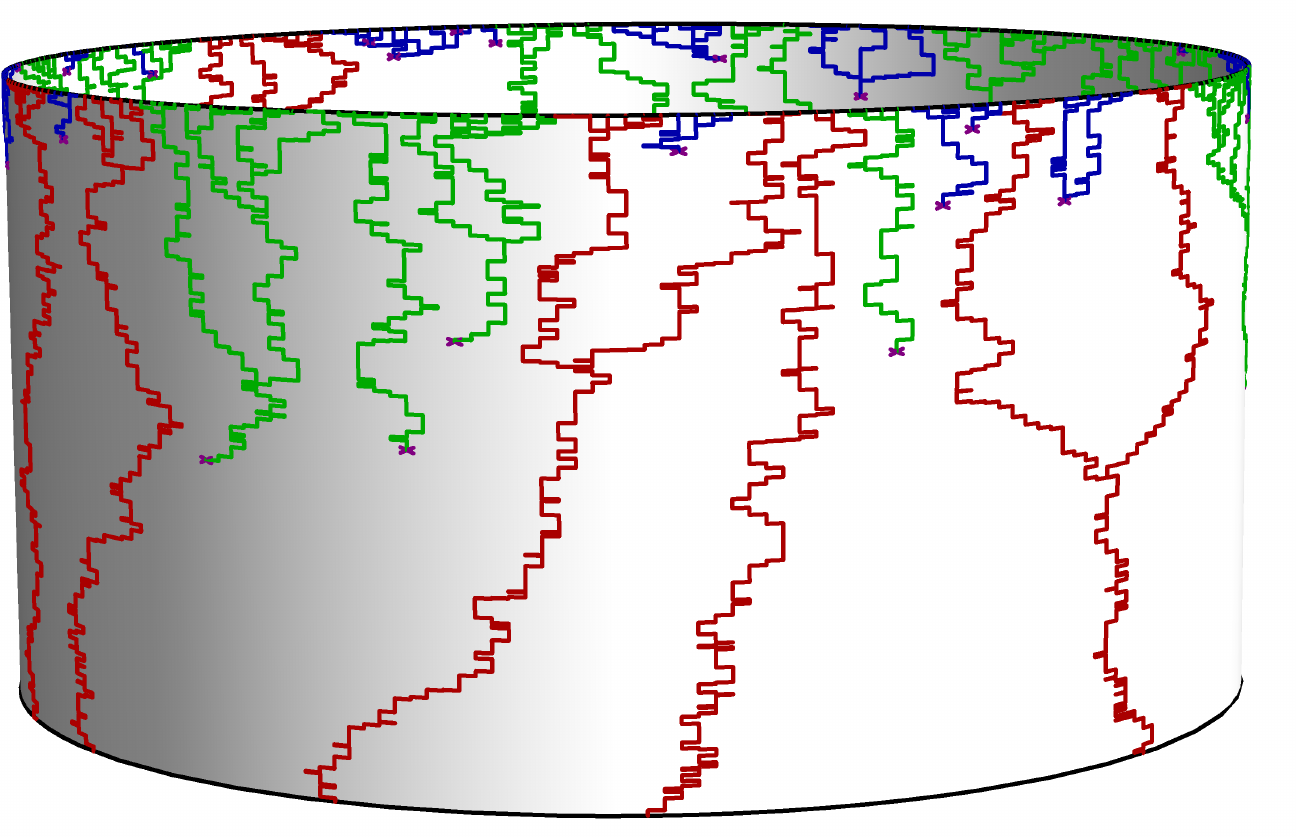}
\vspace{-0.5cm}
\caption{Information percolation clusters in Glauber dynamics for the 1\textsc{d} Ising model:\\
 red (reaching time zero), blue (dying out quickly) and green clusters on $n=256$ sites.}
\label{fig:clusters-1d}
\vspace{-0.35cm}
\end{figure}

We propose to analyze the spatial and temporal aspects of the Glauber dynamics simultaneously by tracking the update process for the Ising model on $(\Z/n\Z)^d$ in the $(d+1)$-dimensional space-time slab. Following is an outline of the approach for heat-bath dynamics\footnote{\vspace{-0.5cm}A single-site heat-bath update replaces a spin by a sample from the Ising measure conditioned on all other spins.}; formal definitions of the framework (which is valid for a class of Glauber dynamics that also includes, e.g., Metropolis) will be given in~\S\ref{sec:framework}.

As a first step, we wish to formulate the dynamics $(X_t)$ so that the update process, viewed backward in time, would behave as subcritical  percolation in $(\Z/n\Z)^d\times\R_+$;
crucially, establishing this subcritical behavior will build on the classical fact that the site magnetization $\sm_t$ (defined in~\eqref{eq-mag-v}), decays exponentially fast to 0 (the proof of which uses the monotonicity of the Ising model; see Lemma~\ref{l:BasicEstimates}).
Recall that each site of $(\Z/n\Z)^d$ is updated via a Poisson point process, whereby every update utilizes an independent unit variable to dictate the new spin, and the probability of both plus and minus is bounded away from 0 for any fixed $\beta>0$ even when all neighbors have the opposing spin.
Hence, we can say that with probability $\theta >0$ bounded away from 0 (explicitly given in~\eqref{eq-def-Xi}), the site is updated to a $\pm1$ fair coin flip independently of the spins at its neighbors, to be referred to as an \emph{oblivious update}.

\vspace{-0.2cm}
\subsubsection*{Clusters definition}
For simplicity, we first give a basic definition that will be useful only for small $\beta$.
Going backward in time from a given site $v$ at time $t$, we reveal the \emph{update history} affecting $X_t(v)$: in case of an oblivious update we ``kill'' the branch, and otherwise we split it into its neighbors, continuing until all sites die out or reach time 0  (see Figure~\ref{fig:clusters-t25}). The final cluster then allows one to recover $X_{t}(v)$ given the unit variables for the updates and the intersections of the cluster with the initial state $x_0$.

Note that the dependencies in the Ising measure show up in this procedure when update histories of different sites at time $t$ merge into a single cluster, turning the spins at time $t$ into a complicated function of the update variables and the initial state.  Of course, since the probability of an oblivious update $\theta$ goes to $1$ as $\beta\to 0$, for a small enough $\beta$ the aforementioned branching process is indeed subcritical, and so these clusters should have an exponential tail (see Figure~\ref{fig:clusters-all}). For $\beta$ close to the critical point in lattices, this is no longer the case, and one needs to refine the definition of an information percolation cluster --- roughly, it is the subset of the update history that the designated spin truly depends on (e.g., in the original procedure above, an update can cause the function determining $X_t(v)$ to become independent of another site in the cluster, whence the latter is removed without being directly updated).

The motivation behind studying these clusters is the following. Picture a typical cluster as a single strand, linking between ``sausages'' of branches that split and quickly dye out. If this strand dies before reaching time 0 then the spin atop would be uniform, and otherwise, starting e.g.\ from all-plus, that spin would be plus. Therefore, our definition of the cutoff time $\tcut$ has that about $\sqrt{|\Lambda|}$ of the sites reach time 0; in this way, most sites are independent of the initial state, and so $X_t$ would be well mixed.
Further seen now is the role of the initial state $x_0$, opening the door to non-worst-case analysis: one can analyze the distribution of the spins atop a cluster in terms of its intersection with $x_0$ at time $0$.

\subsubsection*{Red, green and blue clusters}
To quantify the above, we classify the clusters into three types: informally,
\begin{compactitem}
  \item a cluster is \blue\ if it dies out very quickly both in space and in time;
  \item a cluster is \red\ if the initial state affects the spins atop;
  \item a cluster is \green\ in all other situations.
\end{compactitem}
(See \S\ref{sec:framework} for formal definitions, and Figure~\ref{fig:clusters-1d} for an illustration of these for the Ising model on $\Z/n\Z$.)
Once we condition on the green clusters (to be thought of as having a negligible effect on mixing), what remains is a competition between red clusters --- embodying the dependence on the initial state $x_0$ --- and blue ones, the projection on which is just a product measure (independent of $x_0$).
Then, one wants to establish that red clusters are uncommon and ``lost within a sea of blue clusters''. This is achieved via a simple yet insightful lemma of Miller and Peres~\cite{MP}, bounding the total-variation distance in terms of a certain exponential moment; in our case, an exponential of the intersection of the set of vertices in \red\ clusters between two i.i.d.\ instances of the dynamics. Our main task --- naturally becoming increasingly more delicate as $\beta$ approaches $\beta_c$ --- will be to bound this exponential moment, by showing that each red set behaves essentially as a uniformly chosen subset of size $O(e^{-c s}\sqrt{|\Lambda|})$ at time $\tcut+s$; thus, the exponential moment will approach 1 as $s\to\infty$, implying mixing.

\subsubsection*{Flavors of the framework}
Adaptations of the general framework above can be used in different settings:
\begin{compactitem}[\noindent$\bullet$]
   \item To tackle arbitrary graphs at high enough temperatures (Theorem~\ref{mainthm-gen}), a blue cluster is one that dies out before reaching the bottom (time 0) and has a singleton spin at the top (the target time $t$), and a red cluster is one where the spins at the top have a nontrivial dependence on the initial state $x_0$.

   \item    For lattices at any $\beta<\beta_c$, the branching processes encountered are not sufficiently subcritical, and one needs to boost them via a phase in which (roughly) some of the oblivious updates are deferred, only to be sprinkled at the end of the analysis. This entails a more complicated definition of blue clusters, referring to whether history dies out quickly enough from the end of that special phase, whereas red clusters remain defined as ones where the top spins are affected by the initial state $x_0$.

   \item For random initial states (Theorem~\ref{mainthm-Z-ann-que}) we define a red cluster as one in which the intersection with $x_0$ is of size at least 2 and coalesces to a single point \emph{before} time 0 under \emph{Coupling From The Past}.
       The fact that pairs of sites surviving to time 0 are now the dominant term (as opposed to singletons) explains the factor of 2 between the annealed/worst-case settings (cf.\ the two parts of Theorem~\ref{mainthm-Z-ann-que}).
 \end{compactitem}

\subsection{Organization}
The rest of this paper is organized as follows.
In~\S\ref{sec:framework} we give the formal definitions of the above described framework, while \S\ref{sec:lattice-framework} contains the modification of the general framework tailored to lattices up to the critical point, including three lemmas analyzing the information percolation clusters. In~\S\ref{sec:inf-perc} we prove the cutoff results in Theorems~\ref{mainthm-Zd} and~\ref{mainthm-Zd-external-field} modulo these technical lemmas, which are subsequently proved in~\S\ref{sec:cluster-analysis}. The final section, \S\ref{sec:ann-que}, is devoted to the analysis of non-worst-case initial states (random vs.\ deterministic, annealed vs.\ quenched) and the proof of Theorem~\ref{mainthm-Z-ann-que}.

\section{Framework of information percolation}\label{sec:framework}

\subsection{Preliminaries}\label{sec:prelim}
In what follows we set up standard notation for analyzing the mixing of Glauber dynamics for the Ising model; see~\cite{LS1} and its references for additional information and background.

 \subsubsection*{Mixing time and the cutoff phenomenon}\label{sec:prelim-cutoff}
The total-variation distance between two probability measures $\nu_1,\nu_2$
on a finite space $\Omega$ --- one of the most important gauges in MCMC theory for measuring the convergence of a Markov chain to stationarity ---
is defined as
\[
\|\nu_1-\nu_2\|_\tv = \max_{A\subset \Omega} |\nu_1(A)-\nu_2(A)| = \tfrac12\sum_{\sigma\in\Omega} |\nu_1(\sigma)-\nu_2(\sigma)| \,,
\]
i.e., half the $L^1$-distance between the two measures.
Let $(X_t)$ be an ergodic finite Markov chain with stationary measure $\pi$. Its total-variation mixing-time, denoted $\tmix(\epsilon)$ for $0<\epsilon<1$, is defined to be
\[ \tmix(\epsilon) \deq \inf\Big\{t \;:\; \max_{x_0 \in \Omega} \| \P_{x_0}(X_t \in \cdot)- \pi\|_\tv \leq \epsilon \Big\}\,,\]
where here and in what follows $\P_{x_0}$ denotes the probability given $X_0=x_0$.
A family of ergodic finite Markov chains $(X_t)$, indexed by an implicit parameter $n$, is said to exhibit \emph{cutoff} (this concept going back to the works~\cites{Aldous,DiSh}) iff the following sharp transition in its convergence to stationarity occurs:
\begin{equation}
\label{eq-cutoff-def}
\lim_{n\to\infty} \frac{\tmix(\epsilon)}{\tmix(1-\epsilon)}=1 \quad\mbox{ for any $0 < \epsilon < 1$}\,.
\end{equation}
That is, $\tmix(\alpha)=(1+o(1))\tmix(\beta)$ for any fixed $0<\alpha<\beta<1$. The \emph{cutoff window} addresses the rate of convergence in~\eqref{eq-cutoff-def}: a sequence $w_n = o\big(\tmix(e^{-1})\big)$ is a cutoff window if $\tmix(\epsilon) = \tmix(1-\epsilon) + O(w_n)$ holds for any $0<\epsilon<1$ with an implicit constant that may depend on $\epsilon$.
Equivalently, if $t_n$ and $w_n$ are sequences with $w_n =o(t_n)$, we say that a sequence of chains exhibits cutoff at $t_n$ with window $w_n$ if
\[\left\{\begin{array}
  {r}\displaystyle{\lim_{\gamma\to\infty} \liminf_{n\to\infty}
 \max_{x_0 \in \Omega} \| \P_{x_0}(X_{t_n-\gamma w_n} \in \cdot)- \pi\|_\tv
  = 1}\,,\\
  \displaystyle{\lim_{\gamma\to \infty} \limsup_{n\to\infty}
 \max_{x_0 \in \Omega} \| \P_{x_0}(X_{t_n+\gamma w_n} \in \cdot)- \pi\|_\tv
  = 0}\,.
\end{array}\right.\]
Verifying cutoff is often quite challenging, e.g., even for the simple random walk on a bounded-degree graph, no examples were known prior to~\cite{LS2}, while this had been conjectured for almost all such graphs.

\subsubsection*{Glauber dynamics for the Ising model}\label{sec:prelim-ising}
Let $G$ be a finite graph $G$ with vertex-set $V$ and edge-set $E$. The Ising model on $G$ is a distribution over the set $\Omega=\{\pm1\}^V$
of possible configurations, each
corresponding to an assignment of plus/minus spins to the sites in $V$. The probability of $\sigma \in \Omega$ is given by
\begin{equation}
  \label{eq-Ising}
  \pi(\sigma)  = Z^{-1} e^{\beta \sum_{uv\in E} \sigma(u)\sigma(v) + h \sum_{u \in V} \sigma(u)} \,,
\end{equation}
where the normalizer $Z=Z(\beta,h)$ is the partition function.
The parameter $\beta$ is the inverse-temperature, which we always to take to be non-negative (ferromagnetic), and $h$ is the external field, taken to be 0 unless stated otherwise. These definitions extend to infinite locally finite graphs (see, e.g.,~\cites{Liggett,Martinelli97}).

The Glauber dynamics for the Ising model (the \emph{Stochastic Ising} model) is a family of continuous-time Markov chains on the state space $\Omega$, reversible w.r.t.\ the Ising measure $\pi$, given by the generator
\begin{equation}
  \label{eq-Glauber-gen}
  (\mathscr{L}f)(\sigma)=\sum_{u\in \Lambda} c(u,\sigma) \left(f(\sigma^u)-f(\sigma)\right)
\end{equation}
where $\sigma^u$ is the configuration $\sigma$ with the spin at $u$ flipped and $c(u,\sigma)$ is the rate of flipping (cf.~\cite{Liggett}).
We focus on the two most notable examples of Glauber dynamics, each having an intuitive and useful graphical interpretation where each site receives updates via an associated i.i.d.\ rate-one Poisson clock:
\begin{compactenum}[(i)]
\item \emph{Metropolis}: flip $\sigma(u)$ if the new state $\sigma^u$ has a lower energy (i.e., $\pi(\sigma^u)\geq \pi(\sigma)$), otherwise perform the flip with probability $\pi(\sigma^u)/\pi(\sigma)$.
    This corresponds to $  c(u,\sigma) = \exp\left(2\beta\sigma(u)\sum_{v \sim u}\sigma(y)\right)  \;\wedge\; 1 $.
\item \emph{Heat-bath}:  erase $\sigma(u)$ and replace it with a sample from the conditional distribution given the spins at its neighboring sites. This corresponds to $c(u,\sigma) = 1/\left[1+ \exp\left(-2\beta\sigma(u)\sum_{v \sim u}\sigma(v)\right)\right]$.
\end{compactenum}
It is easy to verify that these chains are indeed ergodic and reversible w.r.t.\ the Ising distribution $\pi$.
Until recently, sharp mixing results for this dynamics were obtained in relatively few cases, with cutoff only known for the complete graph~\cites{DLP,LLP} prior to the works~\cites{LS1,LS3}.

\begin{figure}[t]
\begin{center}
\vspace{-0.1cm}
 \includegraphics[width=.6\textwidth]{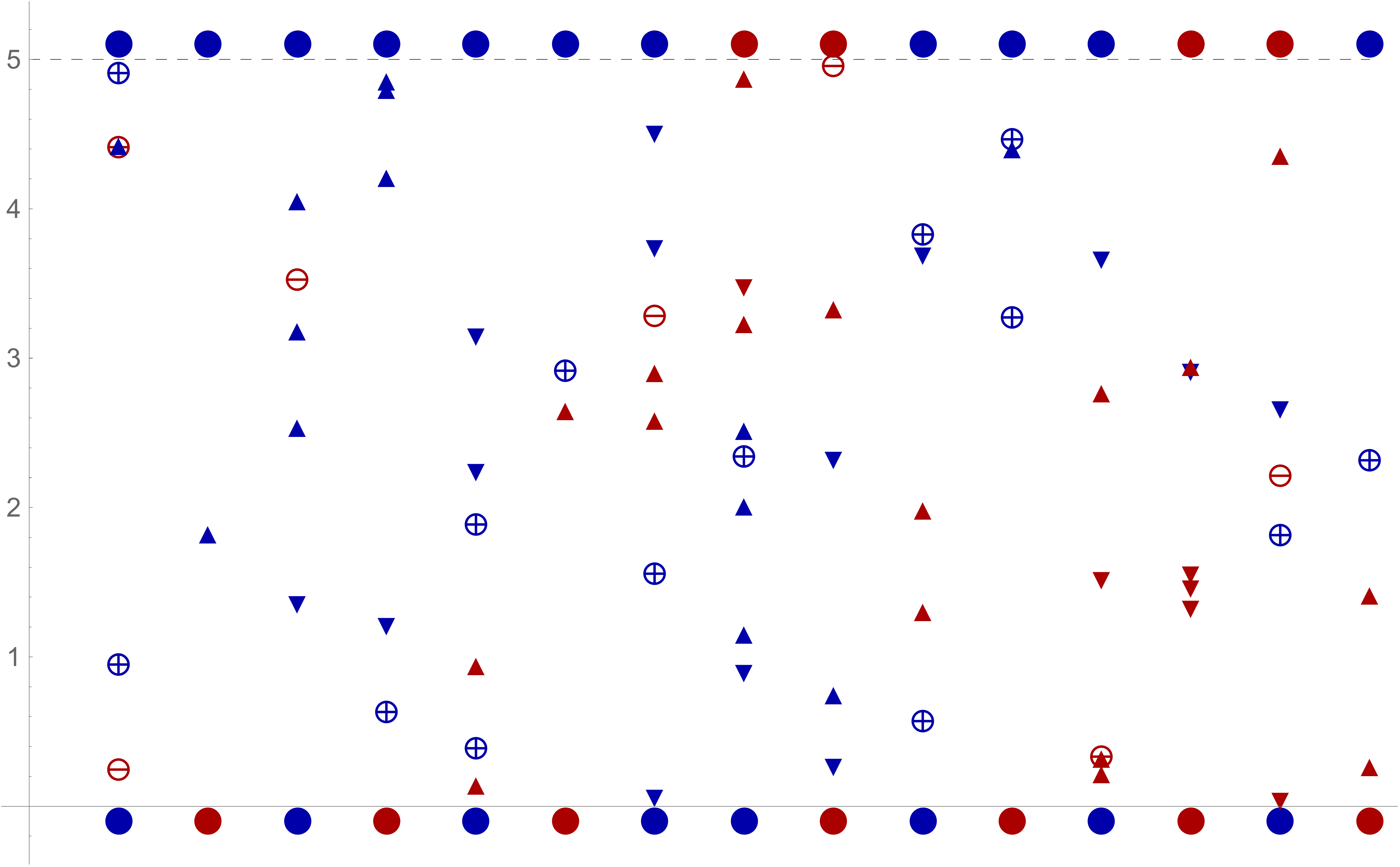}
\end{center}
\vspace{-0.4cm}
\caption{Heat-bath dynamics for the \oned Ising model, marking oblivious updates by $\oplus$ and $\ominus$ and non-oblivious updates by $\blacktriangle$ and $\blacktriangledown$ to denote $\sigma(i)\mapsto\sigma(i-1)\wedge \sigma(i+1)$ and $\sigma(i)\mapsto\sigma(i-1)\vee \sigma(i+1)$.}
\label{fig:upd1d}
\end{figure}

\subsection{Update history and support}\label{sec:update-hist}

The \emph{update sequence} along an interval $(t_0,t_1]$ is a set of tuples $(J,U,T)$, where $t_0< T\leq t_1$ is the update time,
$J\in \Lambda$ is the site to be updated and $U$ is a uniform unit variable.
Given this update sequence, $X_{t_1}$ is a deterministic function of $X_{t_0}$, right-continuous w.r.t.\ $t_1$.
(For instance, in heat-bath Glauber dynamics, if $s$ is the sum of spins at the neighbors of $J$ at time $T$ then the update $(J,U,T)$ results in a minus spin if $U \leq \frac12(1-\tanh(\beta s))$, and in a plus spin otherwise.)

We call a given update $(J,U,T)$ an \emph{oblivious update} iff $U \leq \theta$ for
\begin{equation}\label{eq-def-Xi}
  \theta = \theta_{\beta,d} := 1 - \tanh(\Delta\beta)\qquad\mbox{ where $\Delta=2d$ is the vertex degree}\,,
\end{equation}
since in that situation one can update the spin at $J$ to plus/minus with equal probability (that is, with probability $\theta/2$ each via the same $U$) independently of the spins at the neighbors of the vertex $J$, and a properly chosen rule for the case  $U > \theta$ legally extends this protocol to the Glauber dynamics. (For instance, in heat-bath Glauber dynamics, the update is oblivious if $U\leq 1-\theta/2$ or $U\geq 1-\theta/2$, corresponding to minus and plus updates, respectively; see Figure~\ref{fig:upd1d} for an example in the case $d=1$.)

The following functions will be used to unfold the update history of a set $A$ at time $t_2$ to time $t_1<t_2$:
\begin{itemize}[\noindent$\bullet$]
\item \emph{The update function $\fupd(A,t_1,t_2)$}: the random set that, given the update sequence along the interval $(t_1,t_2]$, contains every site $u\in \Lambda$ that $A$ ``reaches'' through the updates in reverse chronological order;
    that is, every $u\in\Lambda$ such that there exists a subsequence of the updates $(J_i,U_i,T_i)$ with increasing $T_i$'s in the interval $(t_1,t_2]$, such that $J_1,J_2,\ldots$ is a path in $\Lambda$ that connects $u$ to some vertex in $A$.

\item \emph{The update support function $\fsup(A,t_1,t_2)$}: the random set whose value, given the update sequence along the interval $(t_1,t_2]$, is the \emph{update support} of $X_{t_2}(A)$ as a function of $X_{t_1}$; that is, it is the minimal subset $S\subset \Lambda$ which determines the spins of $A$ given the update sequence (this concept from~\cite{LS1} extends more generally to random mapping representations of Markov chains, see Definition~\ref{def-mc-support}).

\end{itemize}

The following lemma establishes the exponential decay of both these update functions for any $\beta<\beta_c$. Of these, $\fsup$ is tied to the magnetization $\sm_t$ whose exponential decay, as mentioned in~\S\ref{sec:methods}, in a sense characterizes the one phase region $\beta<\beta_c$ and serves as a keystone to our analysis of the subcritical nature of the information percolation clusters.
Here and in what follows, for a subset $A \subset \Z^d$ and $r>0$, let $B(A,r)$ denote the set of all sites in $\Z^d$ with $L^\infty$ distance at most $r$ from $A$.

\begin{lemma}\label{l:BasicEstimates}
The update functions for the Ising model on $\Lambda=(\Z/n\Z)^d$ satisfy the following for any $\beta<\beta_c$.
There exist some constant $c_{\beta,d}>0$ such that for any $\Lambda'\subset\Lambda$, any vertex $v\in \Lambda'$ and any $h>0$,
\begin{equation}\label{e:BasicSupport}
\P(\fsup(v,t-h,t)\neq \emptyset) =  \sm_{h} \leq 2 e^{-c_{\beta,d} h}
\end{equation}
with $\sm_h=\sm_h(v)$ as defined in~\eqref{eq-t*-def}, whereas for $\ell> 20 d h$,
\begin{equation}\label{e:BasicUpdate}
\P(\fupd(v,t-h,t)\not\subset B(v, \ell)) \leq  e^{-\ell }\,.
\end{equation}
\end{lemma}
\begin{proof}
The left-hand equality in~\eqref{e:BasicSupport} is by definition, whereas the right-hand inequality was derived from the weak spatial mixing property of the Ising model using the monotonicity of the model in the seminal works of Martinelli and Olivieri~\cites{MO,MO2} (see Theorem 3.1 in~\cite{MO} as well as Theorem 4.1 in~\cite{Martinelli97});
we note that this is the main point where our arguments rely on the monotonicity of the Ising model.
 As it was shown in~\cite{Holley}*{Theorem 2.3} that $\lim_{t\to\infty} \frac{-1}{t} \log \sm_t = \gap$ where $\gap$ is the smallest positive eigenvalue of the generator of the dynamics, this is equivalent to having $\gap$ be bounded away from 0.

  We therefore turn our attention to~\eqref{e:BasicUpdate}, which is a consequence of the finite speed of information flow vs.\ the amenability of lattices.
Let $\cW$ denote the set of sequences of vertices
\[
\cW=\big\{\tilde{w}=(w_1,w_2,\ldots,w_{\ell}):w_1=v,\, \|w_{i-1}-w_{i}\|_1=1\big\}\,.
\]
For $\fupd(v,t-h,t)\not\subset B(v, \ell)$ to hold there must be some $w\in \cW$  and a sequence $t>t_1>\ldots > t_{\ell} > t-h$ so that vertex $w_i$ was updated at time $t_i$.  If this event holds call it $M_{\tilde{w}}$.
It is easy to see that
\begin{align*}
\P(M_{\tilde{w}}) = \P(\Po(h)\geq \ell)\leq  e^{- \ell (\log(\ell/h) - 1)}\,,
\end{align*}
where the last transition is by Bennet's inequality.  By a union bound over $\cW$ we have that for $\ell>20 dh$,
\begin{equation*}
\P\left(\fupd(v,t-h,t)\not\subset B(v, \ell)\right) \leq  (2d)^{\ell} \exp(- \ell (\log(\ell/h) - 1))\leq  e^{-\ell}\,,
\end{equation*}
thus establishing~\eqref{e:BasicUpdate} and completing the proof.
\end{proof}

\subsection{Red, green and blue clusters}
In what follows, we describe the basic setting of the framework, which will be enhanced in \S\ref{sec:lattice-framework} to support all $\beta<\beta_c$.
Consider some designated target time $\tpluss$ for analyzing the distribution of the dynamics on $\Lambda=(\Z/n\Z)^d$.
The \emph{update support} of $X_{\tpluss}(v)$ at time $t$ is
 \[ \sH_v(t) = \fsup(v,t,\tpluss)\,,\]
 i.e., the minimum subset of sites whose spins at time $t$ determine $X_{\tpluss}(v)$ given the updates along $(t,\tpluss]$.
Developing $\{\sH_v(t):t\leq \tpluss\}$ backward in time, started at time $\tpluss$, gives rise to a subgraph $\sH_v$ of the space-time slab $\Lambda\times[0,\tpluss]$, where we connect $(u,t)$ with $(u,t')$ (a temporal edge) if $u\in\sH_v(t)$ and there are no updates along $(t',t]$, and connect $(u,t)$ with $(u',t)$ (a spatial edge) when $u\in\sH_v(t)$, $u'\notin\sH_v(t)$ and $u'\in \sH_v(t-\delta)$ for any small enough $\delta>0$ due to an update at $(u,t)$
(see Figure~\ref{fig:supp1d}).

\begin{figure}[t]
\begin{center}
  \begin{tikzpicture}[font=\tiny,plotd/.style={draw=black,dotted}]
    \newcommand{\hsep}{7.5cm}
    \node (plot1) at (0,0) {
      \includegraphics[width=.4\textwidth]{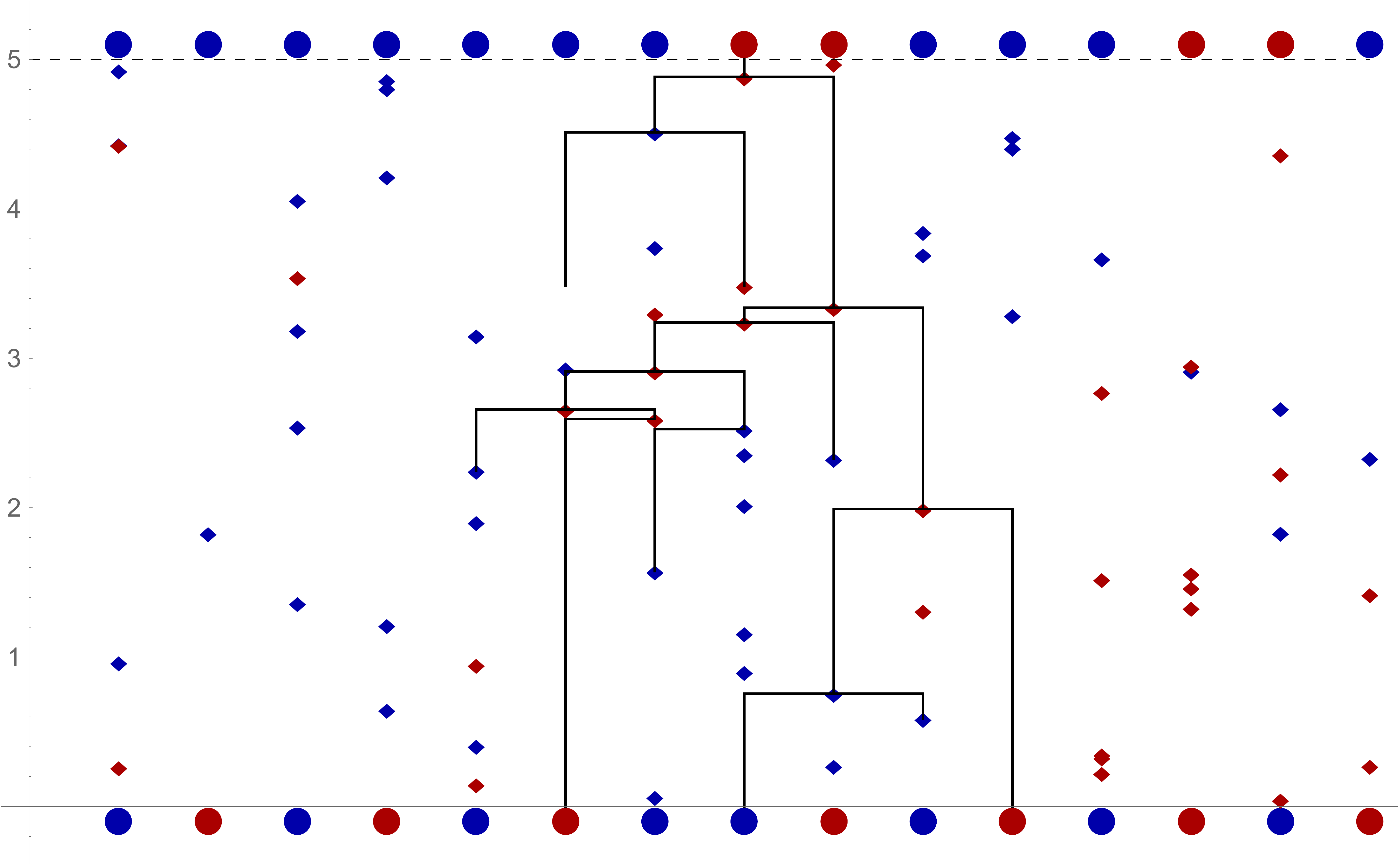}};
    \node[plotd] (plot2) at (\hsep,0) {
      \includegraphics[width=.4\textwidth]{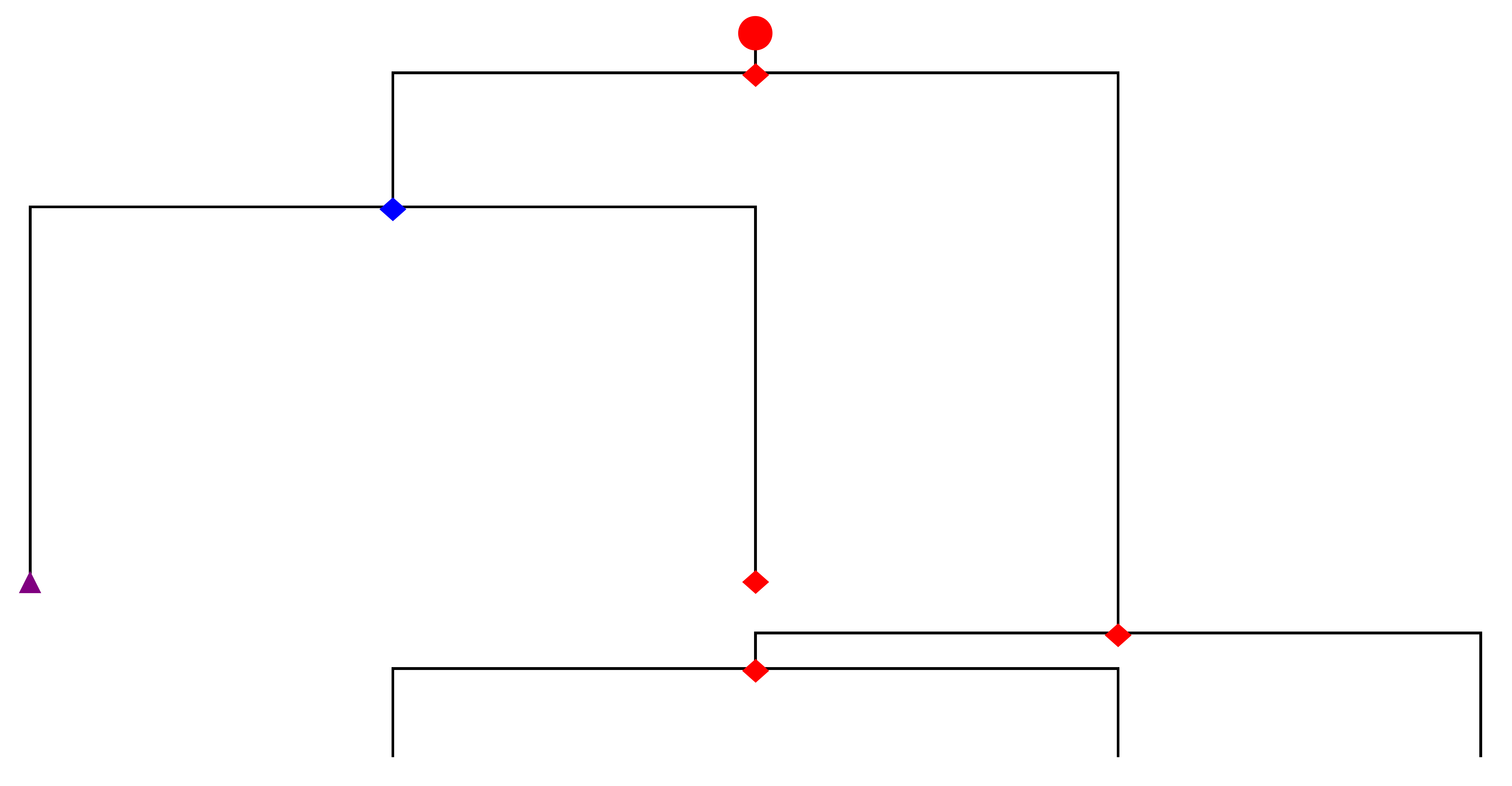}};
    \begin{scope}[shift={(plot2.south west)}]
    \node at (3.5,3.15) {$0.41$};
    \node at (8,3.15) {$x_2 \wedge x_4$};
    \node at (1.88,2.55) {$0.74$};
    \node at (8,2.55) {$(x_1 \vee x_3)\wedge x_4$};
    \node at (3.2,1.05) {$0.59$};
    \node at (8,1.05) {$x_4$};
    \node at (5.55,0.7) {$0.25$};
    \node at (8,0.7) {$x_3\wedge x_5$};
    \node at (3.53,0.45) {$0.19$};
    \node at (8,0.4) {$x_1 \wedge x_3\wedge x_5$};
    \node at (0.3,-0.13) {$x_1$};
    \node at (1.9,-0.13) {$x_2$};
    \node at (3.55,-0.13) {$x_3$};
    \node at (5.25,-0.13) {$x_4$};
    \node at (6.85,-0.13) {$x_5$};
    \end{scope}
    \begin{scope}[shift={(plot1.south west)}]
    \draw[black,dotted] (2.75,2.7) -- (2.75,4.25) -- (4.8,4.25) -- (4.8,2.7) -- cycle;
    \end{scope}
  \end{tikzpicture}
\end{center}
\caption{Update support for heat-bath dynamics for the \oned Ising model at $\beta=0.4$ ($\theta\approx 0.34$); zoomed-in part shows the update history with the root spin as a deterministic function of the leaves.}
\label{fig:supp1d}
\end{figure}
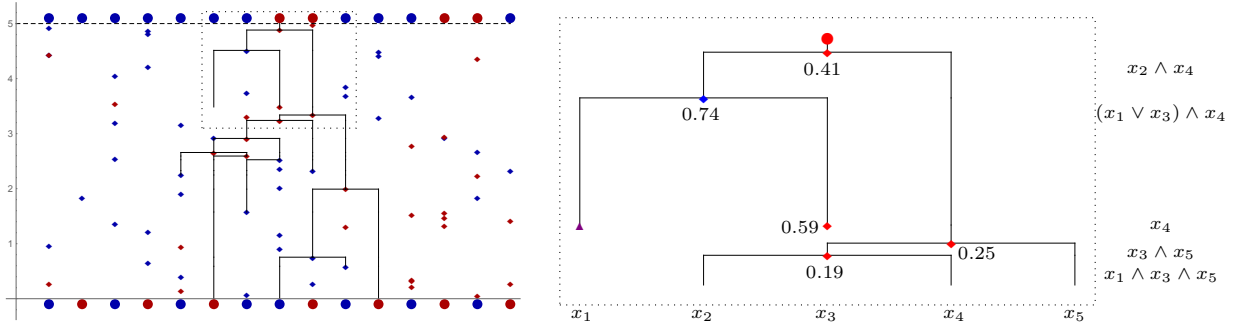

\begin{figure}
\begin{center}
  \begin{tikzpicture}[font=\tiny]
  \node (plot1) at (0,0) {
  \includegraphics[width=.265\textwidth]{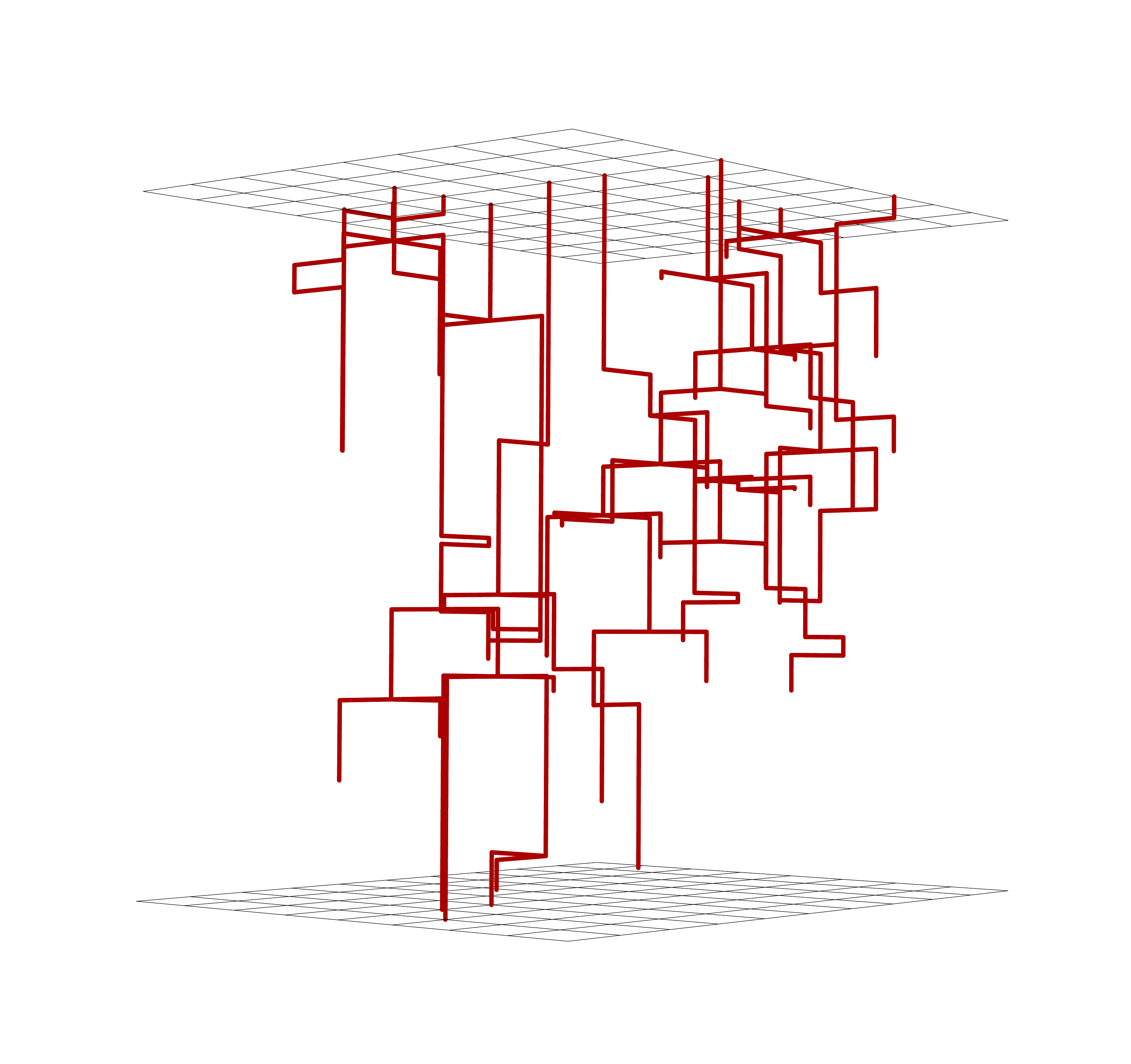}};
  \node (plot2) at (3.2,.1) {
  \includegraphics[width=.12\textwidth]{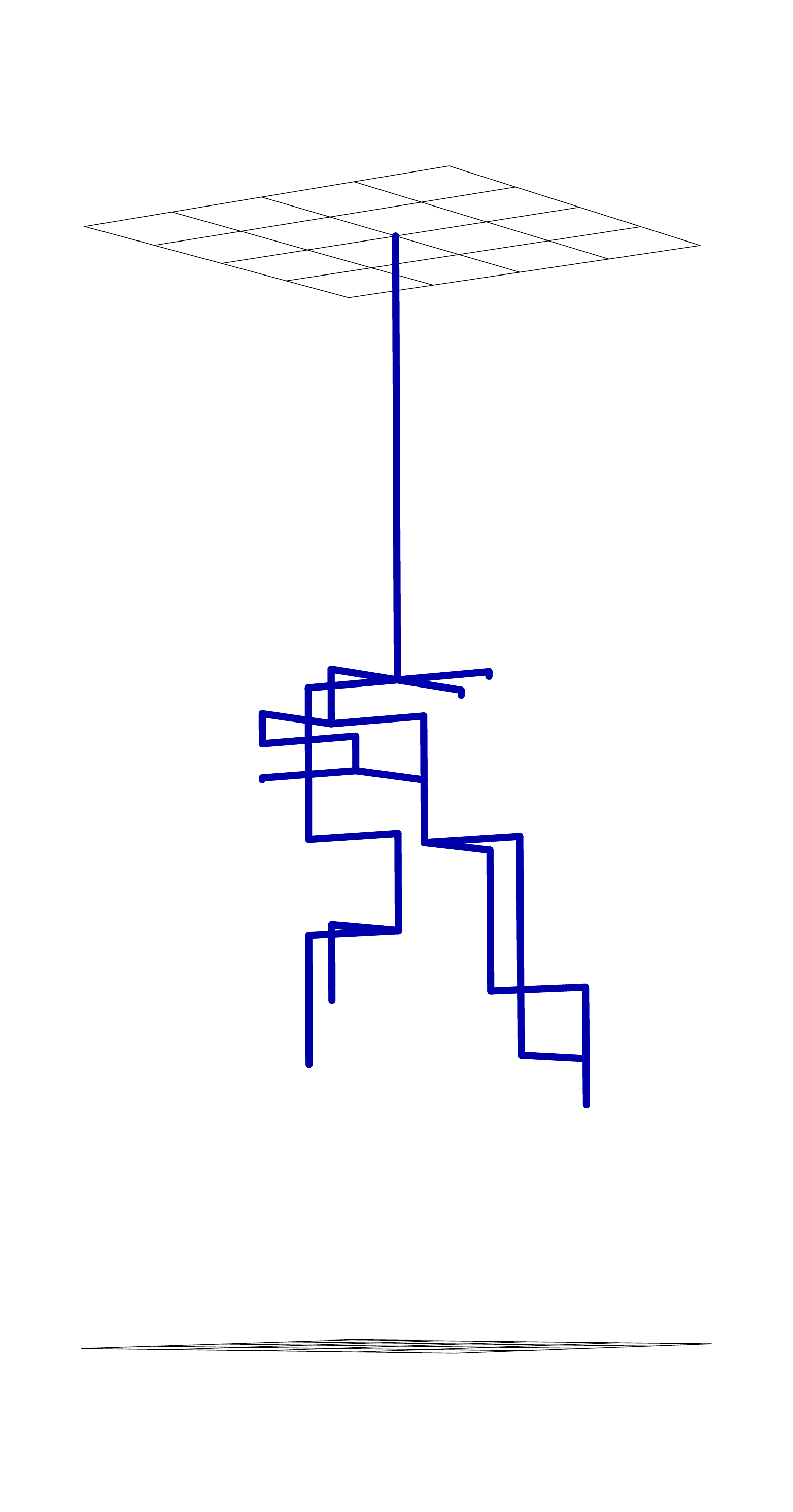}};
  \node (plot3) at (6,0.15) {
  \includegraphics[width=.21\textwidth]{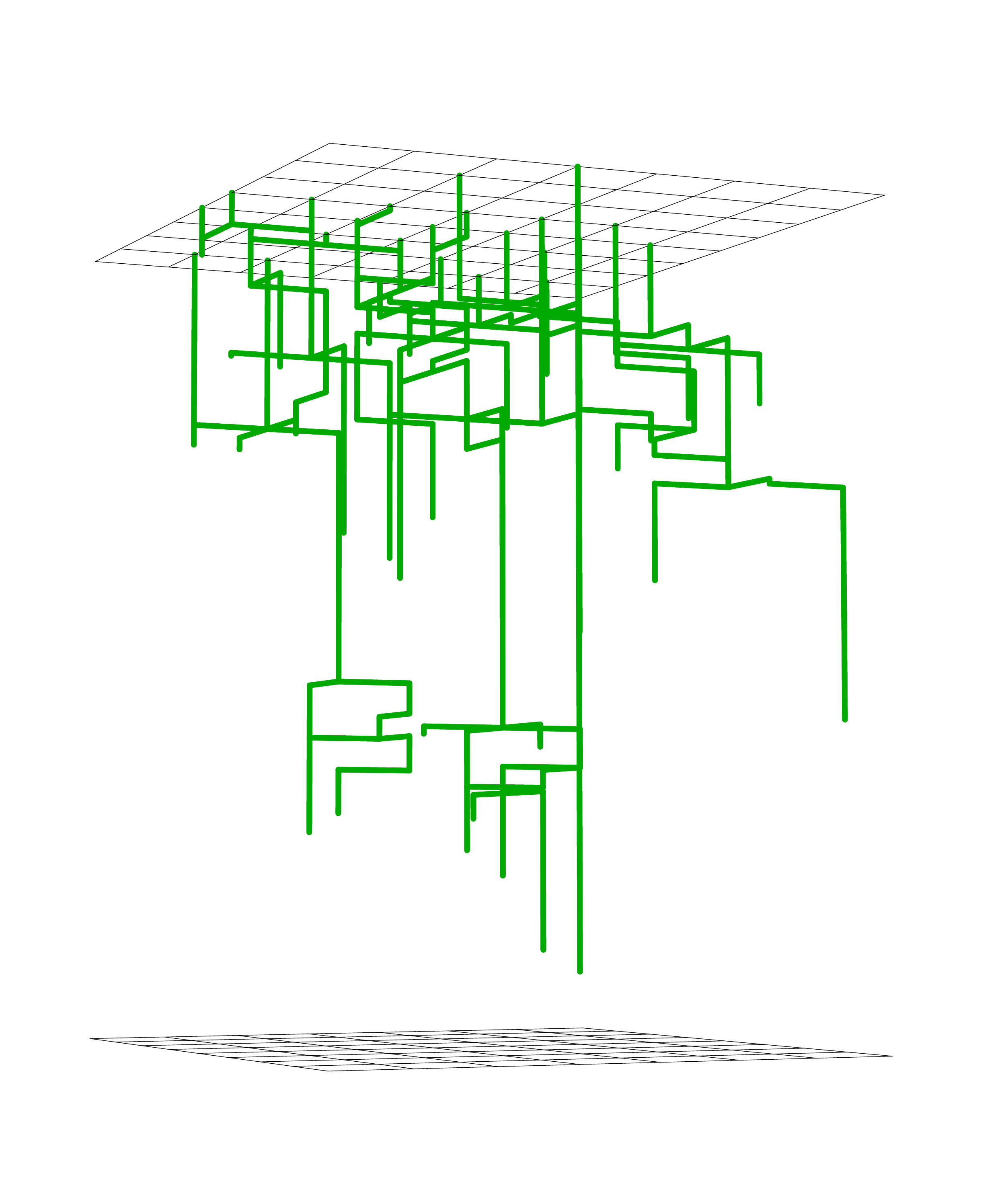}};
  \begin{scope}[shift={(-2.5,-2.33)}]
  \draw[black,->] (-.5,0.75) -- (-.5,3.9) [above];
  \draw[black,thin] (-.55,.85) -- (-.45,.85);
  \draw[black,thin] (-.55,3.7) -- (-.45,3.7);
  \draw[black,thin,dashed] (-.45,.85) -- (.45,.85);
  \draw[black,thin,dashed] (-.45,3.7) -- (.45,3.7);
    \node at (-.7,0.86) {$0$};
    \node at (-.7,3.7) {$\tpluss$};
  \end{scope}
  \end{tikzpicture}
\end{center}
\vspace{-0.25cm}
\caption{Red, blue and green information percolation clusters as per Definition~\ref{def:rgb-simple}.}
\label{fig:rgbclusters}
\end{figure}
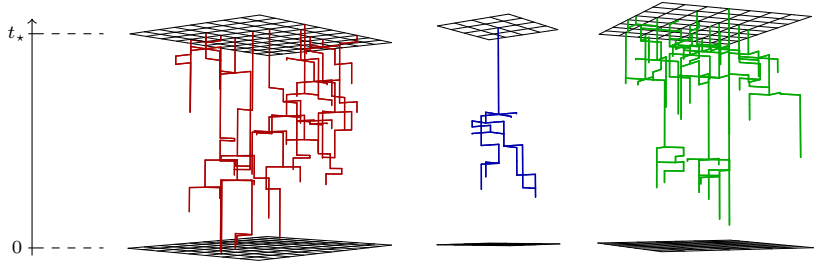

\begin{remark}\label{rem:far-update}
  An oblivious update at $(u,t)$ clearly removes $u$ from $\sH_v(t-\delta)$; however, the support may also shrink due to non-oblivious updates: the zoomed-in update history in Figure~\ref{fig:supp1d} shows $x_1,x_3$ being removed from $\sH_v(t)$ due to the update $x_3 \mapsto x_2\vee x_4$, as the entire function then collapses to $x_4$.
\end{remark}

The \emph{information percolation clusters} are the connected components in the space-time slab $\Lambda\times[0,\tpluss]$ of the aforementioned subgraphs $\{\sH_v : v\in\Lambda\}$.

\begin{definition}\label{def:rgb-simple}
An information percolation cluster is marked \red\ if it has a nonempty intersection with the bottom slab $\Lambda\times\{0\}$; it is \blue\ if it does not intersect the bottom slab and has a singleton in the top slab, $v \times\{\tpluss\}$ for some $v\in\Lambda$; all other clusters are classified as \green. (See Figure~\ref{fig:rgbclusters}.)
\end{definition}
Observe that if a cluster is blue then the distribution of its singleton at the top does not depend on the initial state $x_0$; hence, by symmetry, it is $(\frac12,\frac12)$ plus/minus.

Let $\Lambda_\red$ denote the union of the red clusters, and let $\sH_\red$ be the its collective history --- the union of $\sH_v(t)$ for all $v\in\Lambda_\red$ and $0\leq t<\tpluss$ (with analogous definitions for blue/green).
A beautiful short lemma of Miller and Peres~\cite{MP} shows that, if a measure $\mu$ on $\Omega$ is constructed by sampling a random variable $R\subset \Lambda$ and using an arbitrary law for its spins and a product of Bernoulli($\frac12$) for $\Lambda\setminus R$, then the $L^2$-distance of $\mu$ from the uniform measure is bounded by $\E 2^{|R\cap R'|}-1$ for two i.i.d.\ copies $R,R'$. (See Lemma~\ref{lem:MP} below for a generalization of this, as we will have a product of complicated measures.)
Applied to our setting, if we condition on $\sH_\green$ and look at the spins of $\Lambda\setminus\Lambda_\green$ then $\Lambda_\red$ can assume the role of the variable $R$, as the remaining blue clusters are a product of Bernoulli($\frac12$) variables.

In this conditional space, since the law of the spins of $\Lambda_\green$, albeit potentially complicated, is independent of the initial state, we can safely project the configurations
on $\Lambda\setminus\Lambda_\green$ without increasing the total-variation distance between the distributions started at the two extreme states.
Hence, a sharp upper bound on worst-case mixing will follow by showing for this exponential moment
\begin{equation}
  \label{eq-exp-moment-bound}
  \E \Big[ 2^{|\Lambda_\red \cap\Lambda_\red'|} \;\big|\; \sH_\green\Big] \to 1\quad\mbox{ in probability as }n\to\infty\,,
\end{equation}
by coupling the distribution of the dynamics at time $\tpluss$ from any initial state to the uniform measure.
Finally, with the green clusters out of the picture by the conditioning (which has its own toll, forcing various updates along history so that no other cluster would intersect with those nor become green), we can bound the probability that a subset of sites would become a red cluster by its ratio with the probability of all sites being blue clusters. Being red entails connecting the subset in the space-time slab, hence the exponential decay needed for~\eqref{eq-exp-moment-bound}.

\begin{example}[Red, green and blue clusters in the \oned Ising model]
 Consider the relatively simple special case of $\Lambda=\Z/n\Z$ to illustrate the approach outlined above. Here, since the vertex degree is 2, an update either writes a new spin independently of the neighbors (with probability $\theta$) or, by symmetry, it takes the spin of a uniformly chosen neighbor. Thus, the update history from any vertex $v$ is simply a continuous-time simple random walk that moves at rate $1-\theta$ and dies at rate $\theta$; the collection of these for all $v\in\Lambda$ forms coalescing (but never splitting) histories (recall Figure~\ref{fig:clusters-1d}).

The probability that $\fsup(v,0,t)\neq \emptyset$ (the history of $X_t(v)$ is nontrivially supported on the bottom of the space-time slab) is therefore $e^{-\theta t}$, which becomes $1/\sqrt{n}$  once we take $t=\tcut=\frac1{2\theta}\log n$.
If we ignore the conditioning on the green clusters (which poses a technical difficulty for the analysis---as the red and blue histories must avoid them---but does not change the overall behavior by much), then
$ \P(v\in \Lambda_\red \cap \Lambda'_\red) = \P(\sH_v(0)\neq\emptyset)^2 = e^{-2\theta \tpluss}$
by the independence of the copies $\Lambda_\red,\Lambda'_\red$. Furthermore, if the events $\{v\in\Lambda_\red\cap\Lambda'_\red\}_{v\in\Lambda}$ were mutually independent (of course they are not, yet the intuition is still correct) then $\E[ 2^{|\Lambda_\red\cap\Lambda'_\red|}] = \E \big[ \prod_{v}(1+\one_{\{v\in\Lambda_\red\cap\Lambda'_\red\}})\big]$ would translate into
\[ \prod_{v}\E\left[1+\one_{\{v\in\Lambda_\red\cap\Lambda'_\red\}}\right] = \left(1+e^{-2\theta\tpluss}\right)^n \leq \exp\left(n e^{-2\theta\tpluss}\right)\,,\]
which for $\tpluss=\tcut+s$ is at most $\exp(e^{-2\theta s})$. As we increase the constant $s>0$ this last quantity approaches 1, from which the desired upper bound on the mixing time will follow via~\eqref{eq-exp-moment-bound}.
\end{example}

The above example demonstrated (modulo conditioning on $\sH_\green$ and dependencies between sites) how this framework can yield sharp upper bounds on mixing when the update history corresponds to a subcritical branching process.
However, in dimension $d\geq 2$, this stops being the case midway through the high temperature regime in lattices, and in \S\ref{sec:lattice-framework} we describe the additional ideas that are needed to extend the framework to all $\beta<\beta_c$.

\section{Enhancements for the lattice up to criticality}\label{sec:lattice-framework}
To extend the framework to all $\beta<\beta_c$ we will modify the definition of the support at time $t$ for $t>\tcut$, as well as introduce new notions in the space-time slab both for $t>\tcut$ and for $t<\tcut$ within the cutoff window. These are described in \S\ref{sec:above-ground} and \S\ref{sec:below-ground}, resp., along with three key lemmas (Lemmas~\ref{l:Ai-to-Bi}--\ref{l:redConnection}) whose proofs are postponed to \S\ref{sec:cluster-analysis}.

\begin{figure}
\begin{center}
  \begin{tikzpicture}[font=\tiny]
  \node (plot1) at (0,0) {
  \includegraphics[width=0.9\textwidth]{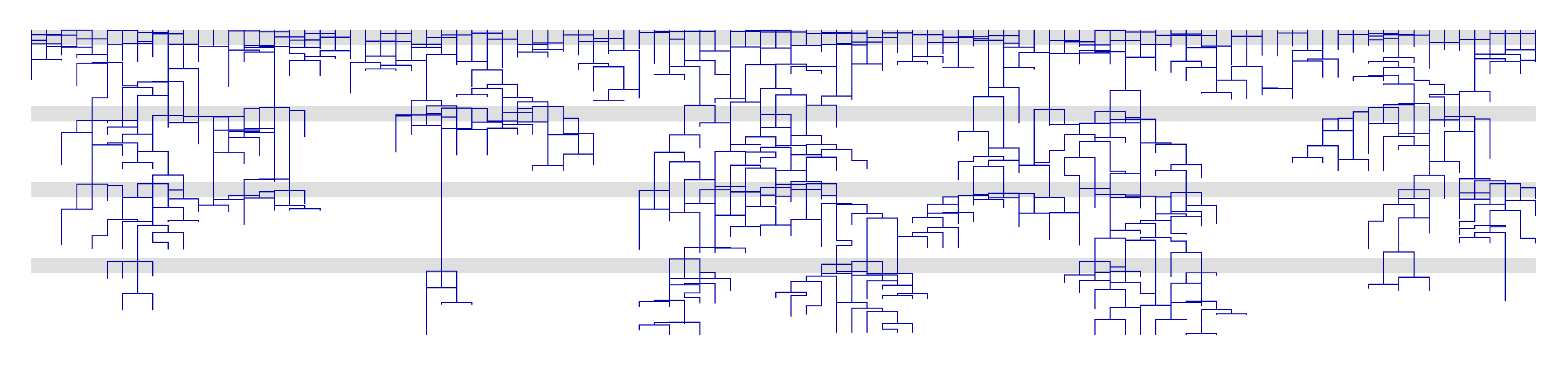}};
  \begin{scope}[shift={(7.8,-1.94)}]
  \node at (-0.1,3.35) {$\cI'_\lambda$};
  \node at (-0.1,2.95) {$\cI_{\lambda}$};
  \node at (0.05,2.58) {$\cI'_{\lambda-1}$};
  \node at (0.05,2.2) {$\cI_{\lambda-1}$};
  \node at (-0.1,1.7) {$\vdots$};
  \node at (-0.1,1.1) {$\cI'_{1}$};
  \node at (-0.1,.7) {$\cI_{1}$};
  \end{scope}
  \begin{scope}[shift={(-7.6,-2.2)}]
  \draw[black,->] (-.5,0.55) -- (-.5,3.9) [above];
  \draw[black,thin] (-.55,3.7) -- (-.45,3.7);
  \draw[black,thin] (-.55,2.95) -- (-.45,2.95);
  \draw[black,thin] (-.55,2.2) -- (-.45,2.2);
  \draw[black,thin] (-.55,1.45) -- (-.45,1.45);
  \draw[black,thin] (-.55,.7) -- (-.45,.7);
  \draw[black,thin,dotted] (-.45,3.7) -- (.15,3.7);
  \draw[black,thin,dotted] (-.45,2.95) -- (.15,2.95);
  \draw[black,thin,dotted] (-.45,2.2) -- (.15,2.2);
  \draw[black,thin,dotted] (-.45,1.45) -- (.15,1.45);
  \draw[black,thin,dotted] (-.45,.7) -- (15.,.7);
  \draw[black,thin,dotted] (-.45,2.05) -- (.15,2.05);
  \draw[black,thin,<->] (-.1,2.05) -- (-.1,2.2);
  \draw[black,thin,<->] (-.1,.7) -- (-.1,1.45);
  \node at (0.07,2.13) {${}_1$};
  \node at (0.3,1.05) {$\scut/\lambda$};
  \node at (-1,3.7) {$\tau_\lambda=\tpluss$};
  \node at (-1,2.95) {$\tau_{\lambda-1}$};
  \node at (-1,2.) {$\vdots$};
  \node at (-1,1.4) {$\tau_{1}$};
  \node at (-1,.7) {$\tau_0=\tcut$};
  \end{scope}
  \end{tikzpicture}
\end{center}
\vspace{-0.25cm}
\caption{Regular phases ($\cI_k$) and deferred phases ($\cI'_k$) of the update history in the range $t>\tcut$: no vertices are removed from the support along the deferred phases (marked by shaded regions).}
\label{fig:deferred}
\end{figure}
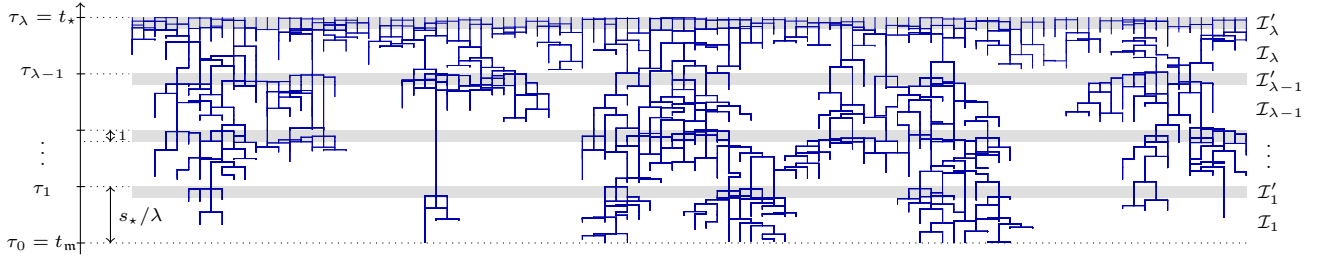

\subsection{Post mixing analysis: percolation components}\label{sec:above-ground}
Let $\lambda>0$ be some large enough integer, and let $\scut>0$ denote some larger constant to be set last. As illustrated in Figure~\ref{fig:deferred}, set
\[ \tpluss = \tcut + \scut\,, \]
for $k=0,1,\ldots,\lambda$ let
\[ \tau_k = \tcut + k \scut/\lambda\,,\]
and partition each interval $(\tau_{k-1},\tau_k]$ for $k=1,\ldots,\lambda$ into the subintervals
\[
\cI_k =(\tau_{k-1},\tau_k-1 ]\,,\qquad \cI_k' =(\tau_k-1, \tau_k]\,.
\]
We refer to $\cI_k$ as a {\it regular phase} and to $\cI'_k$ as a {\it deferred phase}.

\begin{definition*}[the support $\sH_v(t)$ for $t>\tcut$]
Starting from time $\tpluss =\tcut+\scut $ and going backwards to time $\tcut$ we develop the history of a vertex $v\in V$ rooted at time $\tpluss$ as follows:
\begin{itemize}[\indent$\bullet$]
\item Regular phases ($\cI_k$ for $k=1,\ldots,\lambda$):
 For any $\tau_{k-1} < t \leq \tau_k - 1$,
\[ \sH_v(t) = \fsup(\sH_v(\tau_{k}-1),t,\tau_k-1)\,.\]
Note that an oblivious update at time $t$ to some $w\in\sH_v(t)$ will cause it to be removed from the corresponding support (so $w\notin\sH_v(t-\delta)$ for any small enough $\delta>0$), while a non-oblivious update replaces it by a subset of its neighbors. We stress that $w$ may become irrelevant (thus ejected from the support) due to an update to some other, potentially distant, vertex $z$ (see Remark~\ref{rem:far-update}).
\item Deferred phases ($\cI'_k$ for $k=1,\ldots,\lambda$): For any $\tau_{k}-1 < t\leq\tau_k$,
\[ \sH_v(t) = \fupd(\sH_v(\tau_{k}),t,\tau_k)\,.\]
Here vertices do not leave the support: an update to $w\in\sH_v(t)$ adds its $2d$ neighbors to $\sH_v(t-\delta)$.
\end{itemize}
Recalling the form $(J_i,U_i,T_i)$ of updates (see~\S\ref{sec:update-hist}), let the \emph{undeferred randomness} $\cU$ be the updates along $(\tcut,\tpluss]$ excluding the uniform unit variables $U_i$ when $T_i\in\cup_k\cI'_k$, and let the \emph{deferred randomness} $\cU'$ denote this set of excluded uniform unit variables (corresponding to updates in the deferred phases).
\end{definition*}

\begin{remark*}
Observe that this definition of $\{\sH_v(t) : t>\tcut\}$ is  a function of the undeferred randomness $\cU$ alone (as the deferred phases $\cI'_k$ involved $\fupd$ as opposed to $\fsup$); thus, $X_{\tpluss}$ may be obtained from $X_{\tcut}$ by first exposing $\cU$, then incorporating the deferred randomness $\cU'$ along the deferred phases $\cI'_k$.
\end{remark*}

\begin{remark*}
The goal behind introducing the deferred phases $\cI'_k$ is to boost the subcritical behavior of the support $\sH_v(t)$ towards an analog of the exponential moment in~\eqref{eq-exp-moment-bound}. In what follows we will describe how the set of sites with $\sH_v(\tcut)\neq\emptyset$ are partitioned into components (according to proximity and intersection of their histories); roughly put, by exposing $\cU$ but not $\cU'$ one can identify, for each set of vertices $B$ in such a component, a time $t$ in which $\sH_B(t)$ is suitably ``thin'' (we will refer to it as a ``cut-set'') so that --- recalling that a branch of the history is killed at rate $\theta>0$ via oblivious updates --- one obtains a good lower bound on the probability of arriving at any configuration for the spin-set $\sH_B(t)$ (which then determines $X_{\tpluss}(B)$) once the undeferred randomness $\cU'$ is incorporated.
\end{remark*}

\begin{figure}
\begin{center}
  \begin{tikzpicture}[font=\tiny,clip=false]
  \node (plot1) at (-1,0) {
  \includegraphics[width=\textwidth]{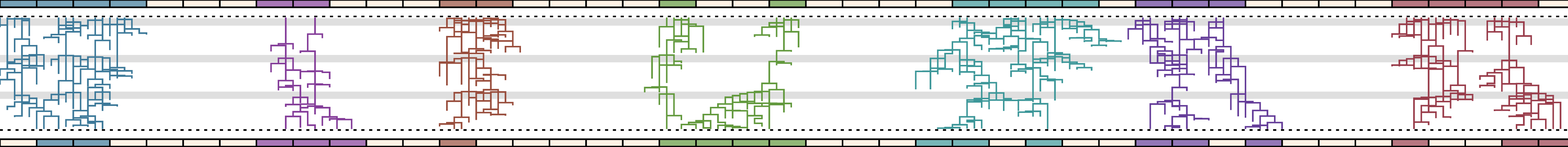}};
  \begin{scope}[shift={(plot1.south west)}]
  \node at (17.43,1.54) {$\tpluss$};
  \node at (17.45,.35) {$\tcut$};
  \node at (1.,1.9) {$B_1$};
  \node at (1.,-0.1) {$A_1$};
  \node at (3.35,1.9) {$B_2$};
  \node at (3.55,-0.1) {$A_2$};
  \node at (5.35,1.9) {$B_3$};
  \node at (5.2,-0.1) {$A_3$};
  \node at (8.1,1.9) {$B_4$};
  \node at (8.2,-0.1) {$A_4$};
  \node at (11.4,1.9) {$B_5$};
  \node at (11.1,-0.1) {$A_5$};
  \node at (13.2,1.9) {$B_6$};
  \node at (13.4,-0.1) {$A_6$};
  \node at (16.15,1.9) {$B_7$};
  \node at (16.1,-0.1) {$A_7$};
  \draw[darkgray,thin,<->] (9.35,1.82) -- (9.755,1.82);
  \draw[darkgray,thin] (9.35,1.87) -- (9.35,1.77);
  \draw[darkgray,thin] (9.755,1.87) -- (9.755,1.77);
  \node[darkgray] at (9.55,2) {$\scut^2$};
  \end{scope}
  \end{tikzpicture}
\end{center}
\vspace{-0.25cm}
\caption{\!\!Components $\cA=\{A_i\}$ and $\cB=\{B_i\}$ corresponding to $\{\Upsilon_i\}$. Each $\Upsilon_i$ joins vertices of $\Upsilon$ via \emph{history intersection}; $\Upsilon_4$ and $\Upsilon_7$ also join vertices via \emph{final proximity} and \emph{initial proximity}, resp.}
\label{fig:blocks}
\end{figure}
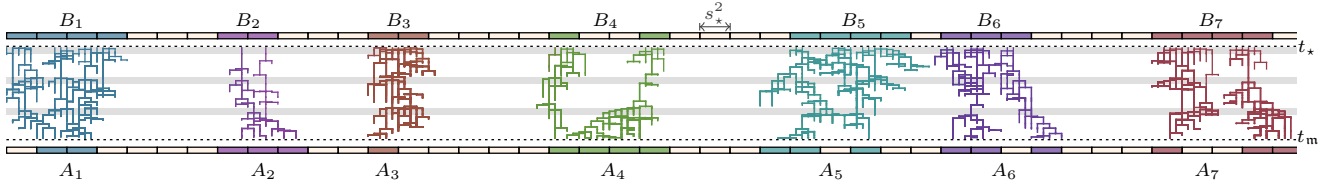

\subsubsection*{Blocks of sites and components of blocks}
Partition $\Z^d$ into boxes of side-length $\scut^2$, referred to as {\em blocks}. We define {\em block components}, composed of subsets of blocks, as follow (see Figure~\ref{fig:blocks}).

\begin{definition*}
  [Block components]
Given the undeferred update sequence $\cU$, we say that $u\sim v$ if one of the following conditions holds:
\begin{compactenum}
  \item History intersection: $\sH_{u}(t) \cap \sH_v(t)\neq\emptyset$ for some $t\in (\tcut,\tpluss]$.
  \item Initial proximity: $u,v$ belong to the same block or to adjacent ones.
  \item Final proximity: there exist $u'\in \sH_u(\tcut)$ and $v'\in \sH_v(\tcut)$ belonging to the same block or to adjacent ones.
\end{compactenum}
Let $\Upsilon = \{ v : \sH_v(\tcut) \neq \emptyset\}$ be the vertices whose history reaches $\tcut$.
We partition $\Upsilon$ into components $\{\Upsilon_i\}$ via the transitive closure of $\sim$.
Let $\sH_{\Upsilon_i}(t) = \cup_{v\in\Upsilon_i} \sH_v(t)$, let $A_i$ be the minimal set of blocks covering $\sH_{\Upsilon_i}(\tcut)$ and let $B_i$ be the minimal set of blocks covering $\sH_{\Upsilon_i}(\tpluss) = \Upsilon_i$.
The collection of all components $\{A_i\}$ is denoted $\cA=\cA(\cU)$ and the collection of all components $\{B_i\}$ is denoted $\cB=\cB(\cU)$.
\end{definition*}

We now state a bound on the probability for witnessing a given set of blocks.
In what follows, let $\anim(S)$ denote the size, in blocks, of the minimal lattice animal\footnote{A lattice animal is a connected subset of sites in the lattice.}
containing the block-set $S$. Further let $\{R\conn S\}$ denote the event for some $i$ the block-sets $(R,S)$ satisfy $R=A_i\in \cA$ and $S=B_i\in\cB$; i.e., $R,S$ correspond to the same component, some $\Upsilon_i$, as the minimal block covers of $\sH_{\Upsilon_i}(t)$ at $t=\tcut,\tpluss$.
\begin{lemma}
  \label{l:Ai-to-Bi}
Let $\beta<\beta_c$.   There exist constants $c(\beta,d),\lambda_0(\beta,d)$ such that, if $\lambda > \lambda_0$ and $\scut > \lambda^2$, then for every collection of pairs of block-sets $\{(R_i,S_i)\}$,
  \[ \P\bigg(\bigcap_i \{R_i \conn S_i\}\bigg)  \leq \exp\bigg[ -c \frac{\scut}\lambda \sum_i \anim(R_i\cup S_i)\bigg]\,.\]
\end{lemma}

\subsubsection*{Cut-sets of components}
The cut-set of a component $A_i$ is defined as follows.
For $k=1,\ldots,\lambda$ let
\begin{align*}
  \chi_{i,k} &= \sH_{\Upsilon_i}(\tau_{k})\,,\\
  \Xi_{i,k} &= \prod_{v\in \chi_{i,k}} \left(\tfrac14 \theta T_{v,k}\right)
\end{align*}
where $\theta = 1 - \tanh(2d\beta)$ is the oblivious update probability, and
$T_{v,k}$ is the time elapsed since the last update to the vertex $v$ within
the deferred phase $\cI_k'$ until $\tau_k$, the end of that interval. That is, $T_{v,k}=(\tau_k - t) \wedge 1$ for the maximum time $t < \tau_k $ at which there was an update to $v$.
With this notation, the \emph{cut-set} of $A_i$ is the pair $(k_i,\chi_i)$ where $k_i$ is the value of $1\leq k \leq \lambda$ minimizing $\Xi_{i,k}$ and $\chi_i=\chi_{i,k_i}$.
The following lemma will be used to estimate these cut-sets.
\begin{lemma}\label{l:exp-cut-set}
Let $\beta<\beta_c$. Let $S$ be a set of blocks, let $\chi_k(S)=\sH_{S}(\tau_{k})$ and
$\Xi_{k}(S) = \prod_{v\in \chi_{k}(S)} \left(\tfrac14 \theta T_{v,k} \right)$
where $T_{v,k}$ is the time that elapsed since the last update to the vertex $v$ within
the deferred phase $\cI_k'$ until $\tau_k$.
If $\lambda$ is large enough in terms of $\beta,d$ and $\scut$ is large enough in terms of $\lambda$
then
\[ \E \left[\min_k \left\{ (\Xi_k(S))^{-4} : 1 \leq k \leq \lambda \right\} \right]\leq 2^{\lambda+3} e^{|S|}\,.\]
\end{lemma}

\subsection{Pre mixing analysis: percolation clusters}\label{sec:below-ground}

Going backwards from time $\tcut$ to time $0$, the history is defined in the same way as it was in the regular phases (see \S\ref{sec:above-ground}); that is,
for any $0 < t \leq \tcut$,
\[ \sH_v(t) = \fsup(\sH_v(\tcut),t,\tcut)\,.\]
Further recall that the set of sites $\Upsilon=\{v : \sH_v(\tcut)\neq\emptyset\}$ were partitioned into components $\Upsilon_i$ (see~\S\ref{sec:above-ground}), and for each $i$ we let $A_i$ be the minimal block-set covering $\sH_{\Upsilon_i}(\tcut)$.

\begin{definition*}
  [Information percolation clusters]
  We write $A_i \sim A_j$ if the supports of these components satisfy either one of the following conditions (recall that $B(V,r)=\{x\in\Z^d : \min_{v\in V}\|x-v\|_\infty<r\}$).
\begin{enumerate}
  \item Intersection:  $\fsup(A_i,t,\tcut)\cap \fsup(A_j,t,\tcut)\neq\emptyset$
  for some $0 \leq t<\tcut$.

  \item Early proximity:
       $\fsup(A_i,t,\tcut) \cap B(A_j,\scut^2/3) \neq\emptyset$  for some $\tcut-\scut \leq t \leq \tcut$, or the analogous statement when the roles of $A_i,A_j$ are reversed.
\end{enumerate}
We partition $\cA=\{A_i\}$ into clusters $\cC^{(1)},\cC^{(2)},\ldots$ according to the transitive closure of the above relation, and then classify these clusters into three color groups:
\begin{itemize}[\indent$\bullet$]
\item \blue: a cluster $\cC^{(k)}$ consisting of a single $A_i$ (for some $i=i(k)$) which dies out within the interval $(\tcut-\scut,\tcut]$ without exiting the ball of radius $\scut^2/3$ around $A_i$:
\[ \cC^{(k)}=\{A_i\}~,~\bigcup_{v\in A_i} \fsup(v,\tcut-\scut,\tcut) = \emptyset~,~
\bigcup_{t > \tcut-\scut,\; v\in A_i} \fsup(v,t,\tcut) \subset B(A_i,\scut^2/3)\,.\]
\item \red: a cluster $\cC^{(k)}$ containing a vertex whose history reaches time $0$:
\[ \bigcup_{v\in A_i\in\cC^{(k)}} \fsup(v,0,\tcut) \neq \emptyset\,.\]
\item \green: all other clusters (neither red nor blue).
\end{itemize}
\end{definition*}
Let $\cA_\red$ be the set of components whose cluster is red, and let $\sH_\red$ be the collective history of all $v\in\cA_\red$ going backwards from time $\tcut$, i.e.,
\[ \sH_\red = \big\{ \sH_v(t) \;:\; v\in\cA_\red~,~ t \leq \tcut\big\}\,,\]
setting the corresponding notation for blue and green clusters analogously.

For a generic collection of blocks $\cC$, the collective history of all $v\notin \cC$ is defined as
\[ \sH^-_\cC = \big\{ \sH_v(t) : v\notin \cup\{A_i\in\cC\}~,~t\leq \tcut \big\}\,,\]
%
and we say $\sH^-_\cC$ is \emph{$\cC$-compatible} if there is a positive probability that $\cC$ is a cluster conditioned on $\sH^{-}_\cC$.

Central to the proof will be to understand the conditional probability of a set of blocks $\cC$ to be a single red cluster as opposed to a collection of blue ones (having ruled out green clusters by conditioning on $\sH_\green$)
given the undeferred randomness $\cU$ and the history up to time $\tcut$ of all the other vertices:
\begin{equation}\label{eq-def-psi}
\Psi_{\cC} = \sup_{\cX} \P\left(\cC\in \red \mid \sH_\cC^-=\cX\,,\,\{\cC\in\red\} \cup \{\cC \subset \blue\}\,,\,\cU\right)
\end{equation}
The next lemma bounds $\Psi_\cC$ in terms of the lattice animals for $\cC$ and the individual $A_i$'s; note that the dependence of this estimate for $\Psi_\cC$ on $\cU$ is through the geometry of the components $A_i$.
 Here and in what follows we let $|x|^+=x\wedge 0$ denote the positive part of $x$.
\begin{lemma}\label{l:redConnection}
Let $\beta<\beta_c$. There exists $c(\beta,d),s_0(\beta,d)>0$ such that, for any $\scut>s_0$, any large enough $n$ and every $\cC\subset \cA$,  the quantity $\Psi_\cC$ from~\eqref{eq-def-psi} satisfies
\begin{align}\label{e:redConnection}
\Psi_\cC \leq  \frac{\scut^{4d}}{\sqrt{|\Lambda|}} e^{4\sum_i |A_i|-c \scut \left|\anim\left(\cC\right) - \sum_{A_i\in\cC}\anim(A_i)\right|^+ }\,.
\end{align}
\end{lemma}

\section{Cutoff with a constant window}\label{sec:inf-perc}

\subsection{Upper bound modulo Lemmas~\ref{l:Ai-to-Bi}--\ref{l:redConnection}}
Let $\cU$ be the undeferred randomness along $(\tcut,\tpluss]$ (the update sequence excluding the uniform unit variables of updates in the deferred phases $\cup_{k=1}^\lambda \cI'_k$).
Let $\bar{d}(t,\cU)$ be the coupling time conditioned on this update sequence, that is,
\begin{align*}
  \bar{d}(t,\cU) = \sup_{x_0,y_0} \left\|\P_{x_0}(X_t\in\cdot\mid \cU)-\P_{y_0}(X_t\in \cdot\mid \cU)\right\|_{\tv}\,.
\end{align*}
Towards an upper bound on $\bar{d}(t,\cU)$ (which will involve several additional definitions; see~\eqref{eq-d(t*+s,U)-bound2} below),
our starting point would be to consider the notion of the {\em support of a random map}, first introduced in~\cite{LS1}. Its following formulation in a more general framework appears in~\cite{LS3}.
Let $K$ be a transition kernel of a finite Markov chain. A \emph{random mapping representation} for $K$ is a pair $(g,W)$ where $g$ is a deterministic map and $W$ is a random variable such that $\P(g(x,W)=y)=K(x,y)$ for all $x,y$ in the state space of $K$. It is well-known that such a representation always exists.

\begin{definition}[\emph{Support of a random mapping representation}]\label{def-mc-support} Let $K$ be a Markov chain on a state space $\Sigma^\Lambda$ for some finite sets $\Sigma$ and $\Lambda$. Let $(g,W)$ be a random mapping representation for $K$. The \emph{support} corresponding to $g$ for a given value of $W$ is the minimum subset $\Lambda_{W} \subset \Lambda$ such that $g(\cdot,W)$ is determined by $x(\Lambda_{W})$ for any $x$, i.e.,
\[ g(x,W) = f_{W}(x(\Lambda_{W})) \mbox{ for some $f_{W}:\Sigma^{\Lambda_{W}}\to \Sigma^{\Lambda}$ and all $x$.}\]
That is, $v \in \Lambda_{W}$ if and only if there exist $x,x'\in \Sigma^\Lambda$ differing only at $v$ such that $g(x,W)\neq g(x',W)$.
\end{definition}

\begin{lemma}[\cite{LS3}*{Lemma~3.3}]\label{lem-support} Let $K$ be a finite Markov chain and let $(g,W)$ be a random mapping representation for it. Denote by $\Lambda_{W}$ the support of $W$ w.r.t.\ $g$ as per Definition~\ref{def-mc-support}. Then for any distributions $\phi,\psi$ on the state space of $K$,
\begin{align*}
\left\| \phi K -\psi K\right\|_\tv &\leq \int \left\|\phi|_{\Lambda_{W}}-\psi|_{\Lambda_{W}} \right\|_\tv d \P({W}).
 \end{align*}
\end{lemma}

To relate this to our context of seeking an upper bound for $\bar{d}(t,\cU)$,
recall that (as remarked following the definition of the regular and deferred phases in~\S\ref{sec:lattice-framework}) $\sH_v(t)$ for $t>\tcut$ is a function of the undeferred randomness $\cU$ alone. Hence, both the $A_i$'s and their corresponding cut-sets $(k_i,\chi_i)$ are completely determined from $\cU$.
Letting $Z_\chi$ denote the joint distribution of $X_{\tau_{k_i}}(\chi_i)$ for all the components $A_i$, we can view $(X_{\tpluss}\in\cdot\mid\cU)$ as a random function of $(\cup_i X_{\tau_{k_i}}(\chi_i)\mid \cU)$ whose randomness arises from the deferred updates $\cU'$ (using which every $X_{\tpluss}(v)$ for $v\in\Upsilon_i$ can be deduced from $X_{\tau_{k_i}}(\chi_i),\cU$, while $X_{\tpluss}(v)$ for $v\notin\Upsilon$ is completely determined by $\cU$). It then follows from Lemma~\ref{lem-support} that
\begin{align}\label{eq-d(t*+s,U)-bound}
  \bar{d}(\tpluss,\cU) &\leq \E\left[\sup_{x_0,y_0}\left\| \P_{x_0}(Z_\chi\in\cdot\mid\cU)-\P_{y_0}(Z_\chi\in\cdot\mid \cU)\right\|_\tv   \;\Big|\;\cU \right] \,.
\end{align}
Conditioning on $\sH_\green$ in the main term in~\eqref{eq-d(t*+s,U)-bound},
then taking expectation,
\begin{align}
\sup_{x_0,y_0}&\left\|\P_{x_0}(Z_\chi\in\cdot\mid\cU)-\P_{y_0}(Z_\chi\in\cdot\mid \cU)\right\|_\tv \nonumber\\
 \leq \sup_{\sH_\green} \sup_{x_0,y_0}&\left\|\P_{x_0}\left(Z'_\chi \in \cdot\mid \sH_\green,\,\cU\right) -\P_{y_0}\left(Z'_\chi \in \cdot\mid \sH_\green,\,\cU\right) \right\|_\tv\label{eq-P(Zchi)-dist}
\end{align}
where $Z'_\chi$ is the joint distribution of $X_{\tau_{k_i}}(\chi_i)$ for $A_i\notin\cA_\green$, i.e., the projection onto cut-sets of blue or red components. (The inequality above replaced the expectation over $\sH_\green$ by a supremum, then used the fact that the values of $\{ X_{\tau_{k_i}}(\chi_i) : A_i\in\cA_\green\}$ are independent of the initial condition, and so taking a projection onto the complement spin-set does not change the total-variation distance.)

Now let $\nu_i$ be the distribution of the spins at the cut-set of $A_i$ when further conditioning that $A_i$ is blue, i.e.,
\[ \nu_i = \left( X_{\tau_{k_i}}(\chi_i)\in\cdot \;\big|\; \sH_\green,\, \cU,\, A_i\in\cA_\blue\right)\,,\]
and further set
\[ \nu^*_i = \min_x \nu_i(x)\,,\quad \nu=\!\!\!\prod_{i:A_i\notin\cA_\green}\!\!\! \nu_i\,.\]
The right-hand side of~\eqref{eq-P(Zchi)-dist} is then clearly at most
\begin{align}
2&\sup_{\sH_\green} \sup_{x_0} \left\|\P_{x_0}\left(Z'_\chi \in \cdot\mid \sH_\green,\,\cU\right) - \nu\right\|_\tv\nonumber \\
\leq 2& \sup_{\sH_\green} \sup_{x_0}\Big[ \left\|\P_{x_0}\left(Z'_\chi \in \cdot\mid \sH_\green,\,\cU\right)  - \nu\right\|_{L^2(\nu)} \;\wedge\;1\Big]\,.\label{eq-P(Zchi)-dist2}
\end{align}

At this point we wish to appeal to the following lemma --- which generalizes~\cite{MP}*{Proposition~3.2}, via the exact same proof, from unbiased coin flips to a general distribution ---
bounding the $L^2$-distance in terms of an exponential moment of the intersection between two i.i.d.\ configurations.
\begin{lemma}\label{lem:MP}
Let $\{\Lambda_i : i\in \cI\}$ be a partition of $\Lambda$, and let $\nu_i$ ($i\in\cI$) be a measure on $\{\pm1\}^{\Lambda_i}$. For each $S\subset \cI$, let $\phi_S$ be a measure on $\{\pm1\}^{\cup_{i\in S}\Lambda_i}$.
Let $\mu$ be a measure on configurations in $\Omega=\{\pm1\}^\Lambda$ obtained by sampling a subset $S\subset \cI$ via some measure $\tilde{\mu}$, then sampling
 $\cup_{i\in S}\Lambda_i$ via $\phi_S$ and setting each $\Lambda_i $ for $i\notin S$ via an independent sample of $\nu_i$. Letting $\nu=\prod_{i\in\cI}\nu_i$,
\[ \|\mu - \nu\|^2_{L^2(\nu)}   \leq
\bigg[\sum _{S,S'} \bigg(\prod_{i\in S\cap S'} \min_{x_i\in \{\pm1\}^{\Lambda_i}}  \nu_i(x_i)\bigg)^{-1} \tilde{\mu}(S)\tilde{\mu}(S')\bigg]-1\,.\]
\end{lemma}
\begin{proof}
For any $S\subset\cI$, let $x_S$ denote the projection of $x$ onto $\cup_{i\in S}\Lambda_i$.  With this notation, by definition of the $L^2(\nu)$ metric (see, e.g.,~\cite{SaloffCoste2}) one has that $\|\mu-\nu\|^2_{L^2(\nu)} + 1 = \int |\mu/\nu - 1|^2 d\nu + 1 $ is equal to
  \begin{align*}
  \sum_{x\in\Omega} \frac{\mu^2(x)}{\nu(x)} 
  &=\sum_{x\in\Omega} \frac{1}{\prod\nu_i(x_i)} \sum_{S,S'}
  \tilde{\mu}(S)\tilde{\mu}(S')
  \phi_S(x_S) \phi_{S'}(x_{S'}) \prod_{i\notin S} \nu_i(x_i)\prod_{i\notin S'}\nu_i(x_i)
  \end{align*}
by the definition of $\mu$. This can in turn be rewritten as
\begin{align*}
  \sum_{S,S'} & \bigg(\sum_{x_{S\cap S'}} \frac{\phi_S(x_{S\cap S'})\phi_{S'}(x_{S\cap S'})}{\prod_{i\in S\cap S'}\nu_i(x_{i})} \bigg)
  \bigg(\sum_{x_{(S \cup S')^c}} \prod_{i\in (S \cup S')^c}\nu_i(x_{i})  \bigg)\\
  \cdot &\bigg(\sum_{x_{S\setminus S'}} \phi_S(x_{S\setminus S'}\mid x_{S\cap S'} ) \bigg)
  \bigg(\sum_{x_{S'\setminus S}} \phi_{S'}(x_{S'\setminus S}\mid x_{S'\cap S} ) \bigg)\tilde{\mu}(S)\tilde{\mu}(S')\\
  =\sum_{S,S'}& \bigg(\sum_{x_{S\cap S'}} \frac{\phi_S(x_{S\cap S'})\phi_{S'}(x_{S\cap S'})}{\prod_{i\in S\cap S'}\nu_i(x_{i})} \bigg)\tilde{\mu}(S)\tilde{\mu}(S')\,,
\end{align*}
which is at most
\begin{align*}
  &\sum_{S,S'} \bigg(\prod_{i\in S\cap S'} \min_{x_i\in \{\pm1\}^{\Lambda_i}}  \nu_i(x_i)\bigg)^{-1}  \bigg(\sum_{x_{S\cap S'}} \phi_S(x_{S\cap S'}) \phi_{S'}(x_{S\cap S'}) \bigg)\tilde{\mu}(S)\tilde{\mu}(S')\\
  \leq&\sum_{S,S'} \bigg(\prod_{i\in S\cap S'} \min_{x_i\in \{\pm1\}^{\Lambda_i}}  \nu_i(x_i)\bigg)^{-1}  \tilde{\mu}(S)\tilde{\mu}(S')\,.\qedhere
  \end{align*}
\end{proof}
Applying the above lemma to the quantity featured in~\eqref{eq-P(Zchi)-dist2} yields
\[ \left\|\P_{x_0}\left(Z'_\chi \in \cdot\mid \sH_\green,\,\cU\right)  - \nu\right\|_{L^2(\nu)}^2 \leq \E\bigg[\prod_{A_i \in \cA_\red\cap \cA_{\red'}}\!\! \frac1{\nu_i^*} \,\Big|\, \sH_\green,\, \cU\bigg]-1\,,\]
where $\cA_\red$ and $\cA_{\red'}$ are two i.i.d.\ samples conditioned on $\sH_\green$ and $\cU$.
Combining the last inequality with~\eqref{eq-d(t*+s,U)-bound},\eqref{eq-P(Zchi)-dist} and~\eqref{eq-P(Zchi)-dist2}, we conclude the following.
\begin{align}\label{eq-d(t*+s,U)-bound2}
  \bar{d}(\tpluss,\cU) &\leq 2\sup_{\sH_\green}\sup_{x_0}
  \Bigg[ \Bigg( \E\bigg[\prod_{A_i \in \cA_\red\cap \cA_{\red'}} \!\frac1{\nu_i^*} \;\Big|\; \sH_\green,\, \cU\bigg]-1\Bigg)^{\frac12}\wedge 1\Bigg]\,.
  \end{align}
Note that the expectation above is only w.r.t.\ the update sequence along the interval $(0,\tcut]$. Indeed, the variables $\cA_\red$ and $\cA_{\red'}$ do not depend on the deferred randomness $\cU'$, which in turn is embodied in the measures $\nu_i$ (and consequently, the values $\nu^*_i$).

The expectation in the right-hand side of~\eqref{eq-d(t*+s,U)-bound2} is treated by the following lemma.
\begin{lemma} \label{lem-nu-bound}
Let $\beta>\beta_c$, let $\cA_{\red}$ and $\cA_{\red'}$ denote the collection of components within red clusters in two independent instances of the dynamics, and define $\Psi_{\cC}$ as in~\eqref{eq-def-psi}. Then
\begin{align*} \E\bigg[\prod_{A_i \in \cA_\red\cap \cA_{\red'}} \!\frac1{\nu_i^*} \;\Big|\; \sH_\green,\, \cU\bigg]
\leq \exp\bigg[\sum_{\cC\cap\cC'\neq\emptyset} \Psi_{\cC}\Psi_{\cC'} \prod_{A_j\in\cC}\frac{1}{\nu_j^*}\prod_{A_{j}\in\cC'}\frac{1}{\nu_{j}^*}\bigg]\,.
\end{align*}
\end{lemma}
One should emphasize the dependence of the bound given by this lemma on $\sH_\green$ and
$\cU$: the dependence on $\sH_\green$ was eliminated thanks to the supremum in the definition of $\Psi_\cC$. On the other hand, both $\Psi_{\cC}$ and $\nu_j^*$
still depend on $\cU$.
\begin{proof}[\textbf{\emph{Proof of Lemma~\ref{lem-nu-bound}}}]
We first claim that, if $\{Y_{\cC}\}$ is a family of independent indicators given by
\[ \P(Y_{\cC}=1) = \Psi_{\cC}\,,\]
then, conditioned on $\sH_\green$, one can couple the distribution of $\cA_\red$ to $\{Y_{\cC}:\cC\subset\cA\setminus\cA_\green\}$ in such a way that
\begin{equation}
  \label{eq-couple-Y-C}
  \{\cC : \cC\in\red\}\subset \{ \cC: Y_{\cC}=1\}\,.
\end{equation}
To see this, order all $\cC\subset \cA\setminus \cA_\green$ arbitrarily as $\{\cC_l\}_{l\geq 1}$ and let $\cF_l$ correspond to the filtration that successively reveals  $\one_{\{\cC_l\in\red\}}$. Then $\P(\cC_l\in\red\mid\cF_{l-1})\leq \Psi_{\cC_l}$ since
\[ \P(\cC_l\in\red\mid\cF_{l-1})\leq \sup_{\sX} \P\left(\cC_l\in\red \mid \cF_{l-1},\,\{\cC_l\in\red\}\cup\{\cC_l\subset\blue\}\,,\sH_{\cC_l}^-=\sX\right)\,,\]
and in the new conditional space the variables $\{\cC_j\in\red\}_{j<l}$ are measurable since
 \begin{compactenum}[(i)]
 \item the event $\{\cC_j\in\red\}$ for any $\cC_j$ disjoint from $\cC_l$ is determined by the conditioning on $\sH_{\cC_l}^-$;
 \item any $\cC_j$ nontrivially intersecting $\cC_l$ is not red under the conditioning on $\{\cC_l\in\red\}\cup\{\cC_l\subset \blue\}$.
 \end{compactenum}
This establishes~\eqref{eq-couple-Y-C}.

Consequently, our next claim is that for a family $\{Y_{\cC,\cC'}\}$ of independent indicators given by
\[ \P(Y_{\cC,\cC'}=1) = \Psi_{\cC}\Psi_{\cC'}\quad\mbox{for any $\cC,\cC'\subset \cA\setminus\cA_\green$}\,,\]
one can couple the conditional distributions of $\cA_{\red}$ and $\cA_{\red'}$ given $\sH_\green$ in such a way that
\begin{equation}
  \label{eq-Ai-C-C'-dom}
 \prod_{A_i \in \cA_\red\cap \cA_{\red'}} \!\frac1{\nu_i^*} \leq
 \prod_{\substack{\cC,\cC'\subset \cA\setminus \cA_\green \\ \cC\cap \cC'\neq \emptyset}} \Big(1 + Y_{\cC,\cC'} \Big(\prod_{A_j\in\cC}\frac{1}{\nu_j^*}\prod_{A_{j}\in\cC'}\frac{1}{\nu_{j}^*}-1\Big)\Big)  \,.
\end{equation}
To see this, take $\{Y_\cC\}$ achieving~\eqref{eq-couple-Y-C} and $\{Y'_\cC\}$ achieving its analog for $\red'$. Letting $\{(\cC_l,\cC'_l)\}_{l\geq 1}$ be an arbitrary ordering of all pairs of potential clusters that intersect ($\cC,\cC'\subset \cA \setminus \cA_\green$ with $\cC\cap \cC'\neq\emptyset $), associate each pair with a variable $R_l$ initially set to 0, then process them sequentially:
 \begin{itemize}
   \item If $(\cC_l,\cC'_l)$ is such that for some $j<l$ we have $R_j=1$ and either $\cC_j\cap \cC_l\neq\emptyset $ or $\cC'_j\cap \cC'_l\neq\emptyset$, then we skip this pair (keeping $R_l=0$).
   \item Otherwise, we set $R_l$ to be the indicator of $\{\cC_l\in\red,\, \cC'_l\in\red'\}$.
 \end{itemize}
 Observe that, if $\cF_l$ denote the natural filtration corresponding to this process, then for all $l$ we have
 \[\P(R_l=1 \mid \cF_{l-1}) \leq \P(Y_{\cC_l}=1,Y_{\cC'_l}=1)= \Psi_{\cC}\Psi_{\cC'}\,,\]
 since testing if $R_l=1$ means that we received only negative information on $\{Y_{\cC_l}=1\}$ and $\{Y_{\cC'_l}=1\}$; this implies the existence of a coupling in which $\{ l : R_l = 1\}\subset \{ l : Y_{\cC_l,\cC'_l}=1\}$.
Hence, if $A_i\in\cA_{\red}\cap\cA_{\red'}$ then $A_i\in\cC_l\cap\cC'_l$ for some $l$ where $Y_{\cC_l}=Y_{\cC'_l}=1$, so either $Y_{\cC_l,\cC'_l}=1$, or else $Y_{\cC_j,\cC'_j}=1$ for a previous pair $(\cC_j,\cC'_j)$ in which $\cC_j=\cC_l$ or $\cC'_j=\cC'_l$ (a nontrivial intersection in either coordinate will not yield a red cluster). Either way, the term $1/\nu_i^*$ is accounted for in the right-hand of~\eqref{eq-Ai-C-C'-dom}, and~\eqref{eq-Ai-C-C'-dom} follows.

Taking expectations in~\eqref{eq-Ai-C-C'-dom} within the conditional space given $\sH_\green,\cU$, and using the definition (and independence) of the $Y_{\cC,\cC'}$'s, we find that
\begin{align*} \E\bigg[\prod_{A_i \in \cA_\red\cap \cA_{\red'}} \!\frac1{\nu_i^*} \;\Big|\; \sH_\green,\, \cU\bigg]
&\leq \prod_{\cC \cap\cC'\neq\emptyset } \bigg(1 + \Psi_{\cC} \Psi_{\cC'}\prod_{A_j\in\cC}\frac{1}{\nu_j^*}\prod_{A_{j}\in\cC'}\frac{1}{\nu_{j}^*}\bigg)\\
&\leq \exp\bigg[\sum_{\cC\cap\cC'\neq\emptyset} \Psi_{\cC}\Psi_{\cC'} \prod_{A_j\in\cC}\frac{1}{\nu_j^*}\prod_{A_{j}\in\cC'}\frac{1}{\nu_{j}^*}\bigg]\,.\qedhere
\end{align*}
\end{proof}
\begin{corollary}\label{cor-dtv-bound}
Let $\beta>\beta_c$. With the above defined $\Psi_\cC$ and $\Xi_j$'s we have
\begin{align}\label{eq-dtv-psi}
\bar{d}(\tpluss) &\leq 4 \Bigg( \E\bigg[\sum_{\cC\cap\cC'\neq\emptyset} \Psi_{\cC}\Psi_{\cC'} \prod_{A_j\in\cC}\frac{1}{\Xi_j}\prod_{A_{j}\in\cC'}\frac{1}{\Xi_{j}}\bigg] \Bigg)^{1/2}\,.
  \end{align}
  \end{corollary}
\begin{proof}[\textbf{\emph{Proof of Corollary~\ref{cor-dtv-bound}}}]
Plugging the bound from Lemma~\ref{lem-nu-bound} into~\eqref{eq-d(t*+s,U)-bound2}, then integrating over the undeferred randomness $\cU$, produces an upper bound on the total-variation distance at time $\tpluss$:
\begin{align*}
  \bar{d}(\tpluss) &\leq 2 \,\E\Bigg[\Bigg(\exp\bigg[\sum_{\cC\cap\cC'\neq\emptyset} \Psi_{\cC}\Psi_{\cC'} \prod_{A_j\in\cC}\frac{1}{\Xi_j}\prod_{A_{j}\in\cC'}\frac{1}{\Xi_{j}}\bigg] -1\Bigg)^{1/2}\wedge 1\Bigg]\,,
  \end{align*}
  where $\E$ denotes expectation w.r.t.\ $\cU$, and we used the observation that $\nu_j^* \geq \Xi_j$ by construction.
  Indeed, $\nu_j^*$ denotes the minimal measure of a configuration of the spins in the cut-set $\chi_j$ of a blue component $A_j$ given $\cU,\sH_\green$ at time $\tau_{k_j}$
  (where $k_j$ is the index of the phase optimizing the choice of the cut-set).
  Clearly, any particular configuration $\eta\in \{\pm 1\}^{\chi_j}$ can occur at time $\tau_{k_j}$ if \emph{every} $x\in\chi_j$ were to receive an \emph{ oblivious deferred update} --- with the appropriate new spin of $\eta_x$ --- before its first splitting point in the deferred phase $\cI'_{k_j}$. Since oblivious updates occur at rate $\theta$, this event has probability at least $\frac12 (1-\exp(-\theta T_x)) \geq \frac14 \theta T_x $ where $T_x$ is the length of the interval between $\tau_{k_j}$ and the first update to $x$ in $\cI'_{k_j}$, and the inequality used $1-e^{-x} \geq x -x^2/2 \geq x/2$ for $x\in[0,1]$ (with $x=\theta T_x \leq 1$). The independence of the deferred updates therefore shows that $\nu_j^* \geq \Xi_j$.

Since $\sqrt{e^x - 1} \leq 2\sqrt{x}$ for $x\in[0,1]$, the inequality $(\sqrt{e^x-1} \wedge 1) \leq 2\sqrt{x}$ holds for all $x\geq 0$; thus, Jensen's inequality allows us to derive~\eqref{eq-dtv-psi} from the last display, as required.
\end{proof}

It now remains to show that the expectation over $\cU$ on the right-hand side of~\eqref{eq-dtv-psi} is at most $\epsilon(\scut)$ for some $\epsilon(\scut)>0$ that is exponentially small in $\scut$, which will be achieved by the following lemma.
\begin{lemma}\label{lem-Psi-Xi-bound}
Let $\beta>\beta_c$. With the above defined $\Psi_\cC$ and $\Xi_j$'s we have
  \[  \E\bigg[\sum_{\cC\cap \cC'\neq\emptyset} \Psi_{\cC}\Psi_{\cC'} \prod_{A_i\in \cC} \frac1{\Xi_i}
  \prod_{A_i\in \cC'} \frac1{\Xi_i}\bigg] \leq e^{-\tfrac15c \scut/\lambda} \,.
\]
\end{lemma}
\begin{proof}
  [\textbf{\emph{Proof of Lemma~\ref{lem-Psi-Xi-bound}}}]
We begin by breaking up the sum over potential clusters $\cC,\cC'$ in the left-hand side of the sought inequality as follows: first, we will root a single component $A \in \cC\cap \cC'$; second, we will enumerate over the partition of these clusters into components: $\cC=\{A_{i_j}\}$ and $\cC' = \{A_{i_k}\}$; finally, we will sum over the
the block-sets $\{B_i\}$ that are the counterparts (via $\cU$) at time $\tpluss$ to the block-sets $\{A_i\}$ at time $\tcut$. Noting that the event $\{A_i\conn B_i\}$ --- testing the consistency of $\{A_i\}$ and $\{B_i\}$ --- is $\cU$-measurable, we have
\begin{align}
  \E\Bigg[\sum_{\cC\cap \cC'\neq\emptyset} \Psi_{\cC}\Psi_{\cC'} &\prod_{A_j\in \cC} \frac1{\Xi_j}
  \prod_{A_j\in \cC'} \frac1{\Xi_j}\Bigg] \nonumber\\
\leq \sum_{A} \sum_{\substack{\cC=\{A_{i_j}\}\ni A \\ \cC'=\{A_{i_k}\}\ni A}} \sum_{\substack{\{B_j\} \\ \{B'_k\}}} \E\bigg[&
\bigg(\prod_{j}\one_{\{A_{i_j}\conn B_j\}} \frac1{\Xi_{i_j}}\bigg)
\Psi_\cC
\bigg(\prod_{k}\one_{\{A_{i_k}\conn B'_k\}} \frac1{\Xi_{i_k}}\bigg)
\Psi_{\cC'}
\bigg]\,.\label{eq-Exp-Psi-Psi'-1}
\end{align}
Recall from~\eqref{eq-def-psi} that Lemma~\ref{l:redConnection} provides us with an upper bound on $\Psi_{\cC}$ in terms of the components $\{A_{i_j}\}$ of $\cC$ and uniformly over $\cU$. Letting $\bar{\Psi}_{\{A_{i_j}\}}$ denote this bound (i.e., the right-hand side of~\eqref{e:redConnection}) for brevity, we can therefore deduce that the expectation in the last display is at most
 \[ \E\Bigg[
\bigg(\prod_{j}\one_{\{A_{i_j}\conn B_j\}} \frac1{\Xi_{i_j}}\bigg)
\bigg(\prod_{k}\one_{\{A_{i_k}\conn B'_k\}} \frac1{\Xi_{i_k}}\bigg)
\Bigg]\bar{\Psi}_{\{A_{i_j}\}}\bar{\Psi}_{\{A_{i_k}\}}\,.\]
H\"older's inequality now implies that the last expectation is at most
\begin{align*}
\Bigg[\P\bigg(\bigcap_{j}\{A_{i_j} \conn B_j\}\bigg)\E\bigg[\prod_{j} \frac1{\Xi_{i_j}^4}\bigg]
\P\bigg(\bigcap_{k}\{A_{i_k} \conn B'_k\}\bigg)
  \E\bigg[\prod_{k} \frac1{\Xi_{i_k}^4}\bigg]\Bigg]^{\frac14}\,,
\end{align*}
and when incorporating the last two steps in~\eqref{eq-Exp-Psi-Psi'-1} it becomes
possible to factorize the terms involving $\cC,\cC'$ and altogether obtain that
\begin{align}\label{e:factorizationBound}
  \E\Bigg[&\sum_{\cC\cap \cC'\neq\emptyset} \Psi_{\cC}\Psi_{\cC'} \prod_{A_i\in \cC} \frac1{\Xi_i}
  \prod_{A_i\in \cC'} \frac1{\Xi_i}\Bigg] \leq \sum_{A} \Bigg[ \sum_{\substack{\cC=\{A_{i_j}\}\\ A\in\cC} } \sum_{\{B_j\}}
  \P\bigg(\bigcap_{j}\{A_{i_j} \conn B_j\}\bigg)^{\frac14}
  \E\bigg[\prod_{j} \frac1{\Xi_{i_j}^4}\bigg]^{\frac14}
  \bar{\Psi}_{\{A_{i_j}\}}\Bigg]^2\,.
\end{align}
The term $\P(\bigcap_{j}\{A_{i_j} \conn B_j\})$ is bounded via Lemma~\ref{l:Ai-to-Bi}.
The term $\E[\prod_{j} \Xi_{i_j}^{-4}]$ is bounded via Lemma~\ref{l:exp-cut-set} using the observation that one can always restrict the choice of phases for the cut-sets (only worsening our bound) to be the same for all the components, whence $\prod_j \Xi_{i_j}$ identifies with a single variable $\Xi$ whose source block-set at time $\tpluss$ is $\cup_j B_j$.
Finally, $\bar{\Psi}_{\{A_{i_j}\}}$ corresponds to the right-hand side of~\eqref{e:redConnection} from Lemma~\ref{l:redConnection}, in which we may decrease the exponent by a factor $\lambda$ (only relaxing the bound as $\lambda>1$).
Altogether, for some $c=c(\beta,d)>0$
(taken as $\frac{c_1}4 \wedge c_2$ where $c_1,c_2$ are the constants from Lemmas~\ref{l:Ai-to-Bi} and~\ref{l:redConnection}, respectively)
the last expression is at most
\begin{align}
\sum_{A} \Bigg( \sum_{\substack{\cC=\{A_{i_j}\}\\ A\in\cC}} \sum_{\{B_j\}}
 \frac{2^{\lambda+3} \scut^{4d}}{\sqrt{|\Lambda|}} &e^{-c  (\scut/\lambda)\left(\left|\anim\left(\cC\right) - \sum\anim(A_{i_j})\right|^+ + \sum \anim(A_{i_j}\cup B_j) \right) + 4\sum |A_{i_j}| +  \frac14 \sum |B_j| }\Bigg)^2\,.
 \label{eq-sum-A-via-anim}
\end{align}
It is easy to see that since $|\anim\left(\cC\right) - \sum_{j}\anim(A_{i_j})|^+ \geq \frac12(\anim\left(\cC\right) - \sum_{j}\anim(A_{i_j}\cup B_j))$, we have that
\[ \sum_j \anim(A_{i_j} \cup B_j) + \Big|\anim\left(\cC\right) - \sum_{j}\anim(A_{i_j})\Big|^+ \geq \frac{1}2  \anim(\cC) + \frac12\sum_j \anim(A_{i_j} \cup B_j)\,.\]
Since each of the summands $|A_{i_j}|$ and $|B_j|$ in the exponent above is readily canceled by $\anim(A_{i_j}\cup B_j)$, we deduce that if $c \scut/\lambda$ is large enough then~\eqref{eq-sum-A-via-anim} is at most
\begin{align}
&\frac{2^{2\lambda+6}\scut^{8d}}{|\Lambda|} \sum_{A} \bigg( \sum_{\substack{\cC=\{A_{i_j}\}\\ A\in\cC}} \sum_{\{B_j\}}
 \exp\bigg[ -\frac{c}4 \frac{\scut}\lambda \anim\Big(\cC \cup \bigcup_j B_j\Big)\bigg]\bigg)^2\,.\label{eq-sum-A-one-animal}
 \end{align}
Now, the number of different lattice animals containing $\kappa$ blocks and rooted at a given block $X$ is easily seen to be at most $(2d)^{2(\kappa-1)}$, since these correspond to trees on $\kappa$ vertices containing a given point in $\Z^d$, and one can enumerate over such trees by traveling along there edges via a depth-first-search: beginning with $2d$ options for the first edge from the root, each additional edge has at most $2d$ options (at most $2d-1$ new vertices plus one edge for backtracking, where backtracking at the root is regarded as terminating the tree). The bound on the number of rooted trees (and hence the number of rooted lattice animals) now follows from the fact that each edge is traversed precisely twice in this manner.

Next, enumerate over collections of blocks $\{A_{i_j},B_j\}$ with $\anim(\cC)=\cS$ and $\sum_j \anim(A_{i_j} \cup B_j) = \cR$:
\begin{compactitem}
\item There are at most $(2d)^{2(\cS-1)}$ ways to choose $\cC$ 
    containing $A$ by the above lattice animal bounds.
\item There are at most $2^{\cS}$ choices of blocks $D_j \in \cC$ so that each $D_j\in A_{i_j}$ will be in a distinct $A_{i_j}$.
\item There are at most $2^{\cR}$ choices of $r_j$ representing $\anim(A_{i_j} \cup B_j)$ since $\sum_j r_j = \cR$.
\item For each $j$ there are at most $(2d)^{2(r_j-1)}$ choices of minimal lattice animals of size $r_j$ rooted at $D_j$ which will contain $A_{i_j} \cup B_j$.  Together, this is at most $(2d)^{\sum_j 2(r-1)} \leq 2^{\cR}$.
\item For each lattice animal there are $2^{r_j}$ ways to assign the vertices to be either in or not in $A_{i_j}$ and $2^{r_j}$ choices to be either in or not in $B_j$.  In total this gives another $4^{\cR}$ choices.
\end{compactitem}
Altogether, we have that the number of choices of the $\{A_{i_j},B_j\}$ is at most $2^{\cS}8^{\cR}(2d)^{2(\cS+\cR)}$.  Thus,
\[
\sum_{\substack{\cC=\{A_{i_j}\}\\ A\in\cC}} \sum_{\{B_j\}}
 \exp\bigg[ -\frac{c}4 \frac{\scut}\lambda  \Big(\anim(\cC) + \sum_j \anim(A_{i_j} \cup B_j)\Big)\bigg] \leq \sum_{\cS,\cR \geq 1} e^{-\frac{c}4 \frac{\scut}\lambda (\cS + \cR)} 2^{\cS}8^{\cR}(2d)^{2(\cS+\cR)} \leq e^{-\frac{c}5 \frac{\scut}\lambda}
\]
provided $\scut$ is large enough compared to $d$.
Plugging this in~\eqref{eq-sum-A-one-animal} finally gives
\[  \E\bigg[\sum_{\cC\cap \cC'\neq\emptyset} \Psi_{\cC}\Psi_{\cC'} \prod_{A_i\in \cC} \frac1{\Xi_i}
  \prod_{A_i\in \cC'} \frac1{\Xi_i}\bigg] \leq \frac{2^{2\lambda+6}\scut^{8d}}{|\Lambda|} \sum_{A} e^{-\tfrac25 c \scut/\lambda} \leq e^{-\tfrac15c \scut/\lambda} \,,
\]
where the last inequality holds whenever, e.g., $\scut\geq \lambda^2$ and $\lambda$ is large enough in terms of $\beta,d$.
\end{proof}
Combining Corollary~\ref{cor-dtv-bound} and Lemma~\ref{lem-Psi-Xi-bound} shows that
$\bar{d}(\tpluss) \leq 4 \exp[ -\tfrac1{10} c\scut/\lambda]$, and so in particular, once we fix $\lambda$ larger than some $\lambda_0(\beta,d)$, the total-variation distance at time $\tpluss$ will decrease in $\scut$ as $O(\exp[- c' \scut])$ for some $c'(\beta,d)>0$, concluding the proof.
\qed

\subsection{Lower bound on the mixing time}
We begin with two simple lemmas, establishing exponential decay for the magnetization in time and for the correlation between spins in $X_t$ in space.
\begin{lemma}\label{l:MagnetizationLB}
There exist $c_1(\beta,d)$ and $c_2(\beta,d)$ such that for all $0<h<t$,
\[
\sm_t \leq c_1 e^{-c_2 h} \sm_{t-h}.
\]
\end{lemma}
\begin{proof}
By Lemma~\ref{l:BasicEstimates},
\begin{align*}
\E\big|\fupd(v,t-h,t)\big|^2 &\leq |B(v,20dh)|^2+ \!\!\sum_{k=20dh}^n \!\! k^{2d}\P\left(\fsup(v,\tcut-h,\tcut)\not\subset B(v,k)\right)\\
&\leq  |B(v,20dh)|^2 +  \sum_{k=20dh}^n k^{2d} e^{-c k} =O(h^{2d})\,.
\end{align*}
Then by Cauchy-Schwarz,
\begin{align*}
\E\big|\fsup(v,t-h,t)\big|&\leq\E\left [\big|\fupd(v,t-h,t)\big|\one_{\{\fsup(v,t-h,t)\neq \emptyset\}}\right] \\ &
 \leq \left(\E\left[\big|\fupd(v,t-h,t)\big|^2\right] \P(\fsup(v,t-h,t)\neq \emptyset)\right)^{1/2} \leq O(h^d e^{-c h /2})\,.
\end{align*}
If $\fsup(v,0,t)\neq \emptyset$ then $\fsup(u,t-h,t)\neq \emptyset$ for some $u\in \fupd(v,t-h,t)$.  Using the translational invariance of the torus (which implies that all vertices have the same magnetization),
\begin{align*}
\sm_t = \P(\fsup(v,0,t)\neq \emptyset) &\leq \E\left[\big|\fupd(v,t-h,t)\big|\right] \P(\fsup(v,t-h,t)\neq \emptyset) \leq  O(h^d e^{-c h /2}) \sm_{t-h}\,,
\end{align*}
as claimed.
\end{proof}

\begin{lemma}\label{l:Correlation}
There exist $c_1(\beta,d),c_2(\beta,d)>0$ so that starting from any initial condition $X_0$,
\[
\Cov(X_{t}(u), X_t(v)) \leq c_1 \exp(-c_2 |u-v|)\quad\mbox{ for any $t>0$ and $u,v\in\Lambda$}\,.
\]
\end{lemma}

\begin{proof}
Let $\cE$ denote the event that the supports of $u$ and $v$  intersect, that is
\[
\cE=\bigg\{\bigcup_{0<t'<t} \left(\fsup(v,t',t) \cap \fsup(v,t',t)\right) =\emptyset  \bigg \}.
\]
Let $X_{t}'$ and $X_{t}''$ be two independent copies of the dynamics.  By exploring the histories of the support we may couple $X_t$ with $X_{t}'$ and $X_{t}''$ so that on the event $\cE$ the history of $v$ in $X_t$ is equal to the history of $v$ in $X_{t}'$ and the history of $u$ in $X_t$ is equal to the history of $u$ in $X_{t}''$.  Hence,
\begin{align*}
\E\left[X_t(v) X_t(u)\right] &= \E\left[ X_t'(v) X_t''(u) + \big(X_t(v) X_t(u) - X_t'(v) X_t''(u)\big)\one_{\cE}\right]\\
&\leq \E\left[ X_t'(v)\right]\E\left[ X_t''(u)\right] + 2\P(\cE)
\end{align*}
and so $\Cov(X_t(v), X_t(u)) \leq 2\P(\cE)$.  Define the event
\[
K_{v,r}=\left\{\fsup\left(v,t-\tfrac{r}{40ed},t\right) = \emptyset,\, \fupd\left(v, t-\tfrac{r}{40ed},t\right)\subset B\big(v,\tfrac{r}{2}\big)\right\}\,.
\]
By Lemma~\ref{l:BasicEstimates},
\[
\P(K_{v,r}) \geq 1 - \exp(-c' \tfrac{r}{40ed}) \,.
\]
If $K_{v,|u-v|}$ and $K_{u,|u-v|}$ both hold then the histories of $u$ and $v$ do not intersect and so
\[
\P(\cE) \leq \P(K_{v,|u-v|}^c \cup K_{u,|u-v|}^c) \leq 2 \exp\left( -c' \tfrac{|u-v|}{40ed}\right)\,.
\]
This completes the proof, as it implies that
\[
\Cov(X_t(v), X_t(u)) \leq 2\P(\cE) \leq 4 \exp\left(-c' \tfrac{|u-v|}{40ed}\right)\,.\qedhere
\]
\end{proof}

We are now ready to prove the lower bound for the mixing time.
To lower bound the total variation distance at time $\tcut-h$ we take the magnetization as a distinguishing statistics.  By Lemma~\ref{l:MagnetizationLB},
\[
\E\bigg[\sum_{v\in\Lambda} X^+_{\tcut-h}(v)\bigg] = |\Lambda|  \sm_{\tcut-h} \geq c_1 e^{c_2 h} |\Lambda|  \sm_{\tcut} = c_1 e^{c_2 h} \sqrt{|\Lambda|}\,,
\]
while Lemma~\ref{l:Correlation} implies that
\begin{align*}
\var\bigg(\sum_{v\in\Lambda} X^+_{\tcut-h}(v)\bigg) &= \sum_{u,v\in\Lambda} \Cov(X^+_t(u), X^+_t(v))
=|\Lambda| \sum_{u\in\Lambda} c_1 e^{-c_2 |u-v|} \leq c' |\Lambda|
\end{align*}
for some $c'=c'(\beta,d)>0$. By Chebyshev's inequality,
\[
\P\bigg( \sum_{v\in\Lambda} X^+_{\tcut-h}(v) >\frac12 |\Lambda|  \sm_{\tcut-h}\bigg) \geq 1- \frac{\var\left(\sum_{v\in\Lambda} X^+_{\tcut-h}(v)\right) }{\frac12\E\left[\sum_{v\in\Lambda} X^+_{\tcut-h}(v)\right]} \geq 1-c' e^{-2 c_2 h}.
\]
Now if $\sigma$ is a configuration drawn from the stationary distribution then $\E[\sum_{v\in\Lambda} \sigma(v)]=0$, and since $X^+_t$ converges in distribution to the stationary distribution,
\[
\var\bigg(\sum_{v\in\Lambda} \sigma(v)\bigg)=\lim_{t\to\infty} \var\bigg(\sum_{v\in\Lambda} X^+_{t}(v)\bigg) \leq   c' |\Lambda|\,.
\]
Hence, by Chebyshev's inequality, the probability that the magnetization is at least $\frac12|\Lambda|\sm_{\tcut-h}$ satisfies
\[
\P\bigg( \sum_{v\in\Lambda} \sigma(v) >\frac12 |\Lambda|  \sm_{\tcut-h}\bigg) \leq  c' e^{-2 c_2 h}.
\]
Thus, considering this as the distinguishing characteristic yields
\begin{align*}
d_{\tv}(X^+_t,\pi) &\geq \P\bigg( \sum_{v\in\Lambda} X^+_{\tcut-h}(v) >\tfrac12 |\Lambda|  \sm_{\tcut-h}\bigg) - \P\bigg( \sum_{v\in\Lambda} \sigma(v) >\tfrac12 |\Lambda|  \sm_{\tcut-h}\bigg) \geq 1 - 2 c' e^{-2 c_2 h}\,,
\end{align*}
concluding the proof of the lower bound.
\qed

\section{Analysis of percolation components and clusters}\label{sec:cluster-analysis}

\subsection{Percolation component structure: Proof of Lemma~\ref{l:Ai-to-Bi}}

To each $v\in\Lambda$, associate the {\em column} $Q_v = B(v,\scut^{3/2})\times (\tcut,\tpluss]$ in the space-time slab $\Z^d \times (\tcut,\tpluss]$.
Recall that the history of vertices gives rises to edges in the above space-time slab as per the description in~\S\ref{sec:lattice-framework}. Namely,
if at time $t$ there is a non-oblivious update at site $x$ we mark up to $2d$ intervals $[(x,t),(y,t)]$ for $x\sim y$, and if a site $x$ is born at time $t'$ and dies at time $t''$ we mark the interval $[(x,t),(x,t'')]$. Given these marked intervals, we say that a column $Q_v$ is \emph{exceptional} if it contains one of the following:
\begin{itemize}[\indent$\bullet$]
  \item {\em spatial crossing}: a path connecting $(x,t)$ to $(y,t')$ for $|x-y|\geq \frac12\scut^{3/2}$ and some $t,t'\in(\tcut,\tpluss]$.
  \item {\em temporal crossing}: a path connecting $(v,\tpluss)$ to $B(v,\scut^{3/2})\times \{\tcut\}$.
\end{itemize}
Eq.~\eqref{e:BasicUpdate} from Lemma~\ref{l:BasicEstimates} tells us that, even if all phases were deferred (i.e., the update support were ignored and vertices would never die) then
the probability of witnessing a spatial crossing of length $\scut^{3/2}$ starting from a given site $x$ during a time interval of $\scut$ is at most
$\exp(-\scut^{3/2})$ provided that $\scut > (20 d)^2 $.
In lieu of such a spatial crossing, the number of points reachable from $(v,\tpluss)$ at time $\tau_k-1 = \tpluss-1$ (marking the transition between the deferred phase $\cI_k'$ and the regular phase $\cI_k$) is $O(\scut^{3d/2})$. By Eq.~\eqref{e:BasicSupport} from that same lemma, there exists some $c_1=c_1(\beta,d)>0$ so that
the probability that the history of a given $u$ would survive the interval $\cI_k$ is at most $2\exp(-c_1\frac{\scut}\lambda)$.
A union bound now shows that, overall, the probability that $Q_v$ is exceptional is $O(\scut^{3d/2} \exp(-c_1 \frac{\scut}\lambda))$, which is at most $\exp(-\frac{c_1}2  \frac{\scut}\lambda)$ if, say, $\scut \geq \lambda^2$ and $\lambda$ is large enough in terms of $\beta,d$.

Consider now the collection of block-set pairs $\{(R_i,S_i)\}$. If $R_i \conn S_i $ on account of some component $\Upsilon_{j_i}$ at times $\tcut$ and $\tpluss$ (i.e., $\Upsilon_{j_i}$ is minimally covered by $S_i$ while $\sH_{\Upsilon_{j_i}}(\tcut)$ is minimally covered by $R_i$) then every block $S\in S_i$ contains some $v\in\Upsilon_{j_i}$ such that $(v,\tpluss)$ is connected by a path (arising from the aforementioned marked intervals) to $(R_i,\tcut)$ and every $R_i$ contains some $w\in\sH_{\Upsilon_{j_i}}(\tcut)$ such that $(w,\tcut)$ is connected to $(S_i,\tpluss)$. Moreover, the set of blocks traversed by these paths necessarily forms a lattice animal (by our definition of the component $\Upsilon_{j_i}$ via the equivalence relation on blocks according to intersecting histories or adjacency at times $\tcut$ or $\tpluss$).
We claim that for any block $X$ in this lattice animal, either $X$ contains some vertex $v$ such that $Q_v$ is exceptional, or one of its $2d$ neighboring blocks does (and belongs to the lattice animal).
Indeed, take $x\in X$ such that there is a path $P$ from some $v\in S_i$ to some $w\in R_i$ going through $x$ (such a path exists by the construction of the lattice animal). If $P$ is contained in $B(x,\frac12 \scut^{3/2}) \times (\tcut,\tpluss]$, and hence also in $B(v,\scut^{3/2})$, then it gives rise to a temporal crossing in $Q_v$ and $v$ belongs either to a neighboring block of $X$ or to $X$ itself. Otherwise, $P$ visits both $x$ and some $y\in\partial B(x,\frac12 \scut^{3/2})$ and in doing so gives rise to a spatial crossing in $Q_x$, as claimed.

It follows that if $(R_i,S_i)$ are the blocks corresponding to the components $\Upsilon_{j_i}$ for all $i$ then there are pairwise disjoint lattice animals, with $m_i \geq \anim(R_i\cup S_i)$ blocks each (recall that $\anim(S)$ is the smallest number of blocks in a lattice animal containing $S$), such that each block either contains some $v$ for which $Q_v$ is exceptional, or it has a neighboring block with such a vertex $v$. Therefore, by going through the blocks in the lattice animals according to an arbitrary ordering, one can find a subset $S$ of at least $\sum m_i / (2d+1)$ blocks, such that each block in $S$ contains a vertex with an exceptional column. Similarly, we can arrive at a subset $S'\subset S$ of size at least $\sum m_i / (2d+1)^2$ such that every pair of blocks in it has distance (in blocks) at least $2$. Since the event that $Q_v$ is exceptional depends only on the updates within $B(v,\scut^{3/2})$, the distances between the blocks in $S'$ ensure that the events of containing such a vertex $v$ are mutually independent. Hence, the probability that a given collection of lattice animals complies with the event $\{R_i\conn S_i\}$ for all $i$ is at most $\exp[-\frac{c_1}2 \frac{\scut}\lambda (2d+1)^{-2} \sum m_i]$, or $\exp(-c_2\frac{\scut}\lambda \sum m_i)$ for $c_2=c_2(\beta,d)$.

Finally, recall from the discussion below~\eqref{eq-sum-A-one-animal} that the number of different lattice animals containing $m$ blocks and rooted at a given block is at most $(2d)^{2(m-1)}$. Combined with the preceding discussion, using $m_i \geq \anim(R_i\cup S_i)$ we find that
\begin{align*} \P(\cap_i\{R_i\conn S_i\}) &\leq \prod_{i} \sum_{m_i \geq \anim(R_i \cup S_i)} (2d)^{2m_i} e^{-c_2 \frac{\scut}\lambda m_i}
\leq \exp\bigg[-\frac{c_2}2  \frac{\scut}\lambda \sum_i \anim(R_i\cup S_i)\bigg]\end{align*}
if for instance $\scut \geq 4\lambda \log(2d)/c_2$, readily guaranteed when $\scut \geq \lambda^2 $ for any $\lambda$ that is sufficiently large in terms of $\beta,d$.
\qed

\subsection{Cut-sets estimates: Proof of Lemma~\ref{l:exp-cut-set}}

Partition the space-time slab $\Lambda \times (\tcut,\tpluss]$ into {\em cubes} of the form $Q \times (t,t+r]$ where $r$ is some large integer to be later specified (its value will depend only on $\beta$ and $d$) and $Q\subset \Z^d$ is a box of side-length $r^2$. We will refer to $Q^+ \times (t,t+r]$ for $Q^+ := B(Q,r^{3/2})$ as the corresponding {\em extended cube}.
Let us first focus on some regular phase $\cI_k$. Similar to the argument from the proof of Lemma~\ref{l:Ai-to-Bi} (yet modified slightly), we will say that a given cube $Q \times (t,t+r] $ is {\em exceptional} if one of the following conditions is met:
\begin{itemize}[\indent$\bullet$]
  \item {\em spatial crossing}: the cube has a path connecting $(x,t')$ to $(y,t'')$ for some $x,y\in Q$ such that $|x-y|\geq r^{3/2}$.
  \item {\em temporal crossing}: the extended cube has a path connecting $(x,t+r)$ to $(y,t)$ for some $x,y\in Q^+$.
\end{itemize}
As before, the probability that a given cube contains a spatial crossing is $O(r^{2d} \exp(-r^{3/2}))$ provided that $r > (20 d)^2 $, by the bound from Eq.~\eqref{e:BasicUpdate}. Similarly, the probability of the aforementioned temporal crossing within the regular phase is $O(r^{2d} \exp(-c_1 r))$ for some $c_1=c_1(\beta,d)>0$ by Eq.~\eqref{e:BasicSupport}.
Combining the two, the probability that a cube is exceptional is at most $\exp(-c_2 r)$ for some $c_2=c_2(\beta,d)>0$ if $r$ is a large enough in terms of $\beta,d$.

Next, break the time interval $\cI_k=(\tau_{k-1},\tau_k-1]$ into length-$r$ subintervals $\cI_{k,1},\ldots,\cI_{k,m}$ (so that $m = \frac{\scut}{(\lambda - 1)r}$) in reverse chronological order, i.e.,
\[ \cI_{k,l}=(\tau_{k}-l r-1,\tau_k-(l-1)r-1]\qquad(l=1,\ldots,m)\,.\]
Further let $\cI_{k,0}=\tau_k-1$, and let $Y_{k,l}$ for $l=0,\ldots,m$ count the number of cubes $Q\times\cI_{k,l}$ (boxes when $l=0$) that, at time $t=\tau_k-(l-1)r-1$ (the end of the subinterval $\cI_{k,l}$), intersect the history of $S$.

Our next goal is now to bound the exponential moments of $Y_{k,m}$, the number of cubes intersecting the history of $S$ at time $\tau_{k-1}$, which will be achieved by the following claim:
\begin{claim}\label{clm:Y-k-m-laplace}
For any $k=1,\ldots,\lambda$, the above defined variables $(Y_{k,l})$ satisfy
\begin{align}\label{eq-laplace-regular}
  \E\left[e^{a Y_{k,m}} \right] \leq 1 + \left(\tfrac34\right)^{m} \exp\left[ (\tfrac23)^{m} a Y_{k,0}\right]\quad\mbox{ for any }0<a<\tfrac13\,.
\end{align}
\end{claim}
\begin{proof}
Throughout the proof of the claim 
we drop the subscript $k$ from the $Y_{k,l}$'s and simply write $(Y_l)$.

If $v\in Q\times \tau_k-(l-1)r-1$ belongs to the history-line, we can trace its origin in the cube $Q\times \cI_{k,l-1}$ and necessarily either that cube is exceptional or one of its $2d$ neighbors is (as otherwise there will not be a path from $v$ making it to time $\tpluss$). Hence, $Y_{l+1} \leq (2d+1)X_{l+1}$, where $X_{l+1}$ counts the number of exceptional cubes in the $k$-th subinterval. Moreover, starting from $Y_{l}$ cubes covering the history, the set of exceptional cubes counted by $X_{l+1}$ is comprised of $Y_{l}$ lattice animals --- each rooted at one of those $Y_l$ cubes. So, if $Y_{l}=a$ for some integer $a$ and we consider lattice animals of sizes $w_1,\ldots,w_a$ (cubes) for each of these, the number of configurations for these lattice animals would be at most $(2d)^{2\sum w_i}$ as was noted in the proof of Lemma~\ref{l:Ai-to-Bi}. Out of these, we can always extract a subset of $(\sum w_i)/(2d+1)$ cubes which are pairwise non-adjacent, whereby the events of being exceptional are mutually independent.

Combining these ingredients, and setting $\delta = \frac12 c_2 (2d+1)^{-2}$, if $\cF_{k,l}$ is the $\sigma$-algebra generated by the updates in the subintervals $\cI_{k,l'}$ for $l'\leq l$ then
\begin{align}
  \E &\left[ e^{\delta r  Y_{l+1}}\mid\cF_{k,l}\right]  \leq
  \E \left[ e^{(2d+1)\delta r X_{l+1}}\mid\cF_{k,l}\right] \nonumber\\
  &\leq \sum_{w_1,\ldots, w_{Y_{l}}} \exp\left[\left[(2d+1)\delta r + 2\log(2d) - c_2 (2d+1)^{-1} r \right]\sum w_i\right]\nonumber\\
    &= \bigg[ \sum_{w} e^{-\left[(2d+1)\delta r - 2\log(2d)\right] w}\bigg]^{Y_{l}} \leq \bigg[ \sum_{w} e^{-2d \delta r w}\bigg]^{Y_{l}}
    \leq \exp\left[e^{-\delta r} Y_{l} \right]\,,\label{eq-laplace-large}
\end{align}
where the last two inequalities hold provided that $\delta r$ is sufficiently large, i.e., when $r$ a large enough function of $\beta$ and $d$.
In particular, by Markov's inequality this implies that for any $y>0$,
\begin{align*}
  \P \left( Y_{l+1} \geq y \mid\cF_{k,l}\right) \leq \exp\left[ e^{-\delta r} Y_{l} - \delta r y\right] \leq \exp\left[ \tfrac1{4} Y_{l} - y\right]\,,
\end{align*}
provided that $\delta r$ is large. This enables us to complement the bound in~\eqref{eq-laplace-large} when taking a small factor instead of $\delta r$; namely,
for any $0<a< 1$ we have
\begin{align*}
  \E \left[ e^{a Y_{l+1}}\mid\cF_{k,l}\right]  &= \int_0^\infty \P\left( e^{a Y_{l+1}}\geq t \mid\cF_{k,l}\right)
  \leq \int_0^\infty \bigg(1 \wedge \frac{\exp\left[\frac{1}{4}Y_l\right]}{t^{1/a}}\bigg)dt \\
  &= e^{\frac{a}{4} Y_l } + e^{\frac14 Y_l} \int_{e^{\frac{a}{4} Y_l}}^{\infty} t^{-\frac{1}a} dt
  = e^{\frac{a}4 Y_l} + \frac{a}{1-a} e^{\frac{a}4 Y_l} = \frac{e^{\frac{a}4 Y_l}}{1-a}\,.
\end{align*}
When $Y_l \geq 6$ we can upper bound the last exponent by $\exp(-\frac32 a + \frac12 a Y_l)$ and get
 that for any $0<a<\frac13$,
\begin{align*}
  \E\left[e^{a Y_{l+1}} - 1 \mid \cF_{k,l}\,,\,Y_l\geq 6\right] \leq
\frac{e^{-\frac32 a}}{1-a}e^{\frac{a}2 Y_{l}} - 1 \leq \frac34 \left( e^{\frac23 a Y_{l}} - 1 \right)\,,
\end{align*}
where the last inequality used $1- a \geq \exp(-\frac{a}{1-a})\geq \exp(-\frac32 a)$ for $0< a< \frac13$, followed by the fact that $e^{x/\alpha }-1 \leq \alpha(e^{2x}-1)$ for any $0<\alpha \leq 1$ and $x\geq 0$ thanks
to Jensen's inequality. On the other hand, if $Y_l \leq 6$ then
again by Jensen's inequality (now taking $\alpha = a/\delta r$) and Eq.~\eqref{eq-laplace-large},
\begin{align*}
  \E\left[e^{a Y_{l+1}} - 1 \mid \cF_{k,l}\,,\,Y_l \leq 6\right] \leq
\frac{a}{\delta r}\left( e^{e^{-\delta r} Y_{l}} - 1\right) \leq \frac34 \left( e^{\frac23 a Y_{l}} - 1 \right)\,,
\end{align*}
with the last inequality justified since $\exp(e^{-\delta r}Y_l) \leq \exp(6e^{-\delta r})\leq 2$ for large $\delta r$,
so its left-hand side is at most $a/\delta r \leq a/3$ (again for large $\delta r$), while using $Y_l \geq 1$ in its right-hand side (when $Y_l=0$ both sides are 0)
shows it is always at least $a/2$.

We have thus established the above relation for all values of $Y_l$; iterating it through the $m$ subintervals of $\cI_k$ yields~\eqref{eq-laplace-regular}, as required.
\end{proof}

Moving our attention to the deferred phase $\cI'_k$, here we would like to stochastically dominate the number of vertices in the history at any given time by a rescaled pure birth process $Z_{k,t}$ along a unit interval, where each particle adds $2d$ new ones at rate 1 (recall that by definition particles do not die in deferred phases, and their splitting rate is $1+\tanh(-2d\beta) < 1$) and furthermore, every vertex receives an \emph{extra update} at time $\tau_k$.
Indeed, these can only increase the size of the history at $\tau_k-1$, which in turn can only increase the quantity $\exp(\sum_i 4\Xi_i)$
(by introducing additional cut-vertices in deferred phases further down the history) that we ultimately wish to bound.

Overestimating the splitting rate suffices for our purposes and simplifies the exposition. On the other hand, introducing the extra update at time $\tau_k$ plays a much more significant role:
Let $M_k$ denote the number of vertices in the history at the beginning of each phase $\cI'_k$. By the discussion above, the variables $M_k$ in our process dominate those in the original dynamics, and so $(\Xi_{k}) \succcurlyeq ({\Xi}^+_{k})$ jointly, where
\begin{equation}\label{eq-def-vartheta+} {\Xi}^+_{k} = \prod_{v\in M_k}\left(\tfrac14 \theta T_{v,k}\right)\quad\mbox{ for }
\quad T_{v,k}\sim ({\rm Exp}(1)\wedge 1)\end{equation}
is the analog of $\Xi_{k}$ in the modified process (the variable $T_{v,k}$
corresponding to what would be the update time to $v\in M_k$ nearest to $\tau_k$ in $\cI'_k$ in lieu of the extra update at time $\tau_k$).
Crucially, thanks to the extra updates, $\Xi^+_{k}$ depends only on $M_k$ and has no effect on the history going further back (and in particular on the $M_j$'s for $j<k$).
Therefore, we will (ultimately) condition on the values of all the $M_k$'s, and thereafter the variables $({\Xi}^+_{k})$ will be readily estimated, being conditionally independent.

Indexing the time $0\leq t\leq 1$ of the process $Z_{k,t}$ in reverse chronological order along $\cI'_k$ (identifying $t=0$ and $t=1$ with $\tau_k,\tau_k-1$, resp.), the exponential moments of $Z_{k,1}$ can be estimated as follows.
\begin{claim}
  \label{clm:Z-k-t-laplace}
  For any $k=1,\ldots,\lambda$, the above defined variables $(Z_{k,t})$ satisfy
\begin{align}\label{eq-laplace-Z}
 \E \left[e^{a_1 Z_{k,1}}\mid Z_{k,0} \right] \leq \exp\left[2 e^{2d} a_1 Z_{k,0}\right]~\mbox{ for any $0 < a_1 \leq \frac1{2d}\log\big(\frac1{1-e^{-3d}}\big)$}\,.
\end{align}
\end{claim}
\begin{proof}
Throughout the proof of the claim, put $Z_{t}$ as short for $Z_{k,t}$ for brevity.

One easily sees that for any $\alpha>0$,
\[ \frac{d}{dt}\E \left[e^{\alpha(t)Z_t}\right] = \E \left[\big[\alpha'(t)Z_t +(e^{2d \alpha(t)}-1)Z_t\big]e^{\alpha(t)Z_t}\right]\]
since $\frac{d}{dt} \E[ e^{\alpha Z_t}\mid Z_t] = \lim_{h\to0} e^{\alpha Z_t}(e^{2d\alpha}-1)\P(Z_{t+h}\neq Z_t\mid Z_t)$ for fixed $\alpha$.
Taking $\alpha(t)$ to be the solution to $\alpha'(t) + \exp[2d\alpha(t)] - 1 = 0$, namely
\[ \alpha(t) = \frac{1}{2d} \log\left(\frac1{1-\zeta e^{-2d t}}\right)\qquad\mbox{for $\zeta>0$}\,,\]
we find that $\exp[\alpha(t)Z_t]$ is a martingale and in particular
\begin{equation}
  \label{eq-alpha}
  \E \left[e^{\alpha(1)Z_1}\mid Z_0 \right] = e^{\alpha(0)Z_0}\,.
\end{equation}
Therefore, if we set
\[  \zeta=e^{2d} \left(1-e^{-2d a_1}\right) \quad\mbox{ for } \quad0 < a_1 \leq \frac1{2d}\log\left(\frac1{1-e^{-3d}}\right)
\]
then $0<\zeta \leq e^{-d}$ and so $\alpha(t) $ is real and decreasing along $[0,1]$ to $\alpha(1)=a_1$.
For this choice of parameters we obtain that
\[ \alpha(0) \leq \frac1{2d}\log\left(\frac{1}{1-\zeta}\right) \leq \frac1{2d} \frac{e^{2d}(1-e^{-2d a_1})}{1-e^{-d}} \leq 2 e^{2d} a_1\,,
\]
using that $1/(1-e^{-d}) \leq 2$ for any $d\geq 1$ and $1-x\leq e^{-x}$ for $x>0$. Overall, for any small enough $a_1$ in terms of $d$ (as in the condition above, matching the one in~\eqref{eq-laplace-Z})
we have by~\eqref{eq-alpha} that $ \E [e^{a_1 Z_1}\mid Z_0 ] \leq \exp[2 e^{2d} a_1 Z_0]$.
\end{proof}

Going through the regular phase will enable us to apply Claim~\ref{clm:Z-k-t-laplace} with a value of $a_1$ which is exponentially small in $\scut$, let alone small enough in terms of $d$, easily satisfying the upper bound of roughly
$e^{-3d}/2d$ from the condition in~\eqref{eq-laplace-Z}.

Putting together the analysis of the deferred and regular phases $\cI'_k,\cI_k$ in the last two claims, we can establish a recursion for $M_k$, the number of vertices in $\sH_{S}(\tau_k)$. Using $Y_{k,0}\leq Z_{k,1}$ while
$M_k \leq r^{2d} Y_{k,m}$ (by crudely taking the entire volume of each of the cubes that survived to that point), and recalling~\eqref{eq-laplace-regular}, gives
\begin{align*}
  \E\left[ e^{a M_{k}}\mid \cF_{k+1}\right] &\leq 1 + \left(\tfrac34\right)^m \left(\exp\left[2(er)^{2d}\left(\tfrac23\right)^m a M_{k+1}\right]-1\right)
\end{align*}
as long as $a < \frac13 r^{-2d}$ (to have $a'=r^{2d}a$ qualify for an application of~\eqref{eq-laplace-regular}).
Setting
\begin{align}\label{eq-a-choice}
  \hat{a} = \tfrac14 r^{-2d}
\end{align}
and seeing as for large enough $\scut$ (and therefore large enough $m$) compared to $r$ and $d$, the pre-factor of $M_{k+1}$ is at most $(\frac34)^m$, we finally arrive at
\begin{align}\label{eq-Mk-recursion}
\begin{array}{rcl}  \E\left[ e^{\hat{a} M_{k}}\mid \cF_{k+1}\right] &\leq & 1 + \left(\tfrac34\right)^m \left(\exp\left[\left(\tfrac34\right)^m \hat{a} M_{k+1}\right]-1\right) \,,\\
\noalign{\medskip}
    M_\lambda &\leq& \scut^{2d} |S|\,.
    \end{array}
\end{align}

We will now utilize~\eqref{eq-Mk-recursion} for a bound on the probability that the median of the $M_k$'s exceeds a given integer $b\geq 0$.
More precisely, consider the event that the median of $\{M_0,M_1,\ldots,M_{\lambda-2}\}$,
 which we denote as $\med_{k<\lambda-1}M_k$, exceeds $b$ (it will suffice for our purpose
to consider this event --- which excludes $M_{\lambda-1}$ before taking the median --- and it is convenient to do so since $M_{\lambda}$ was pre-given as input, and hence $\cI_{\lambda}$ is exceptional compared to any other $\cI_k$, where we have better control over $M_k$).
To this end, notice that if $\max\{ M_k : k < \lambda\} \leq  \lambda b$ then the event $\{ \med_{k<\lambda-1}M_k > b\}$ necessitates at least $(\lambda-1)/2$ values of $1\leq k\leq \lambda-1$ for which $M_{k-1}\geq b$ even though $M_k \leq \lambda b$. Therefore,
\begin{align}
  \P\left(\underset{k<\lambda-1}{\med}M_k > b\right) &\leq \P\left(\max_{k < \lambda} \!M_k \geq \lambda b\right)
  + 2^{\lambda}\bigg[\sup_{k < \lambda} \P\left(M_{k-1}> b\mid M_k \leq \lambda b\right)\bigg] ^{\frac{\lambda-1}2}\,.\label{eq-med-1}
\end{align}
The first term in the right-hand side of~\eqref{eq-med-1} can be estimated via~\eqref{eq-Mk-recursion}:
\begin{align*}
\P\left(\max_{k <\lambda} \!M_k \geq \lambda b\right) &\leq \sum_{k<\lambda} \P\left(e^{\hat{a} M_k} \geq e^{\hat{a}\lambda b}\right) \leq \lambda \exp\left[- \lambda \hat{a} b
+ (\tfrac34)^m \hat{a} M_\lambda \right]
\nonumber \\
&\leq \lambda \exp\bigg[-\lambda \hat{a} b + (\tfrac34)^m (\tfrac{\scut}r)^{2d} |S| \bigg]\,.
\end{align*}
Similarly, for the second term, we get from~\eqref{eq-Mk-recursion} that for any $k<\lambda$,
\begin{align*}
 \P\left(M_{k-1}> b\mid M_k \leq \lambda b\right) &\leq \exp\left[-\hat{a} b + (\tfrac34)^m \hat{a} \lambda b\right]
 \leq\exp\left[-\tfrac12 \hat{a} b \right]
\end{align*}
provided that $\scut$ (and hence $m$) is large enough in terms of $\lambda$ (so $(\frac34)^m\lambda < \frac12$).
Plugging these two inequalities in~\eqref{eq-med-1}, while using that $(\lambda-1)/2 > \lambda/3$ for $\lambda$ large
and $(\tfrac34)^m \scut^{2d} \leq 1$ for $\scut$ large enough in terms of $r$ and $\lambda$, yields
\begin{align}
  \label{eq-med-2}
  \P\left(\underset{k<\lambda-1}{\med} M_k > b\right) &\leq 2^{\lambda+1} \exp\bigg[-\tfrac16 \lambda\hat{a} b   + |S|\bigg]\,.
\end{align}

The final step is to derive the desired upper bound on $\min_k\{ (\Xi^+_k)^{-4}\}$ from the estimate~\eqref{eq-med-2}
on the median of the $M_k$'s.  Write
\begin{align}\label{eq-E-vartheta+}
\E\Big[ \min_{k}&\left\{(\Xi^+_k)^{-4}\right\} \;\big|\; \{M_k\}\,,\, \underset{k<\lambda-1}{\med} M_k = b \Big]
= \int dt\, \P\bigg(\min_k \Big\{ (\Xi^+_{k})^{-\frac12}\Big\} \geq t^{\frac18} \,\Big|\, \{M_k\},\,\underset{k<\lambda-1}{\med} M_{k}= b\bigg)\,,
\end{align}
consider some $t>1$ and $b\geq 0$ and condition on the event $\med_{k<\lambda-1} M_k = b$. Revisiting~\eqref{eq-def-vartheta+}, there are at least $\frac{\lambda-1}2$ values of $k\in\{0,\ldots, \lambda-1\}$ such that
\[ \Xi^+_{k} \succcurlyeq \prod_{j=1}^b \left(\tfrac14 \theta T_{v_{j},k}\right)\quad\mbox{ for i.i.d.\ } T_{v_{j},k}\sim ({\rm Exp}(1)\wedge 1)\,,\]
whence, by the independence of the $T_{v_j,k}$'s, if $T \sim {\rm Exp}(1)\wedge 1$ then
\[ \E \left[(\Xi^+_{k})^{-\frac12}\right]\leq \bigg( \E \bigg[\frac{2}{\sqrt{ \theta T}}\bigg]\bigg)^b\,.
\]
The expectation above (involving a single $T$) is easily seen to be equal to
\[\int dx \P\bigg( T < \frac{4}{\theta x^2}\bigg) = O\bigg(\int \frac{dx}{\theta x^2 }\bigg) < C_{\beta,d}\,,\]
for some $C_{\beta,d}>1$ depending only on $\theta$. Hence, by Markov's inequality, under the above conditioning we have
\[ \P\left( (\Xi^+_{k})^{-\frac12} \geq t^{1/8} \right) \leq C_{\beta,d}^b t^{-1/8} \,,
\]
and already the first 10 (say) out of these $\frac{\lambda-1}2$ values of $k$ show that
\[ \P\left( \min_{k<\lambda-1} \left\{(\Xi^+_{k})^{-\frac12}\right\} \geq t^{1/8}\right) \leq (C_{\beta,d})^{10b} t^{-5/4}\,. \]
(Here we could replace 10 by any integer larger than 8, and it is convenient to use an absolute constant rather than a function of $\lambda$
so as to keep the effect of the constant $C_{\beta,d}$ under control). Using~\eqref{eq-E-vartheta+} we find that
\[ \E\left[ \min_{k<\lambda-1} \left\{(\Xi^+_k)^{-4}\right\} \;\big|\; \{M_k\},\,\underset{k<\lambda-1}{\med} M_k = b \right] \leq \int dt \frac{(C_{\beta,d})^{10b}}{t^{5/4}} = 4 (C_{\beta,d})^{10b}\,,
\]
and an integration with respect to $\P(\med_{k<\lambda-1} M_k =b)$ via~\eqref{eq-med-2} establishes that
as long as, say, $\lambda\hat{a} > 100 \log C_{\beta,d}$, we have
\[ \E\left[ \min_{k<\lambda-1} \left\{(\Xi^+_k)^{-4}\right\} \right]\leq 2^{\lambda+3} e^{|S|}\,.\]

In summary, the required result holds for a choice of $r=\lambda^{3d}$ provided that $\lambda$ is large enough in terms of $d,\beta$ (so $r$ is large enough in terms of these as well, while (as we recall that $\hat{a}=\frac14 r^{-2d}$) in addition $\lambda \hat{a} = \frac14 r^d$ is large) and that $\scut$ is then large enough in terms of $\lambda$ (so $m \geq \scut/(r \lambda)$ is large). For these choice, we may take, e.g., $\scut \geq \lambda^{10d}$, whence $m \geq s \lambda^{-(3d+1)} \geq \sqrt{s}$ and all the requirements  above are met for $\lambda$ large enough in terms of $d,\beta$.
\qed

\subsection{Blue percolation clusters given the history of their exterior}

In this section we prove the following \emph{lower bound} on the probability of a cluster to be \blue\ given the update sequence $\cU$ along the $(\tcut,\tpluss]$ and the complete history up to time $\tcut$ of every vertex in its exterior:
\begin{lemma}\label{l:BS}
There exists $s_0(\beta,d)>0$ so that for any $\scut>s_0$, any sufficiently large $n$ and any $\cC\subset\cA$,
\[
\inf_{\sH^{-}_\cC}\P\bigg(\bigcap_{A_i \in \cC} \{A_i\in \blue\}\;\Big|\; \sH^{-}_\cC,\,\cU\bigg) \geq e^{-\sum_{A_i\in\cC}|A_i|}\,,
\]
where the infimum is over all $\cC$-compatible histories.
\end{lemma}
\begin{proof}
Since $\sH^{-}_\cC$ is $\cC$-compatible, the histories of all $A\in \cA\setminus \cC$ do not enter $B(\cC,\scut^2/3)$ before time $\scut$.  Therefore, it is enough to verify for all $A_i\in\cC$ that  $\cup_{\tcut-\scut < t < \tcut}\fsup(A_i,t,\tcut) \subset B(A_i, \scut^2/3)$ and that $\fsup(A_i,\tcut-\scut,\tcut)=\emptyset$. Since these  events depend on disjoint updates and do not depend on $\cU$,
\begin{align}\label{e:BSseparateA}
&\inf_{\sH^{-}_\cC}\P\bigg(\bigcap_{A\in\cC} \{A_i\in \blue\}\mid \sH^{-}_\cC\bigg) \nonumber
&= \prod_{A_i\in\cC} \P\bigg(\fsup(A_i,\tcut-\scut,\tcut)=\emptyset~,~ \bigcup_{\tcut-\scut < t < \tcut}\fsup(A_i,t,\tcut) \subset B(A_i, \scut^2/3)\bigg)\,,
\end{align}
and so we will treat the $A_i$'s separately.  For any $A_i$ we cover $B(A_i, \scut^2/3)$ with a set of tiles as follows.
Let $0=r_0<r_1<\ldots<r_\ell=n$ be such that $r_k-r_{k-1}\in \{\scut^{4d},2 \scut^{4d}\}$.  For each $u\in [\scut^{4d}]^d$ and $k\in[\ell]^d$ denote
\[
V_{k,u} := \{u_1+r_{k_1-1}+1,\ldots,u_1 + r_{k_1}\}\times\ldots\times\{u_d+r_{k_d-1}+1,\ldots,u_d + r_{k_d}\}
\]
where we embed $\{u_j+r_{k_j-1}+1,\ldots,u_j + r_{k_j}\}$ into $\{1,\ldots,n\}$ modulo $n$.  Let $\partial V_{k,u}$ denote the interior boundary of $V_{k,u}$, that is the subset of vertices of $V_{k,u}$ adjacent to a vertex in its complement.  Then by construction
\[
\frac1{\left|[\scut^{4d}]^d\right|}\sum_{u\in [\scut^{4d}]^d} \sum_{k\in [\ell]^d} |\partial V_{k,u} \cap B(A_i,\scut^2/3)| \leq \frac{2d |B(A_i,\scut^2/3)|}{\scut^{4d}} \,,
\]
since in each vertex $v$ and each coordinate $i$ there are at most two choices of $u_i$  for which $v$ will be on the boundary of a block in coordinate $i$.  Hence, it is possible for us to choose some $u\in[\scut^{4d}]^d$ such that $\sum_{k\in [\ell]^d} |\partial V_{k,u} \cap B(A_i,\scut^2/3)| \leq \frac{2d |B(A_i,\scut^2/3)|}{\scut^{4d}}$. Let $\cV$ denote the set of tiles $ V_{k,u}$ such that $ V_{k,u} \cap B(A_i,\scut^2/3) \neq \emptyset$.  Each block of $A_i$ is in at most $2^d$ tiles, so $|\cV| \leq 2^d |A_i|$.

For each $V_k \in \cV$, let $\tilde{V}_k$ denote an isomorphic copy of the graph induced by $V_k$ disconnected from everything else together with a graph bijection $\phi_k:\tilde{V}_k \to V_k$.  Let $\tilde{\Lambda}_i = \cup_{V_k \in \cV} \tilde{V}_k$ and let $\tilde{X}_t$ denote the Glauber dynamics on $\tilde{\Lambda}_i$ started from the all-plus configuration at time $\tcut - \scut$ and run until time $\tcut$. Since the $\tilde{V}_k$ are disconnected, the projections of the chain onto each $\tilde{V}_k$ are independent.  We define the update and support functions $\tfupd$ and $\tfsup$ analogously.  Let $\tilde{\cE}_k$ denote the event that for all $v\in \tilde{V}_k$ the following hold.
\begin{enumerate}
\item The support function dies out by time $\scut$,  $\tfsup(v,\tcut-\scut,\tcut)=\emptyset$.
\item The update function does not travel too far,
\[
\tfupd\left(v,\tcut-\scut,\tcut\right) \subset \phi_k^{-1}\left(V_k\cap B(v, \scut^2/4)\right)\,.
\]
\item All vertices have at most $10 \scut$ updates in the interval $[\tcut-\scut,\tcut]$.
\end{enumerate}
By Lemma~\ref{l:BasicEstimates} and the fact that the number of updates of a vertex in time $\scut$ is $\Po(\scut)$,
\[
\P(\tilde{\cE}_k) \geq 1-|\tilde{V}_k| C e^{-c \scut} \geq \exp(2^{-d-1})\,,
\]
for large enough $\scut$.

Recall that we encode the dynamics $X_t$ by a series of updates $(J_i,U_i,T_i)$ for vertices $J_i\in \Lambda$, unit variables $U_i$ and times $T_i$.  If $S_i$ is the sum of spins of the neighbors of $J_i$ at time $T_i$, then the update sets the new spin of $J_i$ to  $-1$ if $U_i < \frac{e^{-S_i\beta}}{e^{-S_i\beta}+e^{S_i\beta}}$ and to $+1$ otherwise.
We couple the updates of $\tilde{X}_t$ to those of $X_t$ as follows.  For $v\in \tilde{V}_k$ such that $\phi_k(v)\in B(A_i,\scut^2/3)$, we couple the update times, i.e., $v$ has an update at time $t\in[\tcut-\scut,\tcut]$ in $\tilde{X}$ if and only if $\psi_k(v)$ has one in $X$. Furthermore, if in addition $\phi_k(v)\not\in \partial V_k$ then we also couple the unit variable of the update. Otherwise (the case $\phi_k(v)\in \partial V_k$), the unit variables of the updates are taken as independent.

Further recall that an update is oblivious if either $U_i\in[0,\frac{e^{-2d\beta}}{e^{-2d\beta}+e^{2d\beta}}]$ (the new spin is $-1$ irrespective of the neighbors of $J_i$) or $U_i\in[\frac{e^{2d\beta}}{e^{-2d\beta}+e^{2d\beta}},1]$ (similarly, the new spin is $+1$). Let $\cR_k$ denote the event that all updates of $\phi_k(v)\in \partial V_k \cap B(A_i,\scut^2/3)$ are oblivious updates and that the updated values $\tilde{X}_{t}(v)$ and $X_{t}(\phi(v))$ agree.  This has probability $\frac{e^{-2d\beta}}{e^{-2d\beta}+e^{2d\beta}}$ for each update.  Since on $\tilde{\cE}_k$ there are at most $10 \scut |\partial V_k \cap B(A_i,\scut^2/3)|$ updates on $\partial V_k \cap B(A_i,\scut^2/3)$, we have that
\[
\P(\cR_k \mid\tilde{\cE}_k) \geq \exp\left[-C_1 \scut \left|\partial V_k \cap B\big(A_i,\scut^2/3\big)\right|\right]\,,
\]
where $C_1 = 10\log \frac{e^{-2d\beta}+e^{2d\beta}}{e^{-2d\beta}}$.  Since these are independent for each $k$,
\begin{align*}
\P(\cap_{k\in \cV} (\cR_k\cap \tilde{\cE}_k)) &\geq \exp\left[-2^{-d-1}|\cV|-C_1 \scut \left|\partial V_k \cap B\big(A_i,\scut^2/3\big)\right|\right]\nonumber\\
&\geq \exp\left(-\frac12|A_i|-C_1 \scut \frac{2d \left|B\left(A_i,\scut^2/3\right)\right|}{\scut^{4d}}\right)
\geq \exp(-|A_i|)\,,
\end{align*}
provided that $\scut$ is sufficiently large, as $|B(A_i,\scut^2/3)|\leq \left(\frac53\right)^d \scut^{2d}|A_i|$ and $2^d |\cV| \leq |A_i|$.  By Eq.~\eqref{e:BSseparateA}, to complete the lemma it therefore suffices to show that the event $\cap_{k\in \cV} (\cR_k\cap \tilde{\cE}_k)$  implies
\begin{align}
\fsup\left(A_i,\tcut-\scut,\tcut\right)&=\emptyset\,,\label{e:BSDiesOut}\\
\bigcup_{\tcut-\scut < t < \tcut}\!\!\fsup\left(A_i,t,\tcut\right) &\subset B\left(A_i, \scut^2/3\right)\label{e:BSShortRange}\,.
\end{align}
The updates on $\partial V_k\cap B(A_i, \scut^2/3)$ are oblivious updates and hence do not examine the values of their neighbors on the event $\cR_k$.  Combining this with property (2) of the definition of $\tilde{\cE}_k$ and the construction of the coupling implies that for $v\in \tilde{V}_k$ such that $\phi_k(v)\in A_i$, the support of $\phi_k(v)$ is contained in $V_k$.  Hence, by the coupling it follows that
\[
\fsup\left(\phi_k(v),t,\tcut\right)\subset \phi_k\left(\tfupd\left(v,t,\tcut\right)\right) \subset V_k\cap B(\phi_k(v), \scut^2/3)\,,
\]
which implies~\eqref{e:BSShortRange}.  It remains to prove~\eqref{e:BSDiesOut}.

Knowing the updates of course allows one to determine the configuration at a later time from the configuration of an earlier time.
We then define $\tilde{Y}^\eta_{t}(w)$ as follows.  It is the the spin at time $t\in[\tcut-\scut,\tcut]$ of the vertex $w\in  \tfupd(v,t,\tcut)$ generated from the Glauber dynamics with initial configuration $\eta$ on $\tfupd(v,\tcut-\scut,\tcut)$ at time $\tcut-\scut$  using the updates of $\tilde{X}_t$.
Note that, by the definition of $\tfupd$, these are the only initial  values that need to be specified.  Define $Y^\eta_{t}$ in the same way except with the updates of $X_t$ instead of $\tilde{X}_t$, where we take the domain of $\eta$  to be $\phi_k(\tfsup(v,\tcut-\scut,\tcut))$.  As usual, $+$ and $-$ denote the all $+1$ and $-1$ initial conditions, respectively.

Since the initial condition for $\tilde{X}_t$ is all-plus, by the construction of the coupling for every time $t\in[t,\tcut]$ and vertex $w\in  \tfupd(v,t,\tcut)$ we have that
\[
\tilde{X}_{t}(w) = \tilde{Y}^+_{t}(w) = Y^+_{t}(\phi_k(w))\,.
\]
We claim that for all $t$ and $w\in  \tfupd(v,t,\tcut)$, $\tilde{Y}^-_{t}(w) \leq Y^-_{t}(\phi_k(w))$.  This can be seen by induction applying the updates in turn.  Let $\{(t_i,w_i)\}$ denote the set of updates in the update history of $v$ in the interval $[\tcut-\scut,\tcut]$ ordered so that  $\tcut-\scut< t_1<t_2<\ldots<t_q<\tcut$.
For all updates with $w_i \in \tilde{V}_k\setminus \partial  \tilde{V}_k$ this follows by the fact that the updates use the same unit variables, monotonicity of the update rule and the inductive assumption on the values of the neighbors.  For updates $w_i\in\partial\tilde{V}_k$ note that
\[
\tilde{Y}^-_{t_i}(w_i)\leq \tilde{Y}^+_{t_i}(w) = Y^+_{t_i}(\phi_k(w_i)) = Y^-_{t_i}(\phi_k(w_i))\,,
\]
where the first inequality is by monotonicity while the final equality is by the fact that the boundary updates are oblivious ones.  Hence, by induction, $\tilde{Y}^-_{\tcut}(v) \leq Y^-_{\tcut}(\phi_k(v))$.
We know that $\tilde{Y}^+_{\tcut}(v)= \tilde{Y}^-_{\tcut}(v)$ by the definition of the support and $\cE_k$ and so combining the above results yields
\[
Y^+_{\tcut}(\phi_k(w)) = \tilde{Y}^+_{\tcut}(v)= \tilde{Y}^-_{\tcut}(v)\leq Y^-_{\tcut}(\phi_k(v)) \leq Y^+_{\tcut}(\phi_k(v))\,,
\]
so $Y^+_{\tcut}(\phi_k(v))= Y^-_{\tcut}(\phi_k(v))$.  This verifies~\eqref{e:BSDiesOut}, completing the proof.
\end{proof}

\subsection{Red percolation clusters given the history of their exterior}\label{s:RedBound}

This section is devoted to the proof of the following \emph{upper bound} on the probability of a cluster to be \red\ given the update sequence $\cU$ along the $(\tcut,\tcut+\scut]$ and the history up to time $\tcut$ of every vertex in its exterior:

For any cluster of components $\cC$ and each $\ell \geq 1$ we define the relation
\[ A_i \sim_\ell A_j\quad\mbox{ iff } \quad B\Big(A_i,\scut^2 \big(2^{\ell-2}+\tfrac14\big)\Big)\cap B\Big(A_j,\scut^2 \big(2^{\ell-2}+\tfrac14\big)\Big) \neq \emptyset\,,\]
and extend the relation to an equivalence relation.  Let $V_\ell$ denote the set of equivalence classes given by the equivalence relation and for each $v\in V_\ell$ let $A_v$ denote the union of the components in $v$.  We define $\mathfrak{L}$ to be the largest $\ell$ such that $|V_\ell|>1$.  We let $V_0=\cC$ be the set of $A_i$.

\begin{claim}\label{clm:animalConnection}
For any cluster of components $\cC$,
\[
\sum_{\ell=0}^{\mathfrak{L}} 2^{\ell} |V_\ell| \geq  \anim(\cup_{A_i\in\cC} A_i) - \sum_{A_i\in\cC}\anim(A_i)\,.
\]
\end{claim}

\begin{proof}
Let $\hat{A}_i$ denote a minimal lattice animal containing $A_i$ so that $A_i\subset \hat{A}_i$ and $|\hat{A}_i|= \anim(A_i)$.  We construct a lattice animal covering $\cup_i A_i$ by adding blocks to $\cup_i \hat{A}_i$ as follows.  Starting with $\ell=1$ we add the minimum number of blocks needed so that for all $v\in V_\ell$ all the $A_i \in v$ are connected together.  By definition, if $A_i\sim_\ell A_j$ then these can be connected with at most $2^{\ell-1}$ blocks. Thus, after connecting together all sets of components at level $\ell-1$ for each $v\in V_\ell$ we need to add at most $2^{\ell-1}\left(|\{v'\in V_{\ell-1}: v' \subset v\}|-1\right)$ additional blocks to connect together all the components of $V_\ell$.  Summing over $\ell$ from 1 to $\mathfrak{L} +1$ we add a total number of blocks of
\begin{align*}
\sum_{\ell=1}^{\mathfrak{L}+1} \sum_{v\in V_\ell} 2^{\ell-1}\left(|\{v'\in V_{\ell-1}: v' \subset v\}|-1\right) &= \sum_{\ell=1}^{\mathfrak{L}+1} 2^{\ell-1}(|V_{\ell-1}| - |V_\ell|)
\leq \sum_{\ell=0}^{\mathfrak{L}} 2^{\ell} |V_\ell|\,.
\end{align*}
Since adding $\sum_{\ell=0}^{\mathfrak{L}} 2^{\ell} |V_\ell|$ blocks to the $\hat{A}_i$ yields a connected component, the desired result follows.
\end{proof}

\begin{lemma}\label{l:redConnectionNoGreen}
There exists $c(\beta,d),s_0(\beta,d)>0$ such that, for any $\scut>s_0$, any large enough $n$ and every $\cC\subset \cA$,  the quantity $\Psi_\cC$ from~\eqref{eq-def-psi} satisfies
\begin{align}\label{e:redConnectionNoGreen}
\P\left(\cC\in \red\mid \sH^-_\cC,\,\cU\right) \leq  \frac{\scut^{4d}}{\sm_{\tcut}} e^{3\sum_i |A_i|-c \scut \left|\anim\left(\cC\right) - \sum_{A_i\in\cC}\anim(A_i)\right|^+ }\,.
\end{align}
\end{lemma}
\begin{proof}
The bound is trivial for histories which are not $\cC$-compatible so we may restrict our attention to the supremum over $\cC$-compatible histories.
Denote the event $\cE$ that the history of $\cC$ does not intersect the history $\sH^{-}_\cC$.
The set of clusters $\cA$ depends only on $\cU$ and this is the only dependence on $\cU$ in the bound.  Given $\cA$, the partition into clusters and their colors depends only on the updates in $[0,\tcut]$.  Hence, we can view $\red$ as a function $\red(\cA)$.  We can extend this definition to any set of components and write $\red(\cA')$ to denote the set of red clusters had the set of components instead been $\cA'$.  Now if $\cC\in\red(\cA)$ then one also has $\cC\in\red(\cC)$.

We let $\cE(H)$ denote the event that the history of $\cC$ does not intersect $H\subset \Lambda\times[0,\tcut]$, a space-time slab.  If $\cC$ is a cluster then $\cE(\sH^{-}_\cC)$ must hold.  Exploring the history of $\cC$ we see that it does not depend on the history of its complement, $\sH^{-}_\cC$, until the point at which they intersect (since they depend on disjoint updates) and hence
\[
\P(\sH_{\cC}\in \cdot,\, \cE(H)\mid \sH^{-}_\cC=H,\,\cU)  = \P(\sH_{\cC}\in\cdot,\, \cE(H)\mid\cU) \leq   \P(\sH_{\cC}\in\cdot)\,.
\]
Since the event $\cC\in\red(\cC)$ is $\sH_{\cC}$-measurable we have that,
\begin{align*}
\P(\cC\in\red\mid \sH^{-}_\cC,\,\cU) &\leq \P( \cC \in \red(\cC),\, \cE(\sH^{-}_\cC) \mid \sH^{-}_\cC,\,\cU)  
\leq \P(\cC \in \red(\cC))\,.
\end{align*}

Next, we bound $\P(\cC \in \red(\cC))$.
At least one vertex of $\cC$ must have support at time 0 so by a union,
\begin{align*}
\P(\cC \in \red(\cC)) =\P(\fsup(\cup_i A_i,0,\tcut)) \leq \scut^{2d} \sm_{\tcut} \sum_{i} |A_i| = \scut^{2d} \frac1{\sqrt{|\Lambda|}} \sum_{i} |A_i| \,;
\end{align*}
this implies~\eqref{e:redConnection} in the case $\anim(\cup_{i} A_i) - \sum_{i}\anim(A_i)\leq 0$, which we note includes the case $|\cC|=1$.
We may thus restrict our attention to the case $\anim(\cup_{i} A_i) - \sum_{i}\anim(A_i)>0$, in which $|\cC| \geq 2$ and so $|V_{\mathfrak{L}}|\geq 2$.

Our approach will be to define a collection of events that must hold if $\cC\in\red(\cC)$, one that depends on the structure of $\cC$ and its above defined decomposition.
For $\ell \leq \mathfrak{L}$ and $v\in V_\ell$ define the event $K_{v,\ell}$, which roughly says that the update set of $A_v$  spreads unexpectedly quickly, as
\[
K_{v,\ell}= \left\{\exists u\in B(A_v, \tfrac{3}{7} 2^{\ell} \scut^2)\setminus B(A_v, \tfrac{2}{5} 2^{\ell} \scut^2): \fupd(u ,\tcut- 2^{\ell}\scut,\tcut)\not\subset B( u,\tfrac{2^\ell}{100} \scut^2)\right\}\,,
\]
and the event
\begin{align*}
J_{v, \ell} = \bigg\{\fsup\left( B\big(A_v, \left(\tfrac{2}{5} +\tfrac1{100}\right) 2^{\ell} \scut^2\big) ,\tcut- 2^{\ell}\scut,\tcut-2^{\ell-1}\scut\right)\neq\emptyset,\, K_{v,\ell}^c\bigg\}\,,
\end{align*}
which roughly says that the support of $A_v$ lasts for a large time.  Combined, we define $R_{v,\ell}=K_{v,\ell}\cup J_{v, \ell}$.  For the final level $\ell=\mathfrak{L}$ we define a slight variant of these in terms of $\kappa$, some large positive constant to be fixed later,
\[
\hat{K}_{v,\ell}= \left\{\exists u\in B(A_v, \tfrac{3}{7} 2^{\ell} \scut^2)\setminus B(A_v, \tfrac{2}{5} 2^{\ell} \scut^2): \fupd(u ,\tcut- \kappa  2^{\ell}\scut,\tcut)\not\subset B( u,\tfrac{2^\ell}{100} \scut^2)\right\}\,,
\]
as well as
\begin{align*}
\hat{J}_{v, \ell} = \bigg\{\fsup\left( B\big(A_v, \left(\tfrac{2}{5} +\tfrac1{100}\right) 2^{\ell} \scut^2\big) ,\tcut- \kappa 2^{\ell}\scut,\tcut-2^{\ell-1}\scut\right)\neq\emptyset,\, \hat{K}_{v,\ell}^c\bigg\}
\end{align*}
and $\hat{R}_{v,\ell}=\hat{K}_{v,\ell}\cup \hat{J}_{v, \ell}$.
Finally,  we need to consider events describing how the history connects to time 0.  For $v\in V_{\mathfrak{L}}$, denote
\begin{align*}
\breve{J}_{v, \ell} = \bigg\{\fsup\left( B\big(A_v, \left(\tfrac{2}{5} +\tfrac1{100}\right) 2^{\ell} \scut^2\big) ,0,\tcut-2^{\ell-1}\scut\right)\neq\emptyset,\, \hat{K}_{v,\ell}^c\bigg\}\,.
\end{align*}
Define the set
\[
\Gamma=\fupd\left(B(\cup_{A_i\in\cC}A_i, 2^{\mathfrak{L}} \scut^2) ,\tcut- \kappa  2^{\mathfrak{L}}\scut,\tcut\right)
\]
and the event
\[
J_{\Gamma}=\bigg\{\fsup\left(\Gamma,0,\tcut- \kappa  2^{\mathfrak{L}}\scut\right)\neq\emptyset \bigg\}\,.
\]
\begin{claim}
If $\cC \in \red(\cC)$ then the events $R_{v,\ell}$ and $\hat{R}_{v,\ell}$ hold for all $0\leq\ell\leq \mathfrak{L}$ and $v\in V_\ell$. Furthermore, either $\cup_{v\in V_{\mathfrak{L}}} \breve{J}_{v, \mathfrak{L}}$ or $J_{\Gamma}\cap (\cup_{v\in V_{\mathfrak{L}}} \hat{K}_{v,\mathfrak{L}})$ hold.
\end{claim}
\begin{proof}[Proof of claim]
If $\cC \in \red(\cC)$ then the history of $A_v$ must connect to the remainder of the component.  This can occur either by the support of $A_v$ and $\cC\setminus A_v$ meeting in the interval $[\tcut-2^{\ell}\scut,\tcut]$ in which case $K_{v,\ell}$ holds.  Alternatively, the support of $A_i\in A_v$ may enter $B(A_j,\scut^2/3)$ for some $A_j\in\cC\setminus A_v$ in the interval $[\tcut-\scut,\tcut]$, in which case again $K_{v,\ell}$ holds.  Another option is that the support of $A_j\in\cC\setminus A_v$ enters $B(A_i,\scut^2/3)$ for some $A_i\in A_v$ in the interval $[\tcut-\scut,\tcut]$, whence again $K_{v,\ell}$ holds.

The final possibility is that the support survives until  time $\tcut-2^{\ell}\scut$ (i.e., $\fsup( A_v, ,\tcut- 2^{\ell}\scut,\tcut)\neq\emptyset$). In this case one of the following must hold:
\begin{itemize}[$\bullet$]
\item The support travels far in space by time $\tcut- 2^{\ell-1}\scut$ and
\[\fsup\left( A_v,\tcut- 2^{\ell-1}\scut,\tcut\right)\not\subset B\big(A_v, \left(\tfrac{2}{5} +\tfrac1{100}\right) 2^{\ell} \scut^2\big)\,,\]
in which case $K_{v,\ell}$ holds.
\item The support survives until time $\tcut- 2^{\ell}\scut$, given it did not travel too far by time $\tcut- 2^{\ell-1}\scut$, and
    \[ \fsup\left( B\big(A_v, \big(\tfrac{2}{5} +\tfrac1{100}\big) 2^{\ell} \scut^2\big) ,\tcut- 2^{\ell}\scut,\tcut-2^{\ell-1}\scut\right)\neq\emptyset\,,\] in which case $R_{v,\ell} $ holds.
\end{itemize}
This set of possibilities is exhaustive and completes the claim.  The other claims follow similarly.
\end{proof}
We now observe that the events $R_{v, \ell}$ are independent as they depend on disjoint sets of updates.  The event $K_{v,\ell}$ depends only on updates in the space time block
\[
\bigg(B(A_v, \Big(\tfrac{3}{7} +\tfrac1{100} \Big) 2^{\ell} \scut^2)\setminus B(A_v, \Big(\tfrac{2}{5}-\tfrac1{100} \Big) 2^{\ell} \scut^2)\bigg)\times[\tcut-2^\ell \scut, \tcut]
\]
while $J_{v,\ell}$ depends only on the same set (through the event $K_{v,\ell}^c$ in its definition) plus updates in
\[
\left(B\big(A_v,\left(\tfrac{2}{5} +\tfrac2{100}\right) 2^{\ell} \scut^2\big)\right) \times[\tcut- 2^{\ell}\scut,\tcut-2^{\ell-1}\scut]\,.
\]
Since these sets are disjoint for different $v$ it follows that the $R_{v,\ell}$'s are independent.  Similarly, the $\hat{R}_{v,\ell}$'s are independent for $\ell=\mathfrak{L}$ and independent of the $R_{v,\ell}$'s for $\ell < \mathfrak{L}$.  The event $\breve{J}_{v, \ell}$ is also independent of the $R_{v,\ell}$'s and $\hat{R}_{v,\ell}$'s.  The set $\Gamma$ depends only on updates in $(\Lambda\setminus B(A_v, 2^{\ell} \scut^2))\times (\tcut- \kappa  2^{\ell}\scut,\tcut)$ and hence independent of all our constructed events except $J_{\Gamma}$, and $J_{\Gamma}$ is independent of all our constructed events except $\breve{J}_{v, \ell}$.

We now estimate the probability of the above events using Lemma~\ref{l:BasicEstimates}.  Noting that
\[
\left| B\big(A_v, \alpha 2^{\ell} \scut^2\big)\right|\leq (1+2\alpha)^d 2^{\ell d} \scut^{2d} |A_v| \leq e^{C(\ell + \log \scut) + 2^{-\ell} |A_v|}
\]
for some  $c'>0$, we have that
\begin{align*}
\P(K_{v,\ell} ) \leq \left| B\big(A_v, \tfrac{3}{7} 2^{\ell} \scut^2\big)\right|  e^{-c 2^{\ell} \scut^2 }\leq e^{-c' 2^{\ell} \scut^2  + 2^{-\ell} |A_v|}
\end{align*}
for large enough $\scut$.  Similarly,
\begin{align*}
\P(\hat{K}_{v,\ell} )\leq e^{-c' 2^{\ell} \scut^2  + 2^{-\ell} |A_v|}\,.
\end{align*}
Again by Lemma~\ref{l:BasicEstimates} it follows that
\begin{align*}
\P(J_{v,\ell} ) \leq \left| B\big(A_v, \left(\tfrac{2}{5} +\tfrac1{100}\right)2^{\ell} \scut^2\big)\right|  e^{-c 2^{\ell} \scut}\leq e^{-c' 2^{\ell} \scut + 2^{-\ell} |A_v|}\,,
\end{align*}
and similarly
\begin{align*}
\P(\hat{J}_{v,\ell} ) \leq \left| B\big(A_v, \left(\tfrac{2}{5} +\tfrac1{100}\right)2^{\ell} \scut^2\big)\right|  e^{-c \kappa 2^{\ell} \scut}\leq e^{-c' \kappa 2^{\ell} \scut + 2^{-\ell} |A_v|}\,.
\end{align*}
As $\P(\fsup(u,0,t)\neq\emptyset)=\sm_t$ and $\sm_{t+h} \geq e^{-h}\sm_t$ we have that
\begin{align*}
\P(\breve{J}_{v,\mathfrak{L}} ) &\leq \left| B\big(A_v, \left(\tfrac{2}{5} +\tfrac1{100}\right)2^{\mathfrak{L}} \scut^2\big)\right|  \sm_{\tcut-2^{\mathfrak{L}-1}\scut}
\leq \exp(2^{\mathfrak{L}}\scut + 2^{-\ell} |A_v|)\sm_t\,.
\end{align*}
For $h\geq 1/3$ and $\scut$ large enough we have that
\[
\P\left(\fupd(v,\tcut-\kappa 2^\ell \scut,\tcut)\not\subset B\big(v,h 2^{\ell} \scut^2\big)\right) \leq e^{-\scut 2^{\ell} h}\,,
\]
and so $\E |\fupd(v,\tcut -\kappa 2^\ell \scut,\tcut)| \leq \scut^{2d} 2^{\ell d}$.  Hence,
\begin{align*}
\E|\Gamma| &\leq \left|B\big(\cup_{A_i\in\cC}A_i, 2^{\mathfrak{L}} \scut^2\big)\right| \scut^{2d} 2^{\ell d} \leq \left(3 \cdot 2^{2 \mathfrak{L}}\scut^4\right)^d \sum_i |A_i|\,.
\end{align*}
Finally, by a union bound over $u\in \Gamma$,
\[
\P\left(J_{\Gamma}\mid \Gamma\right)\leq  \sm_{\tcut-\kappa 2^\mathfrak{L} \scut} |\Gamma| \leq  \sm_{\tcut} e^{\kappa 2^\mathfrak{L} \scut} |\Gamma| \,.
\]
Combining the above estimates with the claim and the independence of the events we have that
\begin{align*}
\P\left(\cC \in \red(\cC)\right)& \leq \P\Bigg(\bigg(\bigcap_{\ell=0}^{\mathfrak{L}-1}\bigcap_{v\in V_\ell} R_{v,\ell}  \bigg)\cap \bigg( \bigcup_{v\in V_{\mathfrak{L}}} \breve{J}_{v, \mathfrak{L}}  \!\!\bigcap_{u\in V_{\mathfrak{L}}\setminus \{v\}} \!\!\hat{R}_{u,\mathfrak{L}} \bigg) \Bigg)\\
&+ \P\Bigg(\bigg(\bigcap_{\ell=0}^{\mathfrak{L}-1}\bigcap_{v\in V_\ell} R_{v,\ell}\bigg)\cap \bigg( \bigcup_{v\in V_{\mathfrak{L}}} \hat{K}_{v, \mathfrak{L}} \!\! \bigcap_{u\in V_{\mathfrak{L}}\setminus \{v\}} \!\!\hat{R}_{u,\mathfrak{L}} \bigg)\cap J_{\Gamma} \Bigg)\,.
\end{align*}
The right-hand side, in turn, is at most
\begin{align*}
& \exp\left( -c' \scut \sum_{\ell=0}^{\mathfrak{L}-1} 2^{\ell} |V_\ell| - c'\kappa 2^{\mathfrak{L}}(|V_{\mathfrak{L}}|-1) +2 \sum_{i} |A_i|\right)
|V_{\mathfrak{L}}| \sm_{\tcut} \Big(e^{2^{\mathfrak{L}}\scut}+ e^{-c' 2^{\ell} \scut^2 }\E \Gamma e^{\kappa 2^\ell \scut}  \Big)\\
& \leq \frac1{\sqrt{|\Lambda|}}\exp\left( -c' \scut \sum_{\ell=0}^{\mathfrak{L}} 2^{\ell} |V_\ell| + 3 \sum_{i} |A_i|\right)\,,
\end{align*}
where the  final inequality comes from the fact that $|V_{\mathfrak{L}}|\geq 2$ and
\[
|V_{\mathfrak{L}}|  \Big(e^{2^{2\mathfrak{L}}\scut}+ e^{-c' 2^{\ell} \scut^2 } (3 \cdot 2^{\mathfrak{L}}\scut^4)^d   e^{\kappa 2^\ell \scut}\sum_i |A_i|   \Big)  e^{- c'(\kappa-1)2^{\mathfrak{L}} - \sum_{i} |A_i|}\leq 1\,,
\]
provided that $c'(\kappa-1)>1$ and $\scut$ is large.  Combining this with Claim~\ref{clm:animalConnection} gives~\eqref{e:redConnectionNoGreen}, as required.
\end{proof}

\subsection{Proof of Lemma~\ref{l:redConnection}}
The desired bound in~\eqref{e:redConnection} is similar to the one we obtained in~\eqref{e:redConnectionNoGreen}, with the difference being an extra conditioning on $\cG_{\cC}:=\{\cC\in\red\}\cup\{\cC\subset\blue\}$. Since $\{\cC\in\red\} \subset \cG_{\cC}$,
\begin{align*}
\P\left(\cC\in \red\mid \sH^-_\cC,\,\cG_\cC\,,\cU\right) &= \frac{\P\left(\cC\in \red\mid \sH^-_\cC, \,\cU\right)}{\P\left(\cG_\cC \mid \sH^-_\cC,\,\cU\right)} \leq
\frac{\P\left(\cC\in \red \mid \sH^-_\cC, \,\cU\right)}
{\P(\bigcap_{A_i \in \cC}   \{A_i\in \blue\}  \mid \sH^{-}_\cC,\,\cU)}\,,
\end{align*}
from which Eq.~\eqref{e:redConnection} follows after applying Lemmas~\ref{l:BS} and~\ref{l:redConnectionNoGreen} to the numerator and denominator
of the right-hand side, respectively. \qed

\section{Random vs.\ deterministic initial states}
\label{sec:ann-que}
In this section we consider the Ising model on the cycle $\Z/n\Z$ for small $\beta>0$ with a random initial state, both quenched and annealed.  Rather than comparing two worst case initial states we will compare a random one directly with the stationary distribution using coupling from the past.  Recall that for the \oned Ising model we can give a special update rule:  with probability $\theta=\theta_{\beta,d}=1 - \tanh(2\beta)$ update to a  uniformly random value and with probability $1-\theta$ copy the spin of a random neighbor.  The history of a vertex is simply a continuous time random walk which takes steps at rate $1-\theta$ and dies at rate $\theta$;  when such walks meet they coalesce and continue together until dying.  Each component can only decrease in size over time and all vertices in the component receive the same spin.

In the sequel, the annealed setting (Part~\ref{it-ann} of Theorem~\ref{mainthm-Z-ann-que}) is studied in \S\ref{sec:ann-upper} and \S\ref{sec:ann-lower} (upper and lower bounds on mixing, respectively), whereas \S\ref{sec:que} focuses on the quenched setting (Part~\ref{it-que} of Theorem~\ref{mainthm-Z-ann-que}).

\subsection{Annealed analysis: upper bound}\label{sec:ann-upper}

We define $Y_t$ to be the process coupled from the past, so the spins at time $\tcut$ are independent on each of the components with equal probability.  Now let $X_t$ be the process started from an i.i.d.\ configuration at time 0.
Since the magnetization is simply the probability that the walk has not yet died it follows that $\sm_t= e^{-\theta t}$.  We will consider the total variation distance at time $\tcut' = \frac1{4\theta}\log n + \frac1{\theta} \log\log n-\scut$ for some large constant $\scut>0$.
\begin{theorem}
With $\tpluss' = \frac1{4\theta}\log n + \frac1{\theta} \log\log n$ we have that
$
\|\P(X_{\tpluss'}\in\cdot) - \pi(\cdot) \|_\tv \to 0
$.
\end{theorem}

Set $\tcut' = \tpluss' - \scut$ where $\scut>0$ is a large constant. In order to couple the process with the stationary distribution we consider updates in the range $t\in(-\infty,\tpluss']$ with the block components constructed using the updates in the range $[\tcut',\tpluss']$ with deferred and undeferred randomness similarly as before.

In this analysis it is necessary to directly compare the annealed distribution with the stationary distribution and for this we use the coupling from the past paradigm and hence consider updates before time 0.   We modify the \emph{Intersection} property of the construction of clusters to identify $A_i$ and $A_j$ if $\fsup(A_i,t,\tcut)\cap \fsup(A_j,t,\tcut)\neq\emptyset$  for some $-\infty < t<\tcut$ (in place of the condition $0 \leq t<\tcut$).

We also redefine the notion of a red cluster to be one containing two vertices whose history reaches time $0$ without coalescing, that is,
\begin{equation}
  \label{eq-red-def-ann}
  \cC\in\red \quad\mbox{ iff }\quad \Big|\bigcup_{v\in A_i\in\cC} \fsup(v,0,\tcut)\Big | \geq 2
\end{equation}
(note that the histories of vertices are always of size one). We define blue clusters as before and green clusters as the remaining clusters.
We can thus couple $X_t$ and $Y_t$ so that they agree on the blue and green components at time $\tcut$ but possibly not on the red components.

Recalling that $\anim(S)$ is the smallest lattice animal of blocks covering $S$, we let
\[
\anim_2(S):= \min_{S_1,S_2\,:\,S_1\cup S_2=S} \anim(S_1)+\anim(S_2)\,.
\]
We denote $V_\ell$, equivalence classes of components of $\cC$ as in \S\ref{s:RedBound} and this time define $\mathfrak{L}'$ to be the largest $\ell$ such that $|V_\ell|>2$.  The following claim is a simple extension of Claim~\ref{clm:animalConnection}.
\begin{claim}\label{clm:animalConnection2}
For any cluster of components $\cC$,
\[
\sum_{\ell=0}^{\mathfrak{L}'} 2^{\ell} |V_\ell| \geq  \anim_2(\cup_{A_i\in\cC} A_i) - \sum_{A_i\in\cC}\anim(A_i)\,.
\]
\end{claim}
The proof is essentially the same as Claim~\ref{clm:animalConnection}.
\begin{lemma}\label{l:redConnectionNoGreen2}
There exists $c(\beta,d),s_0(\beta,d)>0$ such that, for any $\scut>s_0$, any large enough $n$ and every $\cC\subset \cA$,  the quantity $\Psi_\cC$ from~\eqref{eq-def-psi} satisfies
\begin{align}\label{e:redConnectionNoGreen2}
\P\left(\cC\in \red\mid \sH^-_\cC,\,\cU\right) \leq  \frac{\scut^{4d}}{e^{2 \theta\tcut'}} e^{3\sum_i |A_i|-c \scut \left|\anim_2\left(\cC\right) - \sum_{A_i\in\cC}\anim(A_i)\right|^+ }\,.
\end{align}
\end{lemma}

\begin{proof}
The proof is similar to Lemma~\ref{l:redConnectionNoGreen} and we describe the necessary changes.  Throughout the proof we change all instances of $\mathfrak{L}$ to $\mathfrak{L}'$ and $\tcut$ to $\tcut'$.  We must incorporate the fact that two histories independently reaching time 0 are required for red so we denote
\begin{align*}
\breve{J}_{v, v', \ell}' = \bigg\{\bigg|\fsup\left( B\big(A_v\cup A_{v'}, \left(\tfrac{2}{5} +\tfrac1{100}\right) 2^{\ell} \scut^2\big) ,0,\tcut'-2^{\ell-1}\scut\right)\bigg| \geq 2, \hat{K}_{v,\ell}^c, \hat{K}_{v',\ell}^c\bigg\}\,.
\end{align*}
Since the probability that two separate histories reach time 0 without coalescing is bounded by the square of the probability of a single walk reaching time 0, we have that
\begin{align*}
\P(\breve{J}_{v_1,v_2,\mathfrak{L}}' )&\leq e^{-2\theta \tcut'} \prod_{i=1}^2 |B\big(A_{v_i}, \left(\tfrac{2}{5} +\tfrac1{100}\right) 2^{\ell} \scut^2\big)| \leq \exp\big(2^{\mathfrak{L}}\scut + 2^{-\ell} (|A_{v_1}|+|A_{v_{2}}|)\big) e^{-2\theta \tcut'}\,.
\end{align*}
By essentially the same proof as above we have that
$\E|\Gamma| \leq  \left(3 \cdot 2^{2 \mathfrak{L}}\scut^4\right)^{2d} (\sum_i |A_i|)^2$.
Also, as we require two histories to reach time 0, we let
\[
J_{\Gamma}'=\bigg\{\bigg|\fsup\left(\Gamma,0,\tcut' - \kappa  2^{\mathfrak{L}}\scut\right)\bigg| \geq 2 \bigg\}
\]
and
\[
\P\left(J_{\Gamma}\mid \Gamma\right)\leq  \sm_{\tcut-\kappa 2^\mathfrak{L} \scut}^2 |\Gamma|^2 \leq  e^{-2\theta \tcut'} e^{2\kappa 2^\mathfrak{L} \scut} |\Gamma|^2 \,.
\]
With this notation,
\begin{align*}
\P\left(\cC \in \red(\cC)\right)& \leq \P\Bigg(\bigg(\bigcap_{\ell=0}^{\mathfrak{L}'-1}\bigcap_{v\in V_\ell} R_{v,\ell}  \bigg)\cap \bigg( \bigcup_{v,v'\in V_{\mathfrak{L}'}} \breve{J}_{v,v', \mathfrak{L}'}'  \!\!\bigcap_{u\in V_{\mathfrak{L}'}\setminus \{v,v'\}} \!\!\hat{R}_{u,\mathfrak{L}'} \bigg) \Bigg)\\
&+ \P\Bigg(\bigg(\bigcap_{\ell=0}^{\mathfrak{L}'-1}\bigcap_{v\in V_\ell} R_{v,\ell}\bigg)\cap \bigg( \bigcup_{v\in V_{\mathfrak{L}'}} \hat{K}_{v, \mathfrak{L}'} \!\! \bigcap_{u\in V_{\mathfrak{L}'}\setminus \{v\}} \!\!\hat{R}_{u,\mathfrak{L}'} \bigg)\cap J_{\Gamma}' \Bigg)\,.
\end{align*}
The result now follows similarly to Lemma~\ref{l:redConnectionNoGreen} by substituting our bounds for each of the events.
\end{proof}
Since red components under our modified definition are also red components under the previous definition we have by Lemma~\ref{l:redConnectionNoGreen},
\begin{align*}
\P\left(\cC\in \red\mid \sH^-_\cC,\,\cU\right) \leq  \frac{\scut^{4d}}{e^{\theta\tcut'}} e^{3\sum_i |A_i|-c \scut \left|\anim\left(\cC\right) - \sum_{A_i\in\cC}\anim(A_i)\right|^+ }\,,
\end{align*}
and combining this with~\eqref{e:redConnectionNoGreen2} yields
\begin{align}\label{e:redWeakerAnnBound}
\P\left(\cC\in \red\mid \sH^-_\cC,\,\cU\right) \leq  \frac{\scut^{4d}}{e^{2\theta\tcut'}} e^{3\sum_i |A_i|-c \scut \left|\anim_2\left(\cC\right) - \sum_{A_i\in\cC}\anim(A_i)\right|^+ - c\scut \left|\anim\left(\cC\right) - \anim_2\left(\cC\right) - \frac{\theta}{c\scut}\tcut'\right|^+}\,.
\end{align}
Altogether, this translates into the bound
\[
\bar{\Psi}_{\{A_{i_j}\}} \leq  \frac{\scut^{4d}}{e^{2\theta\tcut'}} e^{4\sum_i |A_i|-c \scut \left|\anim_2\left(\cC\right) - \sum_{A_i\in\cC}\anim(A_i)\right|^+ - c\scut \left|\anim\left(\cC\right) - \anim_2\left(\cC\right) - \frac{\theta}{c\scut}\tcut'\right|^+}\,.
\]
Having coupled $X_t$ and $Y_t$ as described above where only the red components differ in the two versions of the chain, we can follow the analysis of \S\ref{sec:inf-perc} and in place of equation~\eqref{eq-sum-A-one-animal} arrive at
\begin{align}
&\frac{2^{2\lambda+6}\scut^{8d}}{e^{4\theta\tcut'}} \sum_{A} \Bigg( \sum_{\substack{\cC=\{A_{i_j}\}\\ A\in\cC}} \sum_{\{B_j\}}
 \exp\Bigg[ -\frac{c}4 \frac{\scut}\lambda \bigg[\anim(\cC) + \sum_j \anim(A_{i_j} \cup B_j) + \Big(\anim(\cC) - \anim_2(\cC)-\tcut'\Big)^+\bigg]\Bigg]\Bigg)^2\,.\label{eq-sum-A-one-animalAnnealed}
 \end{align}
We now count the number of rooted animals $\cC$ with $\anim(\cC) = \cS$ and $\anim_2(\cC) = \cS'$.  We must cover $\cC$ with two lattice animals, one containing the root $A$ and the second rooted at $A'$.  Since the distance from $A$ to $A'$ is at most $\cS$ there are at most $(2\cS+1)^d$ choices for $A'$. There are $\cS'$ ways to choose the sizes of the two animals as some $k$ and $\cS'-k$ and then $(2d)^{2k + 2(\cS'-k)}$ ways of choosing the animals.  In total, we have at most
$\cS' (2\cS+1)^d (2d)^{2\cS'}$ choices of $\cC$.  The total number of choices of $\{A_{i_j}\}$ and $B_j$ with $\anim(\cC)=\cS$, $\anim_2(\cC)=\cS'$ and $\sum_j \anim(A_{i_j} \cup B_j) = \cR$ is therefore $\cS' (2\cS+1)^d 2^{\cS'}8^{\cR}(2d)^{2(\cS'+\cR)}$; thus,
\begin{align*}
&\sum_{\substack{\cC=\{A_{i_j}\}\\ A\in\cC}} \sum_{\{B_j\}}
 \exp\Bigg[ -\frac{c}4 \frac{\scut}\lambda \bigg[\anim(\cC) + \sum_j \anim(A_{i_j} \cup B_j) + \Big(\anim(\cC) - \anim_2(\cC)-\tcut'\Big)^+\bigg]\Bigg] \\
 &\qquad\leq \sum_{\substack{\cS',\cR \geq 1\\ \cS \geq \cS'}} e^{-\frac{c}4 \frac{\scut}\lambda (\cS' + \cR+(\cS-\cS' - \tcut')^+)} \cS' (2\cS+1)^d 2^{\cS'}8^{\cR}(2d)^{2(\cS'+\cR)} \leq \tcut' e^{-\frac{c}5 \frac{\scut}\lambda}
\end{align*}
provided $\scut$ is large enough compared to $d$.  Plugging this into~\eqref{eq-sum-A-one-animalAnnealed} gives an upper bound on the total variation distance of
\[
\frac{2^{2\lambda+6}\scut^{8d}}{e^{4\theta\tcut'}}|\Lambda| (\tcut')^2 e^{-\frac{2c}5 \frac{\scut}\lambda}
\]
which tends to 0 for $\tcut' = \frac{1}{4\theta} \log n + \frac1{\theta}\log\log n - \scut$, establishing the upper bound on the mixing time.

\subsection{Annealed analysis: lower bound}\label{sec:ann-lower}
We now prove a matching lower bounded on the mixing time from an annealed initial configuration.
\begin{theorem}
For $t = \frac1{4\theta}\log n - \frac{2}{\theta}  \log\log n$ we have that
$
\|\P(X_{t}\in\cdot) - \pi(\cdot) \|_{\tv} \to 1\,.
$
\end{theorem}
It will be convenient to simply omit the deferred and undeferred regions and directly analyze the update support function from time $t=\frac1{4\theta}\log n - \frac{2}{\theta}  \log\log n$.
Taking the coupling above, two vertices $v$ and $v'$ have the same spins in $X$ if their histories merge before time $0$ and are conditionally independent otherwise. By contrast, for $Y$ the spins are equal if the histories merge at any time in the past and are conditionally independent otherwise.  Defining $\cA_{v,v'}$ as the event that the histories of $v$ and $v'$ survive to time 0 and merge at some negative time, we have that
\[
\P\left(Y_{t}(v)=Y_{t}(v')\right)-\P\left(X_{t}(v)=X_{t}(v')\right) = \frac12  \P\left(\cA_{v,v'}\right)\,.
\]
Now, as the history of $v$ and $v'$ are both continuous time random walks,
\begin{align*}
\P\left(\fsup(v,0,t)=\{v\},\,\fsup(v',0,t)=\{v'\}\right)  \geq c_1 \frac1{(\tcut')^2} e^{-2\theta t}\,,
\end{align*}
since the probability is that two random walks started from neighboring vertices do not intersect until time $\tcut'$ and return to their starting locations is at least $c_1/(\tcut')^2$ and the probability that neither walk dies is at least $e^{-2\theta t}$.  They then have a constant probability of merging for some $t_1<0$ so
\[
\P\left(\cA_{v,v'}\right) \geq c_2 \frac1{(\tcut')^2} e^{-2\theta t}\,,
\]
and if we label the vertices around the cycle as $u_1,\ldots,u_n$ then for large $n$,
\[
\E \bigg[\sum_{i=1}^{n-1} Y_{t}(u_i) Y_{t}(u_{i+1}) - X_{t}(u_i) X_{t}(u_{i+1})\bigg] \geq c_2 \frac1{(\tcut')^2} e^{-2\theta t} n \geq \sqrt{n} \log n\,.
\]
By the exponential decay of correlation in the stationary distribution of the \oned Ising model,
\[
\var\bigg(\sum_{i=1}^{n-1} Y_{t}(u_i) Y_{t}(u_{i+1})\bigg) \leq C n\,.
\]
Since the spins in different clusters are independent and with probability at least $1-n^{-10}$ there are no clusters whose diameter is greater than $C \log \log n$ for some large $C(\beta)$ we have that when $|i-i'|\geq C\log n$ then
\[
\Cov\bigg(X_{t}(u_i) X_{t}(u_{i+1}),X_{t}(u_{i'}) X_{t}(u_{i'+1})\bigg)=O(n^{-10})\,,
\]
and hence
\[
\var\bigg(\sum_{i=1}^{n-1} X_{t}(u_i) X_{t}(u_{i+1})\bigg) \leq C n \log n\,.
\]
It follows by Chebyshev's inequality that
\[
\P\left( \sum_{i=1}^{n-1} Y_{t}(u_i) Y_{t}(u_{i+1}) \geq \E \bigg[\sum_{i=1}^{n-1} Y_{t}(u_i) Y_{t}(u_{i+1})\bigg] - \frac12\sqrt{n} \log n\right) = 1-o(1)
\]
while
\[
\P\left( \sum_{i=1}^{n-1} X_{t}(u_i) X_{t}(u_{i+1}) < \E \bigg[\sum_{i=1}^{n-1} Y_{t}(u_i) Y_{t}(u_{i+1})\bigg] - \frac12\sqrt{n} \log n\right) = 1-o(1)\,,
\]
and hence $\|X_{t} - \pi \|_\tv \to 1$.

\subsection{Quenched analysis}\label{sec:que}

Here we show that on the cycle, there is at most a minor, $O(\log \log n)$ improvement when starting from a typical random initial configuration, since almost all configurations bias the magnetizations of most vertices.  For a configuration $x_0$ denote
\begin{equation}
  \label{eq-Rt(x0)}
  R_t(u,x_0)=\sum_{u' \in V} P_t(u,u') x_0(u')\,,
\end{equation}
where $P_t(u,u')$ is the transition probability of a continuous time walk with jumps at rate $(1-\theta)$. Observe that
\begin{equation}
  \label{eq-E[Xt]-Rt-relation}
  \E_{x_0}[X_t(u)] = e^{-\theta t} R_t(u,x_0)\,.
\end{equation}
\begin{proposition}\label{prop:quenched}
Suppose that there exists some sequence $a_n = n^{o(1)}$ such that for any large $n$,
\begin{equation}\label{e:SignedMagnetization}
\frac1n \sum_{u\in \Lambda} |R_{t}(u,x_0)| \geq \frac1{a_n}
\qquad\mbox{ at }\qquad t= \frac1{2\theta}\log n - \frac{1}{\theta}\log \log n - \frac1{\theta}\log a_n\,.
\end{equation}
Then
$
\|\P_{x_0}(X_{t} \in \cdot) - \pi(\cdot) \|_{\tv} = 1 - o(1)
$
as $n\to\infty$.
Furthermore, if $x_0$ is uniformly chosen over $\{\pm 1\}^\Lambda$ then there exists some $C=C(\beta)>0$ such that
\eqref{e:SignedMagnetization} holds with probability $1-o(1)$ for $a_n = C \log n$.
\end{proposition}
\begin{proof}
Let $X_0$ be a uniformly chosen initial configuration.  Since $R_t(u,X_0)$ is a sum of independent increments, when $t$ and $n$ are both large, $R_t(u,X_0)$ is approximately $\cN(0, \sum_{u' \in V} P_t(u,u')^2)$ by the Central Limit Theorem; in particular, $\E|R_t(u,X_0)| \geq c/t$ for some fixed $c>0$, and
\[
\E\bigg[\frac1n \sum_{u\in\Lambda} |R_{t}(u,X_0)|\bigg] \geq \frac{c }{\log n}
\]
for another fixed $c>0$ provided $t$ has order $\log n$. From the decay of $P_t(u,u')$ we infer that if $|u-u'|\geq C\log n$ for a large enough $C>0$ then $\Cov(|R_t(u)|,|R_t(u')|) \leq n^{-10}$, thus implying that
\[
\var\bigg( \frac1n\sum_{u\in\Lambda} |R_{t}(u,X_0)| \bigg) = O\bigg( \frac{\log n}n \bigg)\,,
\]
and altogether
\[
\P\bigg( \frac1n \sum_{u\in\Lambda} |R_{t}(u,X_0)| >  \frac{c}{2\log n} \bigg) \to 1\,.
\]
This establishes~\eqref{e:SignedMagnetization} with probability going to 1 for $a_n = C \log n$ with a suitably chosen $C=C(\theta)>0$.

By the same decay of correlations, for vertices with $|u-u'|\geq C\log n$ for some large enough $C>0$ the histories do not merge in the interval $(0,t]$ with probability $1-O(n^{-10})$. In this case,
\[
\Cov_{x_0}\Big(\sign(R_t(u,x_0))X_t(u) \,, \sign(R_t(u',x_0))X_t(u')\Big) = O(n^{-10})\,,
\]
and again we can deduce that
\[ \var\bigg(\frac1n \sum_{u\in\Lambda} \sign(R_t(u,x_0) X_t(u)\bigg) = O\bigg(\frac{\log n}n\bigg)\,.\]
 Recalling from~\eqref{eq-E[Xt]-Rt-relation} that $\E_{x_0} [\sign(R_t(u,x_0)) X_t(u)] = e^{-\theta t} |R_t(u,x_0)|$, Chebyshev's inequality yields
\[
\P_{x_0}\bigg( \frac1n\sum_{u\in\Lambda}\big( \sign(R_{t}(u,x_0)) X_{t}(u) - |R_{t}(u,x_0)|\big) > \frac{1}{2a_n} \bigg) \leq O\bigg(\frac{\log n}n  \big(e^{\theta t} a_n\big)^2 \bigg) = O\bigg(\frac{1}{\log n}\bigg)
\]
by the definition of $t$. Thus, we conclude that for any $x_0$ satisfying~\eqref{e:SignedMagnetization},
\begin{equation}\label{e:quenchedTestA}
\P_{x_0}\bigg( \frac1n \sum_{u\in\Lambda} \sign(R_{t}(u,x_0)) X_{t}(u)  > \frac1{2a_n} \bigg) \to 1\,.
\end{equation}
By contrast, if $Y$ is chosen independently according to the Ising measure then by the exponential decay of spatial correlations we have that
\[
\var \bigg(\frac1n \sum_{u\in\Lambda} \sign(R_{t}(u,x_0)) Y(u) \bigg) = O(1/n)\,,
\]
and since $\E \sum_{u\in\Lambda} \sign(R_{t}(u)) Y(u)=0$ while $1/a_n \gg 1/\sqrt{n}$ we can infer that
\begin{equation}\label{e:quenchedTestB}
\P\bigg( \frac1n \sum_{u\in\Lambda} \sign(R_{t}(u,x_0)) Y(u)  > \frac1{2 a_n} \bigg) \to 0\,.
\end{equation}
Comparing equations \eqref{e:quenchedTestA} and \eqref{e:quenchedTestB} implies that
\[
\|\P_{x_0}(X_{t} \in \cdot) - \pi(\cdot) \|_{\tv} = 1 - o(1)\,,
\]
completing the result.
\end{proof}

\subsection*{Acknowledgements}  The authors thank the Theory Group of Microsoft Research, Redmond, where much of the research has been carried out, as well as P.\ Diaconis and Y.\ Peres for insightful discussions and two anonymous referees for thorough reviews.

\end{document}